%% file: main.tex
\documentclass{amsart}

\input{preamble}

\title[Murphy's Law for 2-Nilpotent Groups]{Mekler's Construction and Murphy's Law for 2-Nilpotent Groups}

\author{Blaise Boissonneau}
\address{Blaise Boissonneau \\ Heinrich Heine University Düsseldorf, Faculty of Mathematics and Natural Sciences, Universitätsstr.~1, 40225 Düsseldorf, Germany}
\email{blaise.boissonneau@hhu.de}
\author{Aris Papadopoulos}
\address{Aris Papadopoulos\\ Department of Mathematics, University of Maryland, College Park, MD 20742-4015, USA}
\email{aris@umd.edu}
\thanks{A.P. was partially supported through a Leeds Doctoral Scholarship from the University of Leeds.}
\author{Pierre Touchard}
\address{Pierre Touchard \\ Institut f\"ur Algebra, Technische Universit\"at Dresden, 01062 Dresden, Germany and KU Leuven, Department of Mathematics, B-3001 Leuven, Belgium}
\email{pierre.touchard@tu-dresden.de}
\thanks{P.T. was partially supported by KU Leuven through IF C16/23/010.}

\date{\today}

\begin{document}

\begin{abstract}
    Mekler's construction is a powerful technique for building purely algebraic structures from combinatorial ones. Its power lies in the fact that it allows various model-theoretic tameness properties of the combinatorial structure to transfer to the algebraic one. In this paper, we push this ideology much further, describing a broad class of properties that transfer through Mekler's construction. This technique subsumes many well-known results and opens avenues for many more.

    As a straightforward application of our methods, we (1) obtain transfer principles for stably embedded pairs of Mekler groups and (2) construct strictly $\mathsf{NFOP}_k$ pure groups for all $k\in\Nbb_{>2}$. We also answer a question of Chernikov and Hempel on transfer of burden.
\end{abstract}

\maketitle
\setlength{\parskip}{4pt}

\section{Introduction}

\underline{Tame classes of groups}. Model theorists have a good understanding of the definable structure of modules and, in particular, of abelian groups. It is well known that any module is \emph{stable}, i.e. it does not ``code'' a linear order. In broader terms, this means that definable sets in pure abelian groups are ``combinatorially tame''. It could be reasonable to think that other natural classes of groups are also ``tame''. However, as shown by Mekler \cite{Mek81}, already the class $\mathbb{G}_{2,p}$ of $2$-nilpotent groups of prime exponent $p$ -- certainly the easiest class of non-abelian groups that one can think of -- can witness all sorts of model-theoretic behaviours; this was extended by many subsequent works \cite{Bau02,CH18,Ahn20}, many of which are directly related to Shelah's classification programme.

\underline{Classification Theory}. An essential aspect of Shelah's classification theory is, very roughly speaking, based on the idea that the presence/absence of simple combinatorial data (e.g. linear orders, random graphs, and more) can give us a lot of information on the ``complexity'' of a given theory (e.g. stability, NIP, and more). 

The dichotomies given by the presence/absence of combinatorial configurations allow model theorists to divide (hence the name \emph{``dividing lines''}) the class of all first-order theories into smaller regions, which can then be studied in more detail. Since the various regions of the model-theoretic universe are defined using combinatorial data, it is often the case that building combinatorial examples inhabiting each region is a much easier task than building \emph{purely} algebraic ones. 

Mekler's construction is now a classical technique for building a purely algebraic structure (a 2-nilpotent group) from a purely combinatorial one (a ``nice'' graph). We will review the basic ideas of Mekler's construction in \Cref{subsec:Mekler}, but for now, the reader should keep in mind that this construction is done in a way that has been shown to preserve the presence (and absence) of various kinds of combinatorial data, such as, for instance, $\lambda$-stability \cite{Mek81}, non-independence property (of any arity $n\in\Nbb$), the tree property of the second kind \cite{CH18}. See \Cref{thm:KnownTransfers} for an exhaustive (to our knowledge) list of properties preserved by Mekler's construction. 

Hodges somewhat ironically referred to this fact as \href{https://en.wikipedia.org/wiki/Murphy%27s_law}{\emph{Murphy's law}} 
for 2-nilpotent groups. Far from being a negative fact, Murphy's law for 2-nilpotent groups justifies the rather reassuring hope that even the most subtle combinatorial property, if satisfied by an ad-hoc combinatorial example, is probably also satisfied by a pure $2$-nilpotent group.

\underline{Main results of the paper}. Our paper aims to push the ideas discussed in the preceding paragraphs further by developing a method that allows us to prove general transfer principles for groups obtained via Mekler's construction. In a precise sense, we show that Mekler's construction preserves the presence/absence of \emph{all} interesting combinatorial data characterised through generalised indiscernibles (all these concepts will be explained in detail later in the paper). Let us call a \emph{Mekler group} any group elementarily equivalent to a group obtained by Mekler's construction. The first main result of this paper is the following:

\begin{thmx}
    Let $\Ical$ be a Ramsey structure, and $\Jcal$ a reduct of $\Ical$. Let $\Msf$ be a Mekler group and $\gG$ its associated nice graph. Then, the following are equivalent:
    \begin{enumerate}
        \item $\gG$ collapses $\Ical$-indiscernibles (resp. to $\Jcal$-indiscernibles).
        \item $\Msf(\gG)$ collapses $\Ical$-indiscernibles (resp. to $\Jcal$-indiscernibles).
    \end{enumerate}
\end{thmx}
Ramsey structures\footnote{For a precise definition, see e.g. \cite{Nes05}.} are suitable indexing structures, and in particular, indiscernible sequences indexed by a Ramsey structure satisfy desirable properties similar to (ordered)-indiscernible sequences. A structure $\sM$ collapses $\Ical$-indiscernibles if all $\Ical$-indiscernible sequences in $\sM$ are indiscernible in a strictly stronger way (see \cref{def:indiscernible collapsing}). 
In particular, this theorem shows that the class $\mathbb{G}_{2,p}$ does not lie in the tame side of any dividing lines characterised by the collapsing of generalised indiscernibles.

This result is proved in \Cref{subsec:Collapse}. As we have already mentioned, \Cref{thm:TransferCollapsing} generalises many previously known results, and we briefly explain this in \Cref{subsec:Collapse}. In this paper, a novel use of this theorem is the construction of examples of $\mathsf{NFOP}_k$ pure groups for all $k\in\Nbb_{>2}$, confirming an expectation of Abd Aldaim, Conant, and Terry \cite[Remark~2.13]{AACT23}.

We also take the occasion to give a negative answer to a problem of Chernikov and Hempel \cite[Problem~5.8]{CH18} in \Cref{subsubsec:burden}, where we observe that if $\gG$ is any infinite nice graph, then its Mekler group $\Msf(\gG)$ has arbitrarily large finite burden.

The method we developed also allows us to study pairs of Mekler groups and, in particular, to state a transfer principle for \emph{stably embedded} pairs of models. A stably embedded pair $(\Mcal, \Mcal')$ consists of an elementary extension $\Mcal \preccurlyeq \Mcal'$ where all types over $\Mcal$ realised in $\Mcal'$ are definable over $\Mcal$.

\begin{thmx}
    Let $\Msf\preccurlyeq \Msf'$ be an elementary extension of Mekler groups of respective graphs $\gG$ and $\gG'$. 
    \begin{enumerate}
        \item $\Msf$ is stably embedded in $\Msf'$ if and only if $\gG$ is stably embedded in $\gG'$. 
        \item $\Msf$ is uniformly stably embedded in $\Msf'$ if and only if $\gG$ is uniformly stably embedded in $\gG'$.
    \end{enumerate}
\end{thmx}

This theorem is proved in \Cref{subsec:SE}. To our knowledge, this is the first transfer principle for Mekler groups for a property that depends on the model itself and not just on its theory. This suggests that other model-theoretic properties of Mekler groups, depending on the models, can be reduced to the study of their graphs.

To prove \Cref{thm:TransferCollapsing} and \Cref{thm:RSEMekler}, we give a new \emph{relative quantifier elimination} result from the Mekler group $\Msf(\gG)$ down to its nice graph $\gG$ (modulo finite imaginaries), which we believe is of independent interest. This occupies all of \Cref{sec:RQE}, and is then used in \Cref{sec:transfers} to prove the main theorems. 

We also give an alternative and faster proof of \Cref{thm:TransferCollapsing} in \Cref{sec:alternative}, under the additional (but mild) hypothesis of \emph{specific collapsing} (see \Cref{def:indiscernible collapsing}). The two proofs are very different in nature, and we consider them both to be of interest. Finally, in \Cref{sec:OpenQuestions}, we conclude with open questions about transfers for Mekler groups, which we think could be answered by an adaptation of our methods.

\underline{Related results.} The results and methods presented in \cite{EMRS25}, which build in part on the work of Mekler, represent a significant breakthrough in the model-theoretic study of nilpotent groups of finite exponent. The authors successfully employ the Lazard correspondence as a model-theoretic bridge between the theory of Lie algebras and the study of nilpotent groups of prime exponent. Among several classification results, they show that $\mathbf{G}_{c,p}$, the Fra\"issé limit of the class of Lazard groups of exponent $p$ and nilpotency class $c$ equipped with predicates for the Lazard series, is a strictly NIP$_c$ pure group. Furthermore, they show that $\mathbf{G}_{2,p}$ is NFOP$_2$ (see \cite[§3.3]{EMRS25}). 
Our approach in the present paper is rather different, as we restrict to the class of $2$-nilpotent groups, and do not fix any particular complete theory of groups. However, the complementary points of view with \cite{EMRS25} may suggest further development: can the various results concerning $\mathbb{G}^L_{2,p}$ be encapsulated in a more general transfer principle? Or, can Mekler's construction be adapted to groups of higher nilpotency class? We do not address these questions in this paper, leaving them for potential (ambitious) future work.
\clearpage
\setlength{\parskip}{2pt}

\tableofcontents

\setlength{\parskip}{4pt}

\underline{Notation.} Throughout this paper, we assume familiarity with first-order logic and basic model theory (types, saturation, monster models). All of this material can be found in \cite[Chapters~1-5]{TZ12}. Our notation is either standard or explained. We denote by $\lL_{\mathsf{grp}}$ the language of groups with inverse $\{\cdot,^{-1},1\}$, often identified with $\{+,-,0\}$, in the context of abelian groups. Not all functions mentioned in the remainder of the paper will be defined everywhere; to remedy this, for notational convenience, languages include by default a constant symbol $\und$, standing for \emph{``undetermined''}. This will be used (in the obvious way) for functions whose actual domain is smaller than the base set.  

\section{Preliminaries}

To keep this paper self-contained, we will start our preliminaries by giving a brief exposition of Mekler's construction in \Cref{subsec:Mekler}. This construction originates from \cite{Mek81}. We will then briefly review some basic facts around generalised indiscernibles and Ramsey structures in \Cref{subsec:generalised-indiscernibles}. Specific examples of generalised indiscernibles (and how collapses thereof characterise model-theoretic dividing lines) will also be discussed in more detail in \Cref{sub:mt-div}.

\subsection{Mekler's construction}\label{subsec:Mekler}

A standard reference for Mekler's construction is \cite[Appendix~A.3]{Hod93}, and we will closely follow the notation and terminology given there. Several refinements (and some corrections) to statements of \cite{Hod93} were given by Chernikov and Hempel in \cite[Section~2]{CH18}, and we will include some of them here as well. This \namecref{subsec:Mekler} is purposefully laconic, and we refer the reader to both \cite{Hod93} and \cite{CH18} for more information (and, in particular, proofs).

\begin{definition}\label{def:nicegraph}
    A graph\footnote{A \emph{graph} is a structure in a relational language with a single binary relation $E$ which is assumed to be irreflexive and symmetric.} $\gG=(V,E)$ is called \emph{nice} if it satisfies the following properties:
    \begin{enumerate}
        \item\label{item:nice1} $|V|\geq 2$;
        \item\label{item:nice2} For any two distinct vertices $v_1,v_2\in V$ there is some vertex $v\in V\setminus\{v_1,v_2\}$ such that $\{v_1,v\}\in E$ and $\{v_2,v\}\not\in E$;
        \item\label{item:nice3} There are no triangles or squares in the graph.
    \end{enumerate}
\end{definition}
\begin{example}
Natural examples that we will consider are infinite, but it can be useful to think of finite ones. The simplest, and only, example with at most $5$ vertices is the pentagon, and can be used as a building block to build larger finite examples.
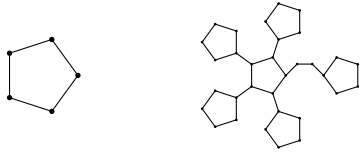
\begin{center}
    \begin{tikzpicture}
    \begin{scope}[shift={(-3,0)}]
        
  \draw \foreach \a in {1,...,5} {
      (\a*72:0.5cm)  -- (\a*72+72:0.5cm) 
    };
    \foreach \a in {1,...,5} {
    \fill (\a*72:0.5cm) circle(1pt);     
    }
    
    \end{scope}
     \begin{scope}[scale=0.5]    
     \draw \foreach \a in {1,...,5} {
      (-\a*72:0.5cm)  -- (-\a*72+72:0.5cm) 
    };
    \foreach \a in {1,...,5} {
    \fill (-\a*72:0.5cm) circle(1pt);     
    }

     \draw \foreach \a in {0,...,4} {
      (2,0)+(\a*72+36:0.5cm)--+(108+\a*72:0.5cm) 
    };
    \foreach \a in {0,...,4} {
    \fill (2,0)+(\a*72+36:0.5cm) circle(1pt);   
    }

  \draw (0.5,0)--(0.8,0.3)--(1.2,0.3)--(1.5,0);
   \fill (0.8,0.3) circle(1pt);     
   \fill (1.2,0.3) circle(1pt);
   \foreach \j in {1,...,4}{
    \foreach \i in {0,...,4}{
        \fill (\j*72:1.5cm)+(\i*72+36:0.5cm) circle(1pt);
        \draw (\j*72:1.5cm)+(\i*72+36:0.5cm) -- +(108+\i*72:0.5cm); };
        \draw (\j*72:1.5cm)+(-180+\j*72:0.5cm)-- (\j*72:0.5cm);
}
 \end{scope}  
\end{tikzpicture}
\end{center}

\end{example}

Conditions (\labelcref{item:nice1})-(\labelcref{item:nice3}) above are mild conditions on the graph, in the following sense:
\begin{fact}\label{fact:everything-is-nice}
    Any structure in a finite relational language is bi-interpretable with a nice graph.
\end{fact}

This is a well-known and folklore fact, and many different proofs exist (e.g.~ \cite[Theorem~5.5.1, Exercise~5.5.9]{Hod93}) so, we have opted not to include our own.

For any graph $\gG$ and odd prime $p$, we write $\Msf(\gG)$ for the $2$-nilpotent group of exponent $p$ which is generated freely in the variety of $2$-nilpotent groups of exponent $p$ by the vertices of $\gG$, with only relations those imposing that two generators commute if and only if they are connected by an edge in $\gG$.

\begin{definition}    
    If $\gG$ is a nice graph, we call $\Msf(\gG)$ \emph{the Mekler group of $\gG$}. More generally, we say that a group is \emph{a Mekler group} if it is elementarily equivalent to a group $\Msf(\gG)$, for a nice graph $\gG$.
\end{definition}

\begin{example}
    Consider the pentagon  $\Psf$, and the Mekler group $G \coloneq \Msf(\Psf)$. One can show that $G$ is a special $p$-group, with centre $\Zsf:= C_p^5$ and quotient $G/\Zsf:= C_p^5$, where $C_p$ is the cyclic group of $p$ elements.
\end{example}
 An axiomatisation of the theory of Mekler groups can be found in \cite[Appendix~A.3]{Hod93}. We now recall some standard terminology.

For the remainder of this section, fix a nice graph $\gG$ and let $\Msf(\gG)$ be the $2$-nilpotent group of exponent $p$, constructed as above. We write $\Zsf$ for the centre of $\Msf(\gG)$. 

\begin{definition}[Equivalence relations $\sim$ and $\approx$]
We define the following two equivalence relations on $\Msf(\gG)$:
\begin{enumerate}
    \item \emph{Centraliser}: $g\sim h$ if and only if $\Csf(g) = \Csf(h)$, where $\Csf(g)$ denotes the centraliser of $g$.
    \item \emph{Powers modulo centre}: $g\approx h$ if and only if there is some $c\in\Zsf$ and some $\alpha\in\{1,\dots,p-1\}$ such that $h = g^\alpha c$.
\end{enumerate}
\end{definition}

\begin{remark}[{\cite[Lemma~A.3.3]{Hod93}}]
    The equivalence relation $\approx$ refines $\sim$, that is, for all $g,h\in \Msf(\gG)$, if $g\approx h$ then $g\sim h$. 
\end{remark}

\begin{definition}[Types and isolation]
Let $g\in \Msf(\gG)\setminus Z$ and $q\in\Nbb$.

$g$ is said to be of \emph{type $q$} if $[g]_\sim$ splits into exactly $q$-many $\approx$-classes.

$g$ is called \emph{isolated} if every non-central element that commutes with $g$ is $\approx$-equivalent to $g$.

If $g$ is of type $q\in\Nbb$, it is said to have \emph{type $q^\iota$} if it is isolated and \emph{type $q^\nu$} otherwise.
\end{definition}

By convention, central elements are excluded from this definition.

\begin{remark}
    For all $q\in\Nbb$, each of the sets of elements of type $q$, $q^\iota$ and $q^\nu$ in $\Msf(\gG)$ are $\emptyset$-definable, since $\sim$ and $\approx$ are $\emptyset$-definable.
\end{remark}

\begin{fact}[{\cite[Lemmas~A.3.2,~A.3.6-A.3.10]{Hod93}}]\label{Fact:Hod}
\
    \begin{enumerate}
        \item Every element of $\Msf(\gG)$ can be written in the form $a_0^{\alpha_0}\cdots a_{m-1}^{\alpha_{m-1}}z$ where $m\in \mathbb{N}$, $a_i\in \gG$ distinct, $\alpha_i\in \{1,\dots,p-1\}$ and $z\in \Zsf$.
        \item The derived subgroup $[\Msf(\gG),\Msf(\gG)]$ is the centre $\Zsf$.
        \item Elements of the form $a^\alpha z$ where $a\in \gG,\alpha \in \{1,\dots,p-1\}, z\in \Zsf$, are of type $1^\nu$.
        \item Elements of the form  $a_0^{\alpha_0}a_1^{\alpha_1}z$ where $\alpha_i\in \{1,\dots,p-1\}$, $a_i \in C$ are distinct and connected, are of type $p-1$.
        \item Elements of the form $g\coloneq a_0^{\alpha_0}\cdots a_{m-1}^{\alpha_{m-1}}z$ where $m\geq 2$, $\alpha_i\in \{1,\dots,p-1\}$, and $a_i$ distinct and connected to a \emph{unique} vertex $a\in \gG$, are of type $p$. This vertex $a$ is called the \emph{handle} of $g$. 
        \item Elements of the form $a_0^{\alpha_0}\cdots a_{m-1}^{\alpha_{m-1}}z$ with distinct vertices $a_i\in C$ and where no vertex in $\gG$ is connected to all $a_i$'s are of type $1^\iota$.
    \end{enumerate}
\end{fact}

In particular, non-central elements of $\Msf(\gG)$ can only be of type $p^\nu$, $(p-1)^\nu$, $1^\nu$ or $1^\iota$, thus, we will suppress the superscript $\nu$, when discussing elements of type $p^\nu$ or type $(p-1)^\nu$. Moreover, elements of type $p-1$ are exactly those elements that can be written as the product of two $\sim$-inequivalent and connected elements of type $1^\nu$.

\begin{notation}
    We denote by $\tE^\nu$, $\tE^p$ and $\tE^\iota$ the set of elements of type respectively $1^\nu$, $p$ and $1^\iota$.  Those subsets of $\Msf(\gG)$ are definable in the language of groups.
\end{notation}

\begin{definition}[Graph $\Gamma$]\label{def:GraphGroup}
    Let $X$ be a subset of a group $G$, such that $X$ is closed under the equivalence relation ${\sim}$. We define $\Gamma(X)$ to be the set $X/{\sim}$ equipped with a graph relation $E$ given by $[a]_{\sim} E [b]_{\sim}$ if and only if $a$ and $b$ commute in $G$. 
\end{definition}

By \cite[Theorem~A.3.10]{Hod93}, the graph $\gG$ is interpreted by $(\Gamma(\tE^\nu),E)$, which is an imaginary sort of $\Msf$. The main goal of this paper is to describe (generalised) indiscernible sequences in a Mekler group
(recall that, in this paper, a Mekler group is a group elementarily equivalent to $\Msf(\gG)$, for some nice graph $\gG$.).
To this end, we need a general description of the elements of $\Msf$. Notice that \Cref{Fact:Hod} fails to give a description of elements in all Mekler groups. Indeed, by compactness, we should expect elements which cannot be written as a finite product of elements of type $1^\nu$. We will construct \emph{independent} sequences of elements that can be extended to a transversal (definition below). 
First, we clarify our notion of independence:

\begin{definition}[Independence]\label{def:independence}
    Let $G$ be a group of exponent $p$. Let $\tE$ be a subset of $G$ and $\bar{a}$ a tuple of elements in $G$. We say that $\bar{a}$ is independent over $\tE$ if for all $\lL_{\mathsf{grp}}$-terms $t(\bar{x}, \bar{y})$ and elements $b\in \tE$, if $G\models t(\bar{a},\bar{b})=1$, then $G\models (\forall x)~t(\bar{x},\bar{b})=1$.
\end{definition}

When the group $G$ is abelian (of exponent $p$), this notion of independence coincides with linear independence if we view $G$ as an $\mathbb{F}_p$-vector space. 

\begin{definition}[Transversal]
    Let $\Msf=(G,\cdot,1)$ be a Mekler group. A \emph{transversal of $\Msf$} is a set $X$ which can be written as the union of three disjoint sets $X^\nu,X^p$, and $X^\iota$ where:
     \begin{itemize}
         \item $X^\nu$ is a subset of $\tE^\nu$ independent over $\Zsf$ and maximal for this property.
         \item $X^p$ is a subset of $\tE^p$ independent over $\braket{\Zsf,\tE^\nu}$ and maximal for this property.
         \item $X^\iota$ is a subset of $\tE^\iota$ independent over $\braket{\Zsf,\tE^\nu,\tE^p}$ and maximal for this property.
     \end{itemize}
\end{definition}

Every Mekler group $\Msf$ admits a transversal $X$, and all elements of $\Msf$ can be written as a finite product 
\[
    a_1^{r_1}\cdots a_n^{r_n} g_1^{s_1} \cdots g_m^{s_m} w_1^{t_1} \cdots w_k^{t_k} z,
\]
where $a_i\in X^\nu$, $g_i \in X^p$, $w_i \in X^\iota, z\in \Zsf$ pairwise distinct and $r_i,s_i,t_i$ are in $\{1,\dots,p-1\}$. This description is unique once we have fixed an order on the elements of $X$. We therefore often see $X$ as a tuple rather than a subset.

\begin{fact}[{\cite[Lemma~2.7]{CH18}}]\label{Fact:CH}
    For every small (possibly infinite) tuple $x$, there is a partial type $\phi(x)$ such that $\Msf \models \phi(x)$ if and only if $x$ can be extended to a transversal of $\Msf$.  
\end{fact}

This description of elements of $\Msf$ (which is usual in the literature of Mekler groups) can be made more precise: Let $X^\zeta$ be a subset of $\Zsf$ independent over $\braket{\tE^\nu,\tE^p,\tE^\iota}$, and maximal with this property. Then an element $z$ of the centre can be written as:
\[
    \prod_{x,y \in X} [ x,y]^{n_{x,y}} \prod_{z \in X^\zeta} z^{n_z}, 
\]
where $n_z,n_{x,y}\in \{0,\dots,p-1\}$ are almost all trivial. Again, we have uniqueness if we fix orders on $X^\zeta$ and $X$ (and if we consider the lexicographic order on $X^2$). This follows easily from \cite[Axiom~8]{Hod93}.

Thus, every element of $\Msf$ can be written in the form:
\[
    \prod_{x \in X}x^{n_x} \prod_{x,y \in X} [ x,y]^{n_{x,y}} \prod_{z \in X^\zeta} z^{n_z}. 
\]

We call $X^\nu X^p X^\iota X^\zeta$ a \emph{full transversal} of $\Msf$.

\subsection{Generalised indiscernibles}\label{subsec:generalised-indiscernibles}

Throughout the paper, if $\sS$ is a first-order structure, we let $\dom(\sS)$ denote the domain of $\sS$. In this section, we fix a language $\mathcal{L}$ and a countable language $\lL^\prime$. Let $T$ be a $\mathcal{L}$-structure with a monster model $\Mbb$ and let $\sI$ be an $\lL^\prime$-structure. 

\begin{definition}[Generalised Indiscernibles]\label{def:generalised-indiscernibles}
    Given an $\sI$-indexed sequence of tuples $ (\bar a_i)_{i\in\sI}$ from $\Mbb$, and a small subset $A$ of $\Mbb$, $(\bar a_i)_{i\in\sI}$ is called an \emph{$\sI$-indiscernible sequence over $A$}, if for all positive integers $n$ and all sequences $i_1,\dots,i_n,j_1,\dots,j_n$ from $\sI$ we have that if:
    \[
        \qftp_\sI^{\lL^\prime}(i_1,\dots,i_n) = \qftp_\sI^{\lL^\prime}(j_1,\dots,j_n)
    \]
   then 
    \[
        \tp(\bar a_{i_1},\dots,\bar a_{i_n}/A) =  \tp(\bar a_{j_1},\dots,\bar a_{j_n}/A).
    \]
    If $A=\emptyset$,  $(\bar a_i)_{i\in\Ical}$ is called an \emph{$\sI$-indiscernible sequence}.
\end{definition} 

We will always assume that the indexing structure $\sI$ is a Ramsey structure, i.e. the Fraïssé limit of a \emph{Ramsey class}. The crucial fact about sequences indexed by Ramsey structures is that the conclusion of \emph{the standard lemma} \cite[Lemma~5.1.3]{TZ12} holds for them. More precisely, we have the following theorem, originally due to Scow, \cite{Sco15}, later generalised in \cite{MP23}:

\begin{theorem}[Generalised Standard Lemma]\label{thm:genstandlem}
    Let $\lL^\prime$ be a first-order language and $\sI$ an infinite, locally finite $\lL^\prime$-structure. Then, the following are equivalent:
    \begin{enumerate}
        \item $\Age(\sI)$ is a Ramsey class.\footnote{Recall, the \emph{age} of a structure $\sI$ is the class of all its finitely generated substructures, denoted here by $\Age(\sI)$.}
        \item $\sI$-indexed indiscernibles have the modelling property.
    \end{enumerate}    
\end{theorem}

The precise definitions of a \emph{Ramsey class} and of the \emph{modelling property} will be omitted, but we will recall all the other tools we need in this paper. For a proper introduction to these concepts, we direct the reader to \cite{Bodirsky_2015,Sco15}. We only recall the notion of sequence locally based on another:

\begin{definition}[Locally based]
    Let $\Mcal$ be an $\Lcal$-structure. Given indexing structures $\Ical$ and $\Ibb$ in the same language, an $\Ibb$-indexed sequence $(\bar a_i:i\in\dom(\Ibb))$ in $\Mcal$ is said \emph{(locally) based} on $(\bar b_i:i\in\dom(\Ical))$ if for any finite set $\Delta$ of $\Lcal$-formulas and any finite tuple $(i_1,\dots,i_n)$ from $\Ibb$ there is a finite tuple $(j_1,\dots,j_n)$ from $\Ical$ such that:
    \begin{itemize}
        \item $\qftp_\Ibb(i_1,\dots,i_n) = \qftp_\Ical(j_1,\dots,j_n)$.
        \item $\tp^\Delta(\bar a_{i_1},\dots,\bar a_{i_n})=\tp^\Delta(\bar b_{j_1},\dots,\bar b_{j_n})$.
    \end{itemize}
\end{definition}

\begin{definition}[Reduct]
      Let $\sI$ and $\sJ$ be structures on the same domain. Then $\sJ$ is a \emph{(first-order) reduct} \emph{of $\sI$} if every relation and every function in $\sJ$ is definable in $\sI$ without parameters. We say that $\sJ$ is a \emph{strict reduct} if it is a reduct of $\sI$, and $\sJ$ and $\sI$ are not interdefinable.
\end{definition}

\begin{definition}\label{def:indiscernible collapsing}
    Let $\sI$ and $\sJ$ be two structures and assume that $\sJ$ is a reduct of $\sI$. A structure $\sM$ \emph{collapses indiscernibles from $\sI$ to $\sJ$} if every $\sI$-indiscernible sequence in $\sM$ is a $\sJ$-indiscernible sequence. A complete theory $T$ \emph{collapses indiscernibles from $\sI$ to $\sJ$} if the monster model collapses indiscernibles from $\sI$ to $\sJ$. We say that $T$ \emph{collapses $\sI$-indiscernibles} if it collapses any $\sI$-indiscernible sequence $(a_i)_{i \in \sI}$ to $\sJ$-indiscernible sequence, where $\sJ$ is some strict reduct of $\sI$ (depending on the sequence $(a_i)_{i \in \sI}$).
\end{definition}

\subsection{Model-theoretic dividing lines}\label{sub:mt-div}

\subsubsection{$\cK$-configurations}
The notion of $\cK$-configurations originated from the work of Guingona and Hill \cite{GH19} and was further developed in \cite{GPS23}. In this section, we will fix a relational first-order language $\lL'$ with signature $\Sig(\lL')$ and an arbitrary first-order language $\lL$. We abusively use the symbol $\lL$ to also denote the set of (parameter-free) $\lL$-formulas. Throughout, we will use $\cK$ to denote a class of finite $\lL'$-structures closed under isomorphism.

\begin{definition}[$\cK$-configuration, {\cite[Definition 5.1]{GPS23}}]\label{def: K-config}
    An $\lL$-structure $\sM$ \emph{admits (or codes) a $\cK$-configuration} if there are a positive integer $n$, a function $I:\Sig(\lL') \rightarrow \lL$ and a sequence of functions $(f_A: A\in \cK)$ such that: 
    \begin{enumerate}
        \item for all $A\in \cK$, $f_A:A \rightarrow \sM^n$,
        \item for all $R\in \Sig(\lL')$, for all $A\in \cK$, for all $a\in A^{\mathsf{arity}(R)}$,
            \[
                A \models R(a) \Leftrightarrow \sM \models I(R)(f_A(a)).
            \]    
    \end{enumerate}
        
    A theory $T$ \emph{admits a $\cK$-configuration} if there is a model $\sM\models T$ which admits a $\cK$-configuration. We denote by $\codingcla{\cK}$ the class of theories that admit a $\cK$-configuration, and by $\ncodingcla{\cK}$ the class of theories that \emph{do not} admit a $\cK$-configuration.
\end{definition}

\begin{example} \ 
    \begin{itemize}
        \item The dense linear order $(\mathbb{Q},<)$ codes (but does not interprets) the discrete order $(\mathbb{N}, < , P_1)$ where $P_1:=\{(n,n+1)\}$.
        Indeed consider an increasing sequence $(a_i)_{i\in \mathbb{N}}$. An $\Age(\mathbb{N})$-configuration is given by:
        \begin{enumerate}
            \item for $J\subseteq \mathbb{N}$, $f_J: n\in J \mapsto (a_n,a_{n+1}) \in \mathbb{Q}^2$,
            \item $I(<)(\mathbb{Q}^2) =  \{ (a,b), (a',b') :   a<a' \} $ and $I(P)(\mathbb{Q}^2)=\{ (a,b), (a',b') :   a=b' \} $. 
        \end{enumerate}
        \item A theory $T$ is stable if and only if it does not admit a $\mathsf{DLO}$-configuration. In other words, if and only if it does not code a linear order.
        \item A theory $T$ is NIP if and only if it does not admit a $\mathsf{RG}$-configuration, where $\mathsf{RG}$ is the random graph.
        \end{itemize}
\end{example}

The following theorem originally appeared in \cite[Theorem 3.14]{GH19} with the additional assumption that $\Age(\sI)$ is a Fraïssé class with the \emph{strong amalgamation property}. Later, in \cite{MPT23}, it was extended to remove this assumption, and this is the version we present below.

\begin{theorem}[{\cite[Theorem 3.3]{MPT23}}]\label{thm:Ramsey class collapse}
    Let $\sI$ be an $\aleph_0$-categorical Fraïssé limit of a Ramsey class. Then the following are equivalent for a theory $T$:
        \begin{enumerate}
            \item\label{item:T noncoding} $T\in \ncodingcla{\sI}$.
            \item\label{item:T collapses} $T$ collapses $\sI$-indiscernibles.
        \end{enumerate}
\end{theorem}

\subsubsection{The NFOP$_k$ property}

Most of the classical model-theoretic dividing lines (such as stability, NIP, NSOP, etc. see \cite{She90}) are essentially \emph{binary} (in that the local combinatorial configuration whose absence guarantees strong structure results is one given by formulas in two tuples of variables). In recent years, the power of these dividing lines in finite \emph{graph}-combinatorics has become evident, for instance, in the \emph{stable regularity lemma} of Malliaris and Shelah \cite{MS2014}. Finding appropriate dividing lines that allow us to prove similar results for \emph{hypergraphs} is one of the most prominent topics in the nexus of model theory and combinatorics. 

In recent years, it has become evident that these dividing lines should not be binary. The $k$-Independence Property is one of the more well-established higher-arity dividing lines, generalising the Independence Property. In \cite{AACT23}, a rather robust higher-arity generalisation of stability (i.e., \emph{order property}) is developed via the \emph{functional order property}\footnote{We note that alternative notions can be found in the literature, such as Takeuchi’s NOP$_2$ dividing line.}. As this material is not (yet) standard, we provide a brief summary below and refer the reader to the introduction of \cite{AACT23} for a historical account of the concept.

\begin{definition}\label{def:FOPk}
    Let $T$ be a complete $\lL$-theory and $k\in\Nbb$. An $\lL$-formula $\varphi(x,x_1\ldots,x_{k})$ has the \emph{$k$-functional order property} (\emph{FOP$_k$}) in $T$ if there is a model $\Mcal\models T$ and    
    \begin{itemize}
        \item a sequence $(a_f)_{f:\omega^{k-1}\to\omega}$ of elements in $M^{x}$,
        \item for $1\leq t\leq k$, sequences $(b^t_i)_{i<\omega}$ of element in $M^{x_{t}}$
    \end{itemize}
    such that 
    \[\Mcal\models\varphi(a_f,b^1_{i_1},\ldots,b^k_{i_k}) \ \Leftrightarrow \ i_k\leq f(i_1,\ldots,i_{k-1}).\]
    
    The theory $T$ is NFOP$_k$ if no $(k+1)$-partitioned formula has FOP$_k$ in $T$. 
\end{definition}

One of the results from \cite{AACT23} that will be important in this paper is a characterisation of FOP$_k$ via collapsing indiscernibles. We will recall this in the sequel, but first, we need to introduce some terminology.

Let $k$ be a positive integer. We denote by $\lL_k$ the following language:
\[
    \lL_k=\{P_1,\ldots,P_{k+1},<,<_k,R\}.
\]
We let $Q$ denote the $k$-ary relation $P_1\times \dots\times P_k$.

\begin{definition}[{\cite[Definition 3.12]{AACT23}}]\label{def:Tk}
    Define $T_k$ to be the $\lL_k$-theory consisting of the following axioms:
    \begin{enumerate}
        \item $P_1,\ldots, P_{k+1}$ is a partition.
        \item $<$ is a linear order with $P_1<\ldots< P_{k+1}$.
        \item $R$ only holds on $P_1\times\ldots\times P_{k+1}$ (which we also view as $Q\times P_{k+1}$).
        \item $<_k$ only holds on $Q\times Q$, and is a linear order on $Q$. 
        \item For any $\bar{x},\bar{y}\in Q$ and $w,z\in P_{k+1}$, $\big(\bar{x}\leq_k\bar{y} \wedge R(\bar{y},w)\wedge w\leq z\big)\rightarrow R(\bar{x},z)$.
    \end{enumerate}
\end{definition}

\begin{definition}
    We denote by $\cla{H}_k$ the class of finite models of $T_k$. 
\end{definition} 

\begin{fact}[{\cite[Corollaries 3.14 \& 3.16]{AACT23}}]
    $\cla{H}_k$ is a Fraïssé class with the Ramsey property. In particular, the Fraïssé limit $\Hcal_k$ has the modelling property.
\end{fact}

\begin{fact}[{\cite[Theorem 4.15]{AACT23}}]\label{fact:NFOP_k-collapse}
    Let $T$ be a complete theory with monster model $\M$. The following are equivalent.
    \begin{enumerate}
        \item $T$ is NFOP$_k$.
        \item Every $\Hcal_k$-indexed indiscernible sequence in $\M$ is $(\lL_k\setminus\{R\})$-indiscernible.
    \end{enumerate}
\end{fact}

\section{Relative quantifier elimination}\label{sec:RQE}
We deal, in this section, with quantifier elimination in Mekler groups. We refer the reader to \cite[Appendix A]{Rid17} for basic definitions of relative quantifier elimination and related notions. We recall simply here the key concept of a closed sort:
\begin{definition}[{\cite[Definition A.7]{Rid17}}]\label{def:closed}
    Let $\sM$ be a multisorted structure. A set of sorts $\Sigma$ is called \emph{closed} if any predicate involving a sort in $\Sigma$ and any function with a domain involving a sort in $\Sigma$ only involves sorts in $\Sigma$. 
\end{definition}

A sort $\mathsf S$ will be abusively called closed if $\mathsf S^\eq$ is closed.
A closed sort has good syntactical properties and all results of quantifier elimination presented in this paper will be therefore relative to a closed sort.

Let $\gG = (V;E)$ be a nice graph. As shown by Mekler, and summarised in the previous section, the graph $\gG$ is interpretable in the Mekler group $\Msf(\gG)=(G,\cdot,1)$. The goal of this section is to provide a detailed description of the ``definable structure'' of $\Msf(\gG)$, relative to the ``definable structure'' of $\gG$. We will proceed in a step-by-step fashion, summarised in the reduction diagram shown in \Cref{fig:reduction}.

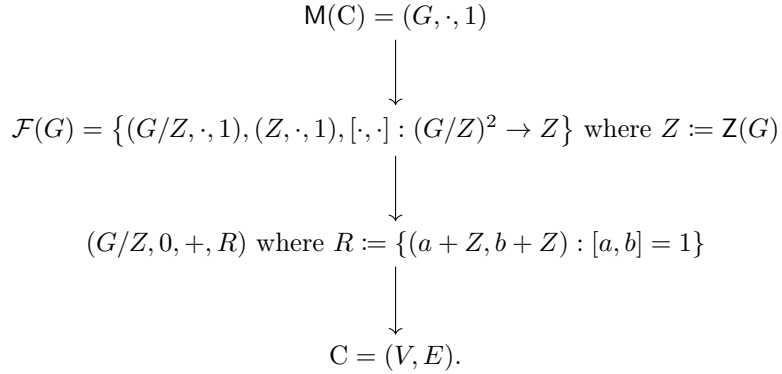
\begin{figure}[H]
\centering
        \begin{tikzpicture}
            \node {$\Msf(\gG) = (G,\cdot,1)$ } 
                child { 
                    node { $\mathcal{F}(G)=\left\{(G/Z,\cdot,1),(Z,\cdot,1),[\cdot,\cdot]:(G/Z)^2 \rightarrow Z \right\}$ where $Z\coloneq\Zsf(G)$} edge from parent[->]
                    child { 
                        node {$(G/Z,0,+,R)$ where $R\coloneq \{(a+Z,b+Z) : [a,b]=1\}$}
                    child{
                            node {$\gG=(V,E)$.} 
                        }
                    }
                };
        \end{tikzpicture}
    \caption{Reduction diagram for Relative Quantifier Elimination}    \label{fig:reduction} 
\end{figure}

A vertical arrow $A\rightarrow B$, in \Cref{fig:reduction}, should be understood as a relative quantifier elimination statement of the form:

\begin{center}
    ``$A$ eliminates quantifiers relative to $B^\eq$.''
\end{center}

The precise languages in which this happens will be introduced in the following paragraphs. These results will be enough for our practical purposes. As discussed in the introduction, the purpose of each relative quantifier elimination $A\rightarrow B$ is to obtain transfer principles, and notably a characterisation of (generalised) indiscernibles in $A$ relative to (generalised) indiscernibles in $B$. This will be done in the next section, where we will also point out how these characterisations allow us to establish various transfers of dividing lines.

Notice that one cannot simply ``combine the arrows'' of our reduction diagram, as we eliminate quantifiers relatively at the cost of adding (finite) imaginary sorts (which is why an arrow $A\to B$ indicates that $A$ eliminates quantifiers relative to $B^\eq$ rather than $B$). In fact, we will not present a complete quantifier elimination result from the Mekler group $\Msf(\gG)$ to the nice graph $\gG$, as it seems to us that such a result would involve a complicated language and will not reveal much more information than our results. This appears to be due to the fact that our structures do not eliminate \emph{finite imaginaries},\footnote{See \cite[Definition~8.4.9(1)]{TZ12}.} and we do not attempt to find sorts in which the various structures eliminate (finite) imaginaries. 

\subsection{First reduction: From \texorpdfstring{$G$}{G} to \texorpdfstring{$\mathcal{F}(G)$}{F(G)}}

Let $G$ be a $2$-nilpotent group of exponent $p$. For the remainder of this section, we will denote by $\Zsf$ its centre, $\Zsf(G)$. As in \cite{Bau02}, $\Fcal(G)$ will denote the following structure:
\[
    \Fcal(G)\coloneq \left((\vV,+,0),\,(\vW,+,0),\,\beta\colon\vV\times \vV\rightarrow \vW \right)
\]
where:
\begin{itemize}
    \item $(\vV,+,0)$ denotes the $\mathbb{F}_p$-vector space $(G/\Zsf,\cdot,1)$ (with an additive notation), 
    \item $(\vW,+,0)$ denotes the $\mathbb{F}_p$-vector space $(\Zsf,\cdot,1)$ (with an additive notation),
    \item $\beta \colon \vV\times \vV \rightarrow \vW$ denotes the \emph{commutator map}, that is:
    \[
    \begin{aligned}
      \beta \colon G/\Zsf\times G/\Zsf&\to \Zsf \\
      (a \bmod \Zsf, b \bmod \Zsf) &\rightarrow [a,b].
    \end{aligned}
    \]
\end{itemize}
Notice that these sorts are all interpretable in the language of groups, so we can expand $(G,\cdot,1)$ to a structure in the (multisorted) language $\Lcal_{G,\Fcal(G)}$ (following the notation above):
\[ 
    \{ (G,\cdot,^{-1},1),\, (\vV,+,0),\, (\vW,+,0),\,\beta \colon\vV\times \vV\rightarrow \vW,\, \pi \colon G \rightarrow \vV,\, \rho \colon G \rightarrow \vW \}. 
\]
where:
\begin{itemize}
    \item $\pi \colon G \rightarrow \vV$ denotes the natural projection map.
    \item $\rho \colon G \rightarrow \vW$ denotes the `inverse' inclusion map, that is: 
    \[
    \begin{aligned}
        \rho \colon G&\rightarrow \vW \\
        g &\mapsto 
            \begin{cases}   0 & \text{ if } g\notin \Zsf, \\ 
                            g & \text{ if } g\in \Zsf.
            \end{cases}
    \end{aligned}
    \]
\end{itemize}

The following fact is a reformulation of \cite[Corollary 3.1]{Bau02}
\begin{fact}\label{fact:RQEGtoF(G)}
    Let $G$ be a $2$-nilpotent group of exponent $p$. Then, $(G,\cdot,1)$ eliminates quantifiers relative to $\mathcal{F}(G)$, in the language $\mathcal{L}_{G,\Fcal(G)}$. 
\end{fact}
In particular, any $\lL_\mathsf{grp}$-formula $\phi(\bar{x})$ is equivalent, modulo $\Th(G)$, to a formula of the form:
\[
    \phi_{\mathcal{F}(G)}(\pi(t(\bar{x})),\rho(t(\bar{x}))),
\]
where $\phi_{\mathcal{F}(G)}(\bar{y},\bar{z})$ is a formula in the language $\{(\vV,+,0),(\vW,+,0),\beta\}$, and $t$ is a tuple of $\mathcal{L}_{grp}$-terms.

Notice, in particular, that the formula:
\[
    t(x)=1
\]
for a term in $\mathcal{L}_\mathsf{grp}$ is equivalent to
\[
    \pi(t(x))=0 \wedge \rho(t(x))=0.
\]

\begin{proof}[Proof of \Cref{fact:RQEGtoF(G)}  (Sketch).]
  We apply the well-known criterion for relative quantifier elimination and argue by back-and-forth and saturation. Let $G$ and $G'$ be $2$-nilpotent groups of exponent $p$, such that $G'$ is $\vert G \vert $-saturated and with isomorphic substructures $\Hcal = (H,H_\mathcal{F}) \subseteq (G,\mathcal{F}(G))$ and $\Hcal' =(H',H_\mathcal{F}') \subseteq (G',\mathcal{F}(G'))$. Let $f=(f_G,f_\mathcal{F})$ denote the isomorphism between $\Hcal$ and $\Hcal'$, and assume that $f_\mathcal{F}\colon  H_\mathcal{F} \rightarrow H_\mathcal{F}'$ is elementary.  By saturation, ${f}_\mathcal{F}$ can be extended to an embedding $\Tilde{f}_{\mathcal{F}}$ of $\mathcal{F}(G)$ into $\mathcal{F}(G') $. 
  Arguing as in \cite[Corollaries 3.1, 3.2]{Bau02}, we can extend $f_G\cup \Tilde{f}_{\mathcal{F}}$ to an embedding of $\Tilde{f}$ of $(G,\mathcal{F}(G))$ into $(G',\mathcal{F}(G'))$. 
\end{proof}

We call the structure $\{(\vV,+,0),(\vW,+,0),\beta \colon \vV^2\rightarrow \vW \}$ \emph{an alternating bilinear system} of $\mathbb{F}_p$-vector spaces.
As shown by Baudisch, in \cite[Section~3]{Bau02}, $\mathcal{F}$ is a functor from the category $\mathbb{G}_{2,p}$ of $2$-nilpotent groups of (finite) exponent $p$ to the category $\mathbb{B}_p$ of alternating bilinear systems of $\Fbb_p$-vector spaces.

When $G$ is a Mekler group $\mathcal{F}(G)$ enjoys an additional property which will be defined now and will be useful for our analysis.

\begin{definition}
    Let $\{(\vV,+,0),(\vW,+,0),\beta \colon \vV^2\rightarrow \vW \}$ be an alternating bilinear system, and let $\vV'\subseteq \vV$ be a subspace. 
    \begin{itemize}
        \item A basis $(v_i)_i$ of $\vV'$ is called \emph{separated} if for all sequences $(\alpha_{i,j})_{i<j}$ of scalars with only finitely many $\alpha_{i,j}$ non-trivial, if
        \[
            \sum_{i<j} \alpha_{i,j}\beta(v_i,v_j) =0,
        \]
        then for all $i<j$, either $\alpha_{i,j}=0$ or $\beta(v_i,v_j)=0$.
        \item  $\{(\vV,+,0),(\vW,+,0),\beta \colon \vV^2\rightarrow \vW \}$ is called \emph{separated} if it admits a separated basis.
    \end{itemize}
\end{definition}

\begin{example}
Let $\vV$ be any vector space.
    \begin{enumerate}
        \item The bilinear system $(\vV,\vV\wedge \vV,\wedge)$, where $\vV \wedge \vV$ is the wedge product, is separated. This follows from the fact that for any basis $(v_i)_{i}$ of $\vV$, $(v_i\wedge v_j)_{i<j}$ is a basis of $\vV\wedge \vV$, by the universal property of the wedge product.  
        \item Let $v_1,v_2,v_3,v_4$ be a basis of $\vV$. Let $\vW$ be $\vV\wedge \vV / \braket{v_1\wedge v_2-v_3\wedge v_4}$ and $\beta\colon\vV^2\rightarrow\vW$ be induced by the wedge product. Then the bilinear system $(\vV,\vW,\beta)$ is not separated. Indeed, for all linearly independent vectors $v,v'\in \vV$, observe that $\beta(v,v')\neq 0$. Therefore, a separated basis of $\vV$ would give six vectors linearly independent in $\vW$, but $\dim(\vW)=5$.
        \item If $\vV$ has dimension $\leqslant 2$, then any bilinear system of the form $(\vV,\vW,\beta)$ is separated.
    \end{enumerate}
\end{example}

Note that if $\Vcal$ is infinite-dimensional, then separatedness of $\Vcal$ is \textit{a priori} not a first-order property. In the case of Mekler groups, we can, in a first-order way, express the stronger property that any finite-dimensional vector subspace is contained in a finite-dimensional separated vector subspace. 
To this end, we need the following fact, which is a slightly weakened reformulation of {\cite[Axioms 8 \& 9]{Hod93}}:
\begin{fact}\label{fact:axioms8&9}
    Let $k,m,n$ be integers and let $g_0,\dots g_{n-1}$ be a sequence of independent elements in a Mekler group $\Msf$ such that:
    \begin{itemize}
        \item $g_0,\dots, g_{k-1}$ are of type $1^\nu$,
        \item $g_k,\dots, g_{m-1}$ are of type $p$, and no combination of $g_k,\dots, g_{m-1}$ is a product of at most $k+1$ elements of type $1^\nu$,
        \item $g_m,\dots, g_{n-1}$ are of type $1^\iota$, and no combination of $g_m,\dots, g_{n-1}$ is a product of at most $m+1$ elements of type $1^\nu$ or of type $p$.
    \end{itemize}
    Then, $g_0 \bmod \Zsf ,\dots, g_{n-1} \bmod \Zsf $ is a separated basis of $\braket{g_0 \bmod \Zsf ,\dots, g_{n-1} \bmod \Zsf }$ in $\mathcal{F}(G)$.
    \end{fact}
    \begin{corollary}\label{cor:SepVectorsupspace}
        There is a function $f:\Nbb\to\Nbb$ such that the following holds:
        For $G$ a Mekler group, let $(\Vcal,\Wcal)$ denotes $\mathcal{F}(G)$. Then, 
        \[\tag{$*_f$}
        \begin{aligned}
            \text{every vector subspace of $\Vcal$ of dimension $n$ is included in } \\
        \text{ a separated vector space of dimension at most $f(n)$.}
        \end{aligned}
        \]
        \end{corollary}
    \begin{proof}
    We show that $f(n)=2^{2^n}$. Let $V$ be a vector space of dimension $n$, and let $0\leq k\leq m\leq n$ be integers and $g_0,\dots g_{n-1}$ a basis of $V$ such that 
        \begin{itemize}
        \item $g_0,\dots, g_{k-1}$ are of type $1^\nu$,
        \item $g_k,\dots, g_{m-1}$ are of type $p$,
        \item $g_m,\dots, g_{n-1}$ are of type $1^\iota$.
    \end{itemize}
    By induction on $(n-m,m-k,k)$ with the lexicographic order, we show that $\Vcal$ is included in a vector space of dimension at most $(k+1)2^{(m+1)2^{n-m}-k}$ (which is less than $2^{2^n}$). If $g_0,\dots g_{n-1}$ satisfy the conditions of \Cref{fact:axioms8&9}, then $g_0,\dots g_{n-1}$ is separated and there is nothing to show. If not, then at least one of the following holds:
    \begin{itemize}
        \item there is a combination of the $g_i$'s, $m \leq i< n$, which is a product of at most $m+1$ elements $h_0
    ,\dots,h_{m}$ of type $p$ or $1^\nu$, or
        \item there is a combination of the $g_i$'s, $k \leq i< m$, which is a product of at most $k+1$ elements $h_0
    ,\dots,h_{k}$ of type $1^\nu$.
    \end{itemize}
    In any case, we consider the vector space $\Vcal'$ generated by $g_0,\dots, g_{n-1}, h_0
    ,\dots,h_{m}$ (resp. $g_0,\dots, g_{n-1}, h_0
    ,\dots,h_{k}$ ) and assume that $h_0
    ,\dots,h_{k}$ are independent. Let $g_{i_0},\dots, g_{i_k}$ a proper subset of the $g_i$'s such that $g_{i_0},\dots, g_{i_k}, h_0
    ,\dots,h_{m}$ \par
    (resp. $g_{i_0},\dots, g_{i_k}, h_0
    ,\dots,h_{k}$) is a basis of $\Vcal'$. By induction, $\Vcal$ is included in a vector space of at most dimension $(k+1)2^{(2m+2)2^{n-m-1}-k} $ (resp. $(2k+2)2^{(m+1)2^{n-m}-k-1} $), which is equal to $(k+1)2^{(m+1)2^{n-m}-k}$.
    \end{proof}

    This bound is probably not optimal. In fact, as far as we know, an optimal bound could be $f(n)=n$, i.e.~ every vector subspace of $\Vcal$ could be separated.

\subsection{Second reduction: From \texorpdfstring{$\{(\vV,+,0),(\vW,+,0),\beta\}$}{\{(V,+,0),(W,+,0), β\} } to \texorpdfstring{$(\vV,+,0,R)$}{(V,+,0,R)}}
We now analyse the definable sets in the bilinear system $\mathcal{F}(G)$. First, note that the following binary relation can be defined on $\vV$:
\[
    R\coloneqq \{(v_0,v_1) \ \vert \ \beta(v_0,v_1)=0 \}.
\]

Heuristically, by separatedness, one should be able to recover all the first-order expressiveness of $\mathcal{F}(G)$ from the structure $(\vV,R,+,0)$. However, within the structure $\{(\vV,+,0),(\vW,+,0),\beta\colon \vV^2\rightarrow \vW \}$, notice that the sort $\vV$ is not closed in the sense of \Cref{def:closed}. This needs to be addressed first. Since vector spaces do not eliminate finite imaginaries, this will come at the cost of introducing sorts $B_n$, for all positive integers $n$, which are a certain quotient of $P_{\leq k_n}(\vV^n)$, the set of subsets of $n$-tuples in $\vV$ of size less than a integer $k_n$ depending on $n$. 


For $n\in \mathbb{N}$, we denote by $\mathbb{A}_n(\mathbb{F}_p)$ the (finite) set of antisymmetric (i.e. skew-symmetric) $n\times n$ matrices $A=(a_{i,j})$ with coefficients in $\mathbb{F}_p$, and by $\mathbb{A}(\mathbb{F}_p)$ the union of $\Abb_n(\Fbb_p)$, for all $n\in\Nbb$.

For $n\in \mathbb{N}$, let $\vW_n$ be the following the set 
\[
    \left\{w\in \vW \ \Bigg\vert \ \exists v_0,\dots,v_{n-1}\in \vV,\, \exists (a_{i,j})\in \mathbb{A}_n(\mathbb{F}_p) \text{ s.t. } w=\sum_{i<j}a_{i,j}\beta(v_i,v_j)\right\},
\]
that is, the set of elements \textit{of order} $n$.
For two positive integers $n>m$, there is a natural inclusion of $\vW_m$ in $\vW_n$ and for every positive integer $n$, there is a function $+\colon (\vW_n)^2 \rightarrow \vW_{2n}$ which is the trace of the addition in $\vW$. 

It follows from separatedness that $\vW_n$, with the structure above, is an imaginary of the structure $(\vV,+,0,R)$, for all $n\in\Nbb$:
\begin{lemma}
    Let $(\vV,\vW,\beta)$ be an alternating bilinear system satisfying Property $(*_f)$ of \Cref{cor:SepVectorsupspace}, and denote by $\simeq$ the equivalence relation on $\mathbb{A}_n\times \vV^n$ given by: 
    \[
        (A,\bar{v})\simeq(A',\bar{v}') \Leftrightarrow \sum_{i<j}a_{i,j}\beta(v_i,v_j)=\sum_{i<j}a_{i,j}'\beta(v_i',v_j').
    \]
    Let $B_n\coloneq\mathbb{A}_n\times \vV^n/\simeq$. Then:
    \begin{enumerate}
        \item The equivalence relation $\simeq$ is interpretable in $(\vV,+,0,R)$.
        \item The set $B_n$ is an imaginary sort of $(\vV,+,0,R)$ and can be identified with $\vW_n$.
        \item The addition $+ \colon (\vW_n)^2 \rightarrow  \vW_{2n}$ is also interpretable in $(\vV,+,0,R)$.  
    \end{enumerate}
\end{lemma}

Notice that $\bigcup_n \vW_n$ is $\braket{\beta(\vV,\vV)}$, the $\Fbb_p$-vector subspace of $\vW$ generated by the image of $\beta$, and it is not equal to $\vW$ in general. 

\begin{notation}\label{notation:pi-f}
    Let $(\vV,\vW,\beta)$ be a bilinear system satisfying Property $(*f)$, and let $\simeq$ be the equivalence relation defined in the above lemma. For $A\in \mathbb{A}_n(\Fbb_p)$, we write:
    \[
    \begin{aligned}
        \pi_A: \vV^n    &\rightarrow B_n,\\
        \bar{v}         &\mapsto (A,\bar{v})/\simeq,
    \end{aligned}
    \]
    for the natural projection, and:
    \[
    \begin{aligned}
        f_n: \vW    &\rightarrow B_n \\
         w          &\mapsto  
                \begin{cases}
                    (A,\bar{v}) / \simeq & \text{if }w=\sum_{i<j}a_{i,j}\beta(v_i,v_j),\\
                    \und & \text{if }w\notin \vW_n.
                \end{cases}
    \end{aligned}
    \]
    for the reverse inclusion. Recall that the constant symbol $\und$ is interpreted as ``undetermined''.
\end{notation}
\begin{proof} \
    \begin{enumerate}
        \item By assumption, Property $(*_f)$ holds, and the vector space $\Vcal \coloneq \braket{\bar{v},\bar{v}'}$ is included in a separated vector space $\Vcal'$ of dimension at most $f(n)$, where $n$ is the dimension of $\Vcal$. Given a separated basis $\bar{u}=(u_i)$ of $\Vcal'$, the equality 
        \[
        \sum_{i<j}a_{i,j}\beta(v_i,v_j)=\sum_{i<j}a_{i,j}'\beta(v_i',v_j')
        \]
        can be rewritten in this basis to give an equation:
        \[
            \sum_{i<j}c_{i,j}\beta(u_i,u_j)=0, 
        \]
        for coefficients $c_{i,j}\in\mathbb{F}_p$.   Therefore, $\sum_{i<j}a_{i,j}\beta(v_i,v_j)=\sum_{i<j}a_{i,j}'\beta(v_i',v_j')$ holds if and only if, in some basis $\bar{u}$ of a vector subspace of dimension  at most $f(n)$, for all $i<j$, $c_{i,j}=0$ or $(u_i,u_j)\in R$. It follows that the equivalence relation $\simeq$ can be written using only the addition in $\vV$ and the predicate $R$.
        \item Immediate.
        \item By Property $(*_f)$, we can work in a finite dimensional separated vector subspace.  Consider two elements $w',w''$ and an appropriate separated basis $(v_i)_{i<m}$ with $A'=(a_{i,j}'), A''=(a_{i,j}'')\in \mathbb{A}_m(\Fbb_p)$, such that $f_m(w')=\pi_{A'}(v_i)$ and $f_m(w'')=\pi_{A''}(v_i)$. 
        
        Suppose that $A=A'+A''$. Then, we simply have that 
        \[
        f_m(w'+w'')=\pi_A(v_i).
        \]
        Therefore, for $b,b',b''\in B_n$,
        \[
            \exists w,w' \in \vW_n \left( f_n(w')=b'\wedge f_n(w'')=b''\wedge f_n(w+w')=b\right)
        \] 
        holds if and only if for some $m\leq 2n$, there are $A',A''\in \mathbb{A}_{m}(\Fbb_p)$ and $\bar{v}'\in \vV^m$ such that $\pi_{A'}(\bar{v})=b'$, $\pi_{A''}(\bar{v})=b''$ and $\pi_{A'+A''}(\bar{v})=b$. The statement follows.
    \end{enumerate}
\end{proof}

Of course, we can recover the alternating map by considering the function $f_{A}$, where 
\[
    A=
    \begin{pmatrix} 
        0  & 1 \\
        -1 & 0 
    \end{pmatrix}.
\]
Then, for a non trivial $w\in \vW$ we have that $f_2(w)=\pi_A(v_0,v_1)$ if and only if $\beta(v_0,v_1)=w$.
In particular:

\begin{corollary}
    The structures \[\{(\vV,+,0),(\vW,+,0),\beta \colon \vV^2\rightarrow \vW \}\] and
    \[\{(\vV,+,0,R)^{\eq},(\vW,+,0), \pi_A: \vV^{n} \rightarrow B_n,  f_n: \vW \rightarrow B_n; n\in \mathbb{N}, A\in \mathbb{A}_n(\Fbb_p) \} \]
    are bi-interpretable. 
\end{corollary}

\begin{proposition}\label{prop: bilinearsystemEQR}
    An alternating bilinear system 
   \[
        \{(\vV,+,0,R)^{\eq},(\vW,+,0), \pi_A , f_n; n\in \mathbb{N},  A\in \mathbb{A}_n\}
    \]
    satisfying Property $(*_f)$ of \Cref{cor:SepVectorsupspace} eliminates quantifiers relative to $(\vV,+,0,R)^{\eq}$.
\end{proposition}

In particular, in such bilinear system $\{(\vV,+,0,R),(\vW,+,0),\beta\}$, every formula $\phi(w_1,\dots,w_n;v_1,\dots,v_n)$,
where $n\in \mathbb{N}$, $v_1,\dots,v_n \in \vV$ and $w_1,\dots,w_n \in \vW$, is equivalent to a Boolean combination of formulas of the form:
\begin{itemize}
    \item $\sum a_i w_i=0$ where $a_i's$ are in $\mathbb{F}_p$,
    \item $\phi_\vV( f_{n}(\sum a_{1,i} w_i),\dots,f_{n}(\sum a_{k,i} w_i), \bar{v})$ where $\phi_{\vV}$ is a formula in the language of $(\vV,+,0,R)^{\eq}$, and $a_{j,i}$'s are in $\mathbb{F}_p$.
\end{itemize}
In particular, $\vV$ is a stably embedded sort and the induced structure on $\vV$ is given by $(\vV,+,0,R)$.

\begin{remark}
    Let us briefly argue why the non-elimination of finite imaginaries prevents us from exhibiting a simpler language for quantifier elimination: Consider, for instance, an element of $\vW$ of the form:
    \[
        \beta(a,b+c)+ \beta(d,c) = \beta(a,b)+\beta(a,c)+ \beta(d,c) =\beta(a,b)+ \beta(d+a,c),
    \]
    where $a,b,c,d \in \vV$ form a separated basis. At this level of generality, none of these three forms can be favoured over the others, and thus we do not try to distinguish between all the possibilities.
\end{remark}

\begin{remark}
    We can deduce the weaker statement that $\{(\vV,+,0),(\vW,+,0),\beta\}$ eliminates $\vW$-quantifiers. For instance, a formula of the form:
    \[
        \phi_\vV( f_{n}(w), \bar{v})
    \]
    where $\phi_\vV$ is a $(\vV,+,0,R)^{\eq}$-formula can be rewritten:
    \[
        \exists v_1',\dots,v_n'\in V^m\left( \phi_\vV'( v_1',\dots,v_n',\bar{v})\wedge w= \sum a_{i,j} \beta(v_i',v_j')\right) 
    \]
    for some appropriate integer $m$, and an appropriate $\{\vV,+,0,R\}$-formula $\phi_\vV'$. The predicate $R$ can then be eliminated using the function $\beta$.
\end{remark}

\begin{proof}[Proof of~\Cref{prop: bilinearsystemEQR}]
    Let
    \[
        \mathcal{M}\coloneq \{(\vV,+,0,R)^{\eq},(\vW,+,0),\beta\}
    \]
    and
    \[
        \mathcal{N}\coloneq \{(\vV',+,0,R)^{\eq},(\vW',+,0),\beta\}
    \]
    be two alternating bilinear systems satisfying Property $(*_f)$ of \Cref{cor:SepVectorsupspace}. Let $\sigma= (\sigma_\vV,\sigma_\vW): \mathcal{A} \rightarrow \mathcal{B}$ be an isomorphism between two substructures $\mathcal{A}=(A_\vV,A_\vW)$ and $\mathcal{B}=(B_\vV,B_\vW)$ of $\mathcal{M}$ and $\mathcal{N}$, respectively. Suppose furthermore that $\sigma_\vV: A_\vV \rightarrow B_\vV$ is elementary. It suffices to show that we can extend $\sigma$ to an embedding $\Tilde{\sigma}$ of $\mathcal{M}$ into $\mathcal{N}$. By elementarity, we can extend $\sigma_\vV$ to $\Tilde{\sigma}_\vV: \vV^{\eq} \rightarrow {\vV'}^{\eq}$.
    
    By identifying $\vW_n$ with $B_n$, we may assume that $\vW_n\subseteq A_\vW$ for every $n$. It remains to show that the map $\sigma$ can be extended to the elements $w$ with non-finite order, i.e. such that $f_A(w)=\und$ for every $n$ and $A\in \mathbb{A}_n(\Fbb_p)$. Take any vector space $\vW_\omega$ such that $A_\vW\oplus \vW_\omega= \vW$, and let $(w_i)_{i\in I}$ be a basis of $\vW_\omega$. By saturation, we can find $\vert I\vert$-many vectors $(w_i')_{i\in I}$ of $\vW'$ such that $(w_i')_{i\in I}$ are linearly independent over $B_\vW\cup \bigcup_n \vW_n' $. Then it follows that mapping $w_i\mapsto w_i'$ gives an extension of the embedding $\sigma: \vW\rightarrow \vW'$.
\end{proof}


    

\subsection{Third reduction: From \texorpdfstring{$(\vV,+,0,R)$}{(V,+,0,R)} to \texorpdfstring{$\gG$}{C}}

We go back to Mekler groups and multiplicative notation. Let $\gG=(V,E)$ be a nice graph, and $G=\Msf(\gG)$ be its Mekler group. Let $\Zsf$ denote the centre of $G$.    By \Cref{cor:SepVectorsupspace} and \Cref{prop: bilinearsystemEQR}, we now know the induced structure on $G/\Zsf$. It is given by 
\[
    (G/\Zsf,\cdot, 1, R),
\]
where 
\[
    R\coloneq \{(a+\Zsf,b+\Zsf) \ \vert \ [a,b]=1\}.
\]

Notice that the equivalence relations $a\sim b$ and $a\approx b$ pass to the quotient modulo $\Zsf$, and abusing notation, we will write \begin{itemize}
    \item $(a\bmod \Zsf) \approx (b\bmod \Zsf)$ if $\bigvee_{0<n<p}b^n\bmod \Zsf=a \bmod \Zsf$; and
    \item $(a\bmod \Zsf) \sim (b\bmod \Zsf)$ if $\forall x \ xRa \leftrightarrow xRb$.
\end{itemize} 
Similarly, abusing terminology, we will also say that an element $a\bmod \Zsf \in G/\Zsf$ is of type $1^\nu$ (resp. $1^\iota$ or $p$) if $a$ is. The subset of elements of type $1^\nu$ is of course still definable and given by the following formula $\varphi(x)$:
\[
x\notin \Zsf \wedge \forall y \, (x \sim y \leftrightarrow x \approx y) \wedge (\exists z \, [\, z \notin \Zsf \wedge z R x \wedge \neg (z \approx x) \, ])
\]
and we recover the nice graph $\gG=(V,E)$ the same way by considering the quotient $\phi(G/\Zsf)/\sim $ where the edge relation $E$ is induced by the predicate $R$ on $G/\Zsf$. We want to analyse the definable structure in $(\vV,+,0,R)$ and reduce it to $\gG=(V,E)$.

Recall from \Cref{Fact:Hod}(5) that an element $a$ of $G/\Zsf$ can be written $a=a_0\cdots a_{n'-1}$ where $n'\in \mathbb{N}$ and $a_0,\dots,a_{n'-1}$ are distinct elements of type $1^\nu$, and this presentation is unique once we fix an order on $\{ [a]_\sim: a\in G \text{ of type }1^\nu    \}$. We call the set $\{[a_0]_\sim,\dots, [a_{n'-1}]_\sim \}$  the \emph{support} of $a$. If $\gG$ is infinite, the support is not interpretable, in the sense that there is no formula $\psi(x,y)$ that holds precisely when $x$ and $y$ have the same support. However, it becomes definable (with quantifiers) when we bound the size of the support. Therefore, the function $a \mapsto \{[a_0]_\sim, \dots,  [a_{n'-1}]_\sim\}$ must be added to the language in order to obtain quantifier elimination. 
 This is however not sufficient. The reason is that, for a general description of elements in an elementary extension, we must use a transcendental basis. We introduce therefore in the next definition a language strong enough to reflect this. 
 
If $g\in G/\Zsf$ is an element of type $p$, recall that there is a unique class $[a]_\sim$ of an element of type $1^\nu$, such that $gRa$. This class $[a]_\sim$ is called \emph{handle} of $g$, denoted by $h(g)$. 

\begin{definition}\label{def:Anm}
    We denote by $A_{n,m}$ the set of elements $a$ consisting of a product at most $n$ elements $a_0,\dots, a_{n'-1}$ of type $1^\nu$ and of at most $m$ elements $g_0,\dots,g_{m'-1}$ of type $p$ such that:
    \begin{enumerate}[label=(\roman*)]
        \item\label{item:Anm1} $g_0,\dots,g_{m'-1}$ are not a product of less than $n+m+1$ elements of type $1^\nu$,
        \item\label{item:Anm2} $[a_0]_\sim, \dots,  [a_{n'-1}]_\sim , h(g_0),\dots,h(g_{m'-1})$ are pairwise distinct.
        \item\label{item:Anm3} $[a_0]_\sim, \dots,  [a_{n'-1}]_\sim$ are not connected (according to $E$) to any of $h(g_0),\dots,h(g_{m'-1})$.

    Note that these conditions are first-order. 
    \end{enumerate} We denote:
    \[
    \begin{array}{rccl}
         S_{n,m}\colon & A_{n,m} &\longrightarrow &\mathcal{P}_{\leqslant n}(\gG)   \\
                  & a       & \longmapsto    &  \{[a_0]_\sim, \dots,  [a_{n'-1}]_\sim\}, 
    \end{array}
    \]
    and:
    \[ 
    \begin{array}{rccl}
        S_{n,m}': & A_{n,m} &\longrightarrow &\mathcal{P}_{\leqslant m}(C)  \\
                  & a       & \longmapsto    &  \{ h(g_0),\dots,h(g_{m'-1})\},
                  
    \end{array}
    \]
    where $\mathcal{P}_{\leq n}(\gG) \coloneq \bigcup_{k\leq n}\mathcal{P}_k(\gG)$ denotes the set of at-most-$n$-elements subsets of $\gG$ \footnote{We will view $\mathcal{P}_{\leqslant n}(\gG)$ as an imaginary sort of $\gG$, or if $\gG$ has elimination of finite imaginaries, as a definable subset of $\gG^n$.}    
    We set $S_{n,m}(a)=S_{n,m}'(a)= \und $ if $a\notin A_{n,m}$.
\end{definition}
Observe that the sets $A_{n,m}$ don't form a partition: for $(n,m),(k,l) \in \Nbb^2$, $A_{n,m}\cap A_{k,l}
$ is not necessarily empty (for example, $A_{n,0}\subseteq A_{n+1,0}$), and in general, $G \subsetneq \bigcup_{n,m} A_{n,m}$. However, for any $a$ there is a pair $(n,m)\in\mathbb{N}^2$ which is minimal for the reverse lexicographic order, such that $A_{n,m}(a)$.

 For such a minimal pair $(n,m)$, \Cref{item:Anm2,item:Anm3} will be easy to verify and \Cref{item:Anm1} will hold automatically.
 
\begin{remark}
    Let $a,b$ be elements of a Mekler group $M$. If $(n,m)$ and $(k,l)$ are minimal such that $A_{n,m}(a)$ and $A_{k,l}(b)$ hold, and $a=a_0\cdots a_{n-1}\cdot g_0 \cdots g_{m-1}$ is the corresponding decomposition, then
    \begin{itemize}
        \item  $g_0,\dots,g_{m-1}$ are not a (finite) product of elements of type $1^\nu$, 
        \item $A_{k+n,l+m}(a\cdot b)$ holds.
    \end{itemize}
    Notice however that $(k+n,l+m)$ may not be minimal such that $A_{k+n,l+m}(a\cdot b)$ holds. The more complicated cases in the proofs below arise when it is indeed not the case.
\end{remark}

\begin{proposition}    
    The functions $S_{n,m}$ and $S_{n,m}'$ are well-defined and $\emptyset$-definable in Mekler groups.  We refer to them as the \emph{support functions}.
\end{proposition}

\begin{proof}
    Consider $G\coloneq \Msf(\gG)$ the Mekler group associated with a nice graph $\gG$ and let $a\in G$. Suppose that we are given two decompositions:
    \[
    a = a_0\cdots a_{n-1} \cdot g_0 \cdots g_{m-1} = a_0'\cdots a_{n-1}' \cdot g_0' \cdots g_{m-1}'
    \] 
    which satisfy \labelcref{item:Anm1,item:Anm2,item:Anm3}, from  \Cref{def:Anm}. We must show $\{[a_0]_\sim, \dots,  [a_{n-1}]_\sim\}= \{[a_0']_\sim, \dots,  [a_{n-1}']_\sim\}$ and $\{ h(g_0),\dots,h(g_{m-1})\} = \{ h(g_0'),\dots,h(g_{m-1}')\}$. We start with the latter.

    Suppose that $h(g_0)\notin \left\{ h(g_0'),\dots,h(g_{m-1}')\right\}$.  Since $\gG$ is triangle-free and square-free, for all $i$'s, $g_0$ and $g_i'$ are products of elements of type $1^\nu$ and at most one of these elements is in both products.  It follows that $g_0$ is composed with at most $n+m$ elements of type $1^\nu$, which contradicts \labelcref{item:Anm1}. Therefore, we have that $h(g_0)\in \{ h(g_0'),\dots,h(g_{m-1}')\}$ and, arguing in the same manner, it follows that  $\left\{ h(g_0),\dots,h(g_{m-1})\right\} = \{ h(g_0'),\dots,h(g_{m-1}')\}$.
    
    Now, we have by assumption that $[a_0]_\sim$ is not equal to and does not commute with any element of $\{ h(g_0'),\dots,h(g_{m-1}')\}$. Therefore, $[a_0]_\sim$ can't be an element composing any of the $g_{i}'$ and we must have  $[a_0]_\sim=[a_i']_\sim $ for some $i<n$. More generally,  we have $\{[a_0]_\sim, \dots,  [a_{n-1}]_\sim\}= \{[a_0']_\sim, \dots,  [a_{n-1}']_\sim\}$, as wanted.

\end{proof}

\begin{remark}
    Let $a\in G$ satisfying $A_{n,m}$. A decomposition $ a=a_0\cdots a_{n'-1}\cdot g_0 \cdots g_{m'-1}$ such that 
    $S_{n,m}(a)=\{[a_0]_\sim,\dots,[a_{n'-1}]_\sim\}$ and $S_{n,m}'(a)=\{h(g_0),\dots, h(g_{m'-1})\}$
    is \emph{almost} uniquely determined. If $a=a_0'\cdots a_{n''-1}'\cdot g_0' \cdots g_{m''-1}'$ is another such decomposition, then 
    \begin{itemize}
        \item $\{a_0,\dots,a_{n'-1}\}=\{a_0',\dots,a_{n''-1}'\}$,
        \item $\{g_0 \pmod {V_c},\dots,g_{m'-1} \pmod {V_c}\}=\{g_0'\pmod {V_c},\dots,g_{m''-1}'\pmod {V_c}\}$
    \end{itemize}
    where $V_c \coloneq \braket{a_{i,j} : i<j<m}$ is the subgroup generated by the set of connections, i.e. the elements $a_{i,j}$ of type $1^\nu$  such that $[a_{i,j}]_\sim$ is connected to both $h(g_i)$ and $h(g_j)$ (as the graph $\gG$ is square free, there are $(p-1)m(m-1)/2$ of them).  The connection $a_{i,j}$ can therefore be ``included'' in the element of type $p$ and handle $h(g_i)$ or in the one of
    handles $h(g_j)$.

    In particular, $n'=n''$ and $m'=m''$, that is, the decompositions have the same length.
\end{remark}

The following proposition lies at the heart of our relative quantifier elimination result. We postpone the proof to the end of the section, in order to first give more context, as well as some essential lemmas. 

\begin{proposition}\label{prop:RQEGraph}
    The structure 
    \[
    \begin{aligned}
        \big\{ \left(G/\Zsf,\cdot, ^{-1}, (A_{n,m})_{n,m}\right)&, \gG=(V, E)^{\eq},\\ &S_{n,m}\colon  A_{n,m} \rightarrow \mathcal{P}_{\leqslant n}(\gG), S_{n,m}': A_{n,m} \rightarrow \mathcal{P}_{\leqslant m}(\gG) \big\}
    \end{aligned}
    \]
    eliminates quantifiers relative to the graph $\gG$. 
    
\end{proposition}

In particular, every formula $\phi(x)$ in the structure $(G/\Zsf,\cdot,1,R)$ is equivalent to a Boolean combination of formulas of the form:
\begin{itemize}
    \item $A_{n,m}(t(x))$, where $n,m\in \mathbb{N}$, $t(x)$ is a group term,
    \item $\phi_\gG(S_{n,m}(t(x)),S_{n,m}'(t(x)))$, where $n,m\in \mathbb{N}$, $t(x)$ is a tuple of group terms and $\phi_\gG(x_\gG,x_\gG')$ is a formula in $\gG=(V,R)$.
\end{itemize}

To illustrate this result, we give below `main-sorted-quantifier-free' definition of important sets (proofs are optional and left to the reader).
\begin{itemize}
    \item $A_{0,0}=\{1\}$, $A_{0,1}\setminus A_{0,0} $ is exactly the set of elements of type $p$, and $A_{1,0}\setminus A_{0,0}$ the set of element of type $1^\nu$. 
    \item For any element of $a$ type $1^\nu$ and element $g$ of type $p$, we have:
    \begin{center}
        $S_{1,0}(a)=\{[a]_\sim \}$ and $S_{0,1}(g)=\{h(g)\}$.
    \end{center}
    \item Elements of type $p-1$ are exactly the elements $x$ with support consisting of two commuting elements: 
    \[
        A_{2,0}(x) \wedge \exists \alpha,\beta\in S_{2,0}(x)  \ \alpha E \beta,
    \]
    \item Elements of type $1^\iota$ are all the other elements: 
    \[
        \neg A_{0,1}(x)\wedge  (\neg A_{2,0}(x)  \vee (A_{2,0}(x) \wedge \exists \alpha,\beta\in S_{2,0}(x)  \ \neg \alpha E \beta \wedge \alpha \neq \beta ).
    \]
    Recall that $A_{0,0},A_{1,0}\subseteq A_{2,0}$.
    \item For $a,b\in G/\Zsf$, the relation $aRb$ induced by commutation on $G/\Zsf$ is given by 
    \[
    \begin{aligned}
        (a\in A_{1,0} ~\wedge & ~b\in A_{0,1} \wedge S_{0,1}'(b)=S_{1,0}(a) ) \\
            &\vee (a\in A_{0,1}\wedge b\in A_{1,0} \wedge S_{0,1}'(a)=S_{1,0}(b) ) \\
            &\vee (a,b\in A_{2,0} \wedge \forall \alpha \in S_{2,0}(a) ~\forall \beta \in S_{2,0}(b) \ \alpha E \beta ).
    \end{aligned}
    \]
\end{itemize}

The proof, that we will detail below, uses the criteria of Shoenfield, consisting of extending a partial isomorphism $f: \sA \subseteq \sM \rightarrow \sB \subseteq \sN$ between two models with saturation. We will do it in five distinct steps:
    \begin{itemize}
        \item[Step 0:] We enlarge $f$ to all elements of $\gG(\sM)^{\eq}$.
        \item[Step 1:] We enlarge $f$ to elements $a$ of type $1^\nu$ occurring in the decomposition of an element in $G_\sA$.
        \item[Step 2:] We enlarge $f$ to elements $g$ of type $p$ occurring in the decomposition of an element in $G_\sA$.
        \item[Step 3:] We find a transversal $X$ that is ``compatible'' with $\sA$, and enlarge $f$ to this transversal.
        \item[Step 4:] We extend $f$ to a full embedding of the graph $\Gamma(X)$ (which contains $C(\sM)$), then to a full embedding of $\sM$.
        
    \end{itemize}

 We chose the language $\lL$ so that the following improvement of \Cref{Fact:CH} holds:
\begin{lemma}\label{Lem:TransversalQFDefinable}
    Given a small set of variables $Y=Y^\nu \frown Y^p \frown Y^\iota$, the statement: 
    \begin{center}
    ``The elements of $Y^\nu$, $Y^p$ and $Y^\iota$ are respectively of type $1^\nu$, $p$ and $1^\iota$, and the set $Y$ can be extended to a transversal of $G$'' 
    \end{center} 
    is a \emph{quantifier-free} type in the variables $Y$ in the language $\lL$.
\end{lemma}
\begin{proof}
    The quantifier-free type 
    \[
        \left\{\neg A_{n,0}\left(\prod v^{k_i}_i\right) \ \middle\vert \ (k_1,\dots,k_l)\in \{0,1,\dots,p-1\}\setminus \{0,\dots,0\}, n\in\Nbb  \right\}
    \]
    expresses that $v_1,\dots, v_l \in \vV$ are independent modulo $\braket{\tE^\nu}$. 
    
    Similarly, the quantifier-free type 
    \[
        \left\{\neg A_{n,m}\left(\prod w^{k_i}_i\right) \ \middle\vert \ (k_1,\dots,k_l)\in \{0,1,\dots,p-1\}\setminus \{0,\dots,0\}, n,m\in\Nbb \right\}
    \]
    expresses that $w_1,\dots, w_l \in V$ are independent modulo $\braket{\tE^\nu,\tE^p}$. The lemma follows immediately. 
\end{proof}

    In the next lemma, we characterise finite subsets $\Delta$ of elements $d$ which share the same element $h_0$ of type $p$.
    Let $d,d'\in S$ and assume that there are pairs of integers $(k,l), (k',l')$, minimal with the reverse lexicographic order, such that $A_{k,\ell}(d)$, $A_{k',\ell'}(d')$ hold. Write accordingly $d= a_0 \cdots a_k \cdot g_0 \cdots g_{l} $ and $d'= b_0 \cdots b_{k'} \cdot h_0 \cdots h_{l'}$.
    Assume that $g_0=h_0$. Then, in the product $d{d'}^{-1}$, the element $g_0$ and $h_0$ cancel out and:
   
    \begin{itemize}
        \item $A_{k+k',\ell+\ell'-2}(d{d'}^{-1})$ holds,
        \item $S_{k+k',\ell+\ell'-2}(dd'^{-1}) \subseteq S_{k,\ell}(d)\cup S_{k',\ell'}(d')$,
        \item $ S'_{k+k',\ell+\ell'-2}(dd'^{-1}) \subseteq S'_{k,\ell}(d)\cup S'_{k',\ell'}(d')\setminus \{[h_0]_\sim\}  $.
    \end{itemize}

    This generalises of course to any finite subset of elements sharing the same element of type $p$ in there decomposition. We need a partial reciprocal of the fact above. It is easier to state if we restrict ourself to the case where only one element of type $p$ ``cancel out´´ in the product $dd´^{-1}$ (i.e. when the inclusion above are strict equalities).
\begin{lemma} \label{lem:vanish}
    Let $\Delta$ be a finite subset of \ $\bigcup_{n,m}A_{n,m}$ and $c\in \gG$. Assume that for all $d,d' \in \Delta$, if $(k,l), (k',l')$ are minimal with the reverse lexicographic order such that $A_{k,\ell}(d)$, $A_{k',\ell'}(d')$ hold,  then 
    \begin{itemize}
        \item $A_{k+k',\ell+\ell'-2}(dd'^{-1})$ holds, 
        \item $S_{k+k',\ell+\ell'-2}(dd'^{-1})=S_{k,\ell}(d)\cup S_{k',\ell'}(d')$ 
        \item $ S'_{k+k',\ell+\ell'-2}(dd'^{-1})= S'_{k,\ell}(d)\cup S'_{k',\ell'}(d')\setminus \{c\}   $. 
    \end{itemize}
    Then, there is an element $h_0$ of type $p$ and with handle $c$ such that for all $d\in \Delta$, and for $(k,\ell)$ minimal such that $A_{k,\ell}(d)$ holds, we have that $A_{k,\ell-1}(dh_0)$ holds, $S_{k,\ell-1}'(dh_0)=S_{k}'(d)$ and $S_{t,\ell-1}'(dh_0)=  S_{t,\ell-1}'(d) \setminus \{c\}$.
\end{lemma}

The proof of this lemma is left to the reader.

\begin{definition}\label{def:cover}
    Let $\gG$ be an infinite nice graph. A \emph{cover} of $\gG$ is a graph $(\Gamma,E)$ containing $\gG$ as a subgraph such that for every $b\in \Gamma  \setminus \gG$, one of the following two statements holds:
    \begin{itemize}
        \item There is a unique vertex $a$ in $\Gamma$ connected to $b$, and moreover $a$ is in $\gG$ and connected to infinitely many elements in $\gG$.
        \item $b$ is an isolated vertex in $\Gamma$.
    \end{itemize}
    A cover of $\gG$ is called \emph{infinite} if moreover: 
    \begin{itemize}
        \item Every vertex $a$ in $\gG$ with infinitely many neighbours in $\gG$ has infinitely many neighbours in $\Gamma\setminus \gG$.
        \item There are infinitely many isolated vertices in $\Gamma \setminus \gG$.
    \end{itemize} 
\end{definition}

\begin{lemma}\label{lem:RQECover}
    Let $(\Gamma,E)$ be an infinite cover of a nice graph $\gG$. Then the structure $(\Gamma,\gG,E)$, with a predicate for the subgraph $\gG$ eliminates quantifiers relative to $(\gG,E)$. 
\end{lemma}
\begin{proof}
    Let $\sM=(\Gamma(\sM),\gG(\sM))$ and $\sN=(\Gamma(\sN),\gG(\sN))$ be two models, and $\sA,\sB$ be substructures of $\sM$ and $\sN$, respectively. Assume $\sN$ is $\vert \sM \vert$-saturated. Consider a partial isomorphism $f:\sA \rightarrow \sB$ such that  $\restriction{f}{\gG(\sM)}$ is elementary in $\gG$. 
    
    Then, $\restriction{f}{\gG(\sM)}$ can be extended to a full embedding $\tilde{f}_{\gG(\sM)}:\gG(\sM)\rightarrow \gG(\sN)$ and $\Tilde{f}_{\gG(\sM)}\cup f$ is a partial isomorphism. We reset the notation by setting $f=\Tilde{f}_{\gG(\sM)}\cup f$. It remains to extend $f$ to a point $b$ in $\Gamma(\sM)\setminus \sA$. If $b$ is connected to a unique $a\in \gG(\sM)$, then $f(a)$ has infinitely many neighbours. By saturation, there is an element $b'\in \Gamma(\sN)\setminus \sB$ joined with $f(a)$.
    If $b$ is isolated, then by saturation we can find an isolated element $b'$ in $\Gamma(\sM)\setminus \sB$. In any case, we see that $\Tilde{f}=f\cup (b,b')$ extends $f$ at $b$, as wanted.
\end{proof}

Recall that given a group $(G,\cdot,1)$ and $X$ a subset of $G$, we denote by $(\Gamma(X),E)$ or simply $\Gamma(X)$ the graph with vertices the classes of elements in $X$ modulo $\sim$ and edge relation $E$ given by commutation: 
\[
    [a]_\sim E [b]_\sim \iff ab=ba.
\] 
Recall also that, in a Mekler group $\mathsf{M}$, if $X^\nu$ is a choice of representatives of elements of type $1^\nu$ modulo $\sim$, then $\Gamma(X^\nu)= C$ is the nice graph associated to $M$.
The following lemma is almost immediate:
\begin{lemma}\label{lem:infinitecover}
    Let $\Msf=(G,\cdot,1)$ be the Mekler group of an infinite nice graph $\gG$, and $X$ be a transversal of $\Msf$. Then $\Gamma(X)$ is a cover of $\gG$. Moreover, if $\Msf/\Zsf$ is $\aleph_0$-saturated,  then $\Gamma(X)$ is an infinite cover of $\gG$.
\end{lemma}
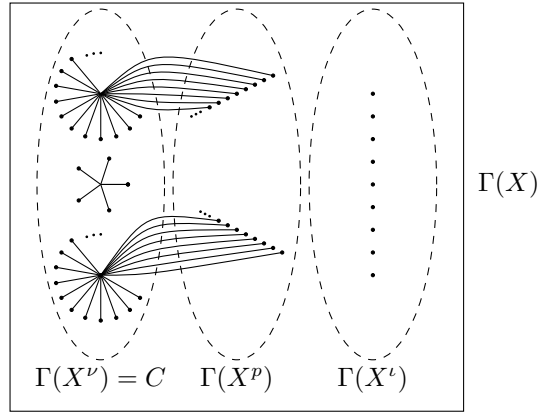
\begin{center}
    
\begin{tikzpicture}[scale=0.6]
    \filldraw (0,0) circle (1pt);
    \draw (-2,-3) rectangle (8,6);
    \draw (9,2) node{$\Gamma(X)$};
    \draw[dashed] (0,2) ellipse (40pt and 110pt);
    \draw (0,-2.3) node{$\Gamma(X^\nu) = C$}; 
    \foreach \i in {0,...,10}{
        \filldraw (130+20*\i:1) circle (1pt);
        \draw (130+20*\i:1)--(0:0);
    }
\foreach \i in {0,...,2}{
        \filldraw (90+10*\i:0.9) circle (0.5pt);
}
    
    \foreach \i in {0,...,10}{
        \filldraw (0,4)++(130+20*\i:1) circle (1pt);
        \draw (0,4)++(130+20*\i:1)--(0,4);
    }
    \foreach \i in {0,...,5}{
        \filldraw (0,2)++(72*\i:0.6) circle (1pt);
        \draw (0,2)++(72*\i:0.6)--(0,2);
    }
    \foreach \i in {0,...,2}{
        \filldraw (0,4)++(90+10*\i:0.9) circle (0.5pt);
}
    \draw[dashed] (3,2) ellipse (40pt and 110pt);
    \draw (3,-2.3) node{$\Gamma(X^p)$}; 
    \foreach \i in {-2,...,5}{
        \filldraw (3,1)++(0.2*\i,-0.1*\i) circle (1pt);
        \draw (3,1)++(0.2*\i,-0.1*\i) .. controls (1,1-0.2*\i) .. (0,0);
    }
    \foreach \i in {-4,-3.5,-3}{
        \filldraw (3,1)++(0.2*\i,-0.1*\i) circle (0.5pt);
    }
    
    \foreach \i in {-3,...,4}{
        \filldraw (3,4)++(0.2*\i,0.1*\i) circle (1pt);
        \draw (3,4)++(0.2*\i,0.1*\i) .. controls (1.1,4.2+0.2*\i) .. (0,4);
    }
    
    \foreach \i in {-5,-4.5,-4}{
        \filldraw (3,4)++(0.2*\i,0.1*\i) circle (0.5pt);
        
    }
    \draw[dashed] (6,2) ellipse (40pt and 110pt);
    \draw (6,-2.3) node{$\Gamma(X^\iota)$};
    \foreach \i in {-4,...,4}{
        \filldraw (6,2)++(0,0.5*\i) circle (1pt);
        }
\end{tikzpicture}

\end{center}
\begin{proof}
    The fact that $\Gamma(X)$ is a cover of $\gG$ is a simple exercise (see \cite[Fact 2.10]{CH18}). Assume that $\Msf/\Zsf$ is $\aleph_0$-saturated. Let $a\in \gG$ be a vertex with an infinite set of neighbours $N_a \subset \gG $. Then ``$ x \text{ is of type $p$ and handle }a$'' is satisfied by products of (at least three) elements of $N_a$. The type $p(x_i : i \in \mathbb{N})$:
    \[
    \bigcup_{i\in\Nbb}\{ x_i \text{ is of type $p$, handle $a$ and linearly independent over all elements of type } 1^\nu\}
    \]
    is consistent.  As $\Msf/\Zsf$ is $\aleph_0$-saturated, $\Msf$ realises this type, and clearly, for any choice of transversal $X= X^\nu \cup X^p \cup X^\iota$, $a$ is connected to infinitely many elements $[g]_\sim$, $g\in X^p$.
    Similarly, if $\gG$ is infinite, there is infinitely many element of type $1^\iota$ and the type $q(x_i : i \in \mathbb{N}) :$
    \[
   \bigcup_{i\in\Nbb}\ \{ x_i \text{ is of type $1^\iota$, and linearly independent over all elements of type } 1^\nu \text{ and of type }p  \}
    \]
    is consistent. Again, by $\aleph_0$-saturation, there will be infinitely many isolated point in $\Gamma (X)\setminus \gG$ for any choice of transversal.
\end{proof}

We are now ready to prove \Cref{prop:RQEGraph}.

\begin{proof}[Proof of \Cref{prop:RQEGraph}]
    We denote, only in this proof, the group sort by $G$, instead of $G/\Zsf$, (in particular, we use a multiplicative notation, despite the group being abelian). 
    In the proof below, we follow the 5-step strategy outlined earlier.
    
    Let $\sA\subseteq \sM, \sB \subseteq \sN$ and $f: \sA \hookrightarrow \sB$ be a partial isomorphism, where $\sM$ is an $\aleph_0$-saturated and $\sN$ an $\vert \sM \vert$-saturated Mekler group of exponent $p$. The embedding $f$ consists of two (compatible) functions $f_G: G_{\sA} \rightarrow G_{\sB}$, $f_\gG:  \gG_{\sA} \rightarrow \gG_{\sB}$, for the sorts $G$ and $\gG$, respectively. Assume that $f_\gG:  \gG_{\sA} \hookrightarrow \gG_{\sB}$ is elementary. Notice that $G_{\sA}$ is already a subgroup since the structure is closed under multiplication and the groups are of exponent $p$. We will show that $f$ can be extended to a full embedding of $\sM$ in $\sN$.

    \textbf{Step 0:} By elementarity, we can extend $f_\gG:  \gG_{\sA} \hookrightarrow \gG_{\sB}$ to $\gG(\sM)^{\eq} \hookrightarrow \gG(\sN)^{\eq}$, and $f\cup f_\gG$ is still a partial isomorphism, as $\gG$ is a closed sort, and

    \[
        \bigcup_{n,m} S_{n,m}(G_\sA)\cup S_{n,m}'(G_\sA) \subseteq  \gG_{\sA}.
    \]
    For simplicity, we update the notation and let $f$ denote the extension. \\

    Let $g\in G_{\Acal}$. Let $(n,m)$ be minimal for the reverse lexicographic order such that $g\in A_{n,m}$ and write $g=a_0 \cdots a_{
    n-1}\cdot g_0 \cdots g_{m-1}$ as a product of $n$ elements of type $1^\nu$ and $m$ elements of type $p$. \\

    \textbf{Step 1:} We reduce to the case where $a_0,\dots,a_{n-1} \in G_{\sA}$.

     Assume, for example, that $a_0 \notin G_{\sA}$. In this case, since $f$ preserves the function $S_{n,m}$, $f_\gG ([a_0]_\sim) \in A_{n,m}(f_G(g))$, and this means that we can find $b_0\in G(\mathcal{N})$ of type $1^\nu$ such that:
     \begin{itemize}
         \item $[b_0]_\sim=f_\gG ([a_0]_\sim)$
         \item $[b_0]_\sim \notin S_{n,m}(f_G(g)\cdot b_0^{-1})$.
     \end{itemize} 
     Then we can extend $f_{G}$ by setting $f_{G}(a_0)=b_0$. We show that this function preserves quantifier-free formulas. To see this, fix $g'\in G_{\sA}$ and $0<r<p$.  
     We should show that, for every integer $s,t$:
     \begin{enumerate}[label=(\roman*)]
        \item\label{item:qf1} $A_{s,t}(g'a_0^r)$ holds if and only if $A_{s,t}(f(g')b_0^r)$ holds,
        
        \item\label{item:qf2}  $f(S_{s,t}(g'a_0^r))=S_{s,t}(f(g')b_0^r)$  if the first point holds,
        \item  \label{item:qf3}$f(S_{s,t}'(g'a_0^r))=S_{s,t}'(f(g')b_0^r)$  if the first point holds.
     \end{enumerate}
     
    Indeed it in particular shows that $f_G \cup \{(a_0,b_0)\}$ preserves equations of the form $g_0 \cdots g_{n-1} a_0^r= 1$ where $r<p$ (since such an equation holds if and only if $A_{0,0}(g_0 \cdots g_{n-1} a_0^r)$ holds). It will follows that for a quantifier-free formula $\phi(x,y,c)$ with parameters $c$ in $\gG(\mathcal{M})$ and $g'$ a tuple from $G_A$: \[ \sM\vDash\phi(g',a_0,c)\text{ if and only if }\Ncal\vDash\phi(f_{G}(g'),f_{G}(a_0),f_\gG(c)).\]

    We start with a trivial case:

    \begin{itemize}
        
        \item \emph{Case 0}: $g'\notin \bigcup_{k,l} A_{k,l}$. Then for all integers $s,t$, $A_{s,t}(g'a_0^r)$ and $A_{s,t}(f(g')b_0^r)$ don't hold and $S_{s,t}(g'a_0^r)=\und $, $S_{s,t}(f(g')b_0^r)=\und$.
    \end{itemize}
    
    Therefore, we may assume that $g'\in  \bigcup_{k,l} A_{k,l}$, and let $(k,\ell)$ be minimal (in the reverse lexicographic order) such that $g'\in A_{k,\ell}$. We should first find the smallest $(k',\ell')$ such that $g'a_0^r\in A_{k',\ell'}$ and then express $S_{k',\ell'}(g'a_0^r)$ and $S_{k',\ell'}'(g'a_0^r)$ in terms of $g,g'$ and $[a_0]_\sim$.
     We distinguish cases:
    \begin{itemize}
        
        \item \emph{Case 1}: $[a_0]_\sim \notin S_{k,\ell} (g')$ and $[a_0]_\sim \notin S_{k,\ell}'(g')$. 
        
        Then $(k',l')=(k+1,l)$ and 
        $S_{k+1,\ell}(g'a_0^r)= S_{k,\ell}(g')\cup \{[a_0]_\sim\}$ and $S_{k+1,\ell}'(g'a_0^r)= S_{k,\ell}'(g')$.
        \item \emph{Case 2}: $[a_0]_\sim \notin S_{k,\ell} (g')$ and $[a_0]_\sim E \alpha$ or $[a_0]_\sim = \alpha$ for some element $\alpha \in S_{k,\ell}' (g')$. 
        
        Then $(k',l')=(k,l)$ and 
        $S_{k,\ell}(g'a_0^r)= S_{k,\ell}(g')$ and $S_{k,\ell}'(g'a_0^r)= S_{k,\ell}'(g')$.
        \item \emph{Case 3}: $[a_0]_\sim \in S_{k,\ell} (g')$ (and, by minimality of $(k,\ell)$, $\neg [a_0]_\sim E \alpha$ for all elements $\alpha \in S_{k,\ell}' (g')$). 
        
        We now need to consider two further subcases:
        \begin{itemize}
            \item \emph{Case 3(A)}: $A_{k+n,\ell+m}(g'g^r)$ and $[a_0]_\sim \notin S_{k+n,\ell+n}(g'g^r)$. This is the subcase where the powers of $a_0$ in $g'$ and in $g^r$ cancel out in the product $g'g^r$. 
            
            Then $(k',l')=(k-1,l)$ and 
            \[S_{k-1,\ell}(g'a_0^r) =S_{k,\ell}(g')\setminus \{[a_0]_{\sim}\},\] and \[S_{k-1,\ell}'(g'a_0^r) =S_{k,\ell}'(g').\]
            \item \emph{Case 3(B)}: $A_{k+n,\ell+m}(g'g^r)$ and $[a_0]_\sim \in S_{k+n,\ell+m}(g'g^r)$.  In this subcase, the powers of $a_0$ in $g'$ and in $g^r$ don't cancel.
            
            Then we have $A_{k,\ell}(g'a_0^r)$ with $(k',l')=(k,l)$ minimal,
            \[S_{k,\ell}(g'a_0^r) =S_{k,\ell}(g')\] and
            \[S_{k,\ell}'(g'a_0^r) =S_{k,\ell}'(g').\]
        \end{itemize}
    \end{itemize}
    The exact same statement holds in $\sN$ if we replace $a_0$ with $b_0$, $g'$ with $f_G(g')$ and $g$ with $f_G(g)$.
    Therefore, \Cref{item:qf1,item:qf2,item:qf3} hold for $(k',l')$. It will automatically follows for the other integers $(s,t)$, as $(k',l')$, $S_{k',l'}$ and $S_{k',l'}'$ determine all the other pairs of integers such that $A_{s,t}(g'a_0^r)$ holds and determine $S_{s,t}(g'a_0^r)$ and $S_{s,t}'(g'a_0^r)$.
    
    This argument shows that $f_G \cup \{(a_0,b_0)\}$ is a partial isomorphism. We can reset the notation by setting $f_G=f_G \cup \{(a_0,b_0)\}$, and therefore assume that $a_0\in G_\sA$.

    Repeating the argument above, we may assume that for any $g=a_0\cdots a_{m-1}\cdot g_0\cdots g_{n-1}\in G_{\sA}\cap A_{m,n}$, we have $a_0,\dots,a_{m-1} \in G_{\sA}$. \\
    
    Fix again a $g\in G_{\sA}$ with $(n,m)$ minimal (for the reverse lexicographic order) such that $g=a_0\cdots a_{n-1}\cdot g_0\cdots g_{m-1}$ a product of $n$ elements of types $1^\nu$ and $m$ elements of type $p$. By the previous step, we may assume that $m=0$ and $g=g_0\cdots g_{m-1}$.\\
    
    \textbf{Step 2:}  We reduce to the case where $g_0,\dots, g_{m-1}\in G_{\sA}$. 
    
    This is similar to Step $1$. Assume, for example, that $g_0 \notin G_{\sA}$. By minimality of $(0,m)$, $g_0$ is not a finite product of elements of type $1^\nu$.
    Consider the set $V_g^{h(g_0)}$ of element $g'$ where $h(g_0)$ ``vanishes'' in the product $g'g$ 
    :
    \begin{align*}
        \{g' \in G  \mid & \text{ for }(k,\ell)\in \mathbb{N}^2 \text{ minimal such that } A_{k,\ell}(g'),\ \\
  &  \ A_{k,\ell+m-2}(g'g)\text{ and }h(g_0) \notin S_{k,\ell+m-2}(g'g)=S_{k,\ell}(g')\cup S_{0,m}(g )\setminus \{h(g_0)\} \}.
    \end{align*}
    ($k, \ell$ depends on $g'$, but we hide this dependence for the sake of notation simplicity). Then, consider the partial type $p(x)$:
    \begin{align*}
        \left\{  x\text{ of type p and handle } h(g_0) :  \ \right.& A_{k,\ell-1}(g'x), S_{k,\ell-1}(g'x)=S_{k,\ell}(g')\\
        & \left. \text{ and } S_{k,\ell-1}'(g'x)=S_{k,\ell}'(g')\setminus \{h(g_0)\}  \right\}.
    \end{align*}
        
    Since $f$ is a partial isomorphism, we can consider the image by $f$ of the type $p(x)$:
\begin{align*} f(p(x)) \coloneq \{  x \text{ of type p and handle } f(h(g_0)) :\  &
 A_{k,\ell-1}(f(g')x),  S_{k,\ell-1}(f(g')x) = S_{k,\ell}(f(g'))\\ &
\text{ and }S_{k,\ell-1}'(f(g')x) = S_{k,\ell}'(f(g'))\setminus f(h(g_0))   \}.\end{align*}

    By \Cref{lem:vanish}, one can see that this type is consistent. 
    Since $\mathcal{N}$ is saturated, we may find $h_0$ in $G(\mathcal{N})$ satisfying $f(p(x))$.
    
    
    Then we extend $f_{G}$ by setting $f_{G}(g_0)=h_0$. We show that this function preserves quantifier-free formulas. To show this, we fix $g'\in G_{\sA}$ and $0<r<p$. Let $(k,\ell)$ be minimal for the reverse lexicographic order such that   $g'\in A_{k,\ell}$. As in Step 1, we can start with the trivial case:
    \begin{itemize}
        \item \emph{Case 0}: $g'\notin \bigcup_{k,l} A_{k,l}$. Then for all integers $s,t$, $A_{s,t}(g'g_0^r)$ and $A_{s,t}(f(g')h_0^r)$ don't hold and $S_{s,t}(g'g_0^r)=\und $, $S_{s,t}(f(g')h_0^r)=\und$.
    \end{itemize}
    
    Therefore, we may assume that $g'\in  \bigcup_{k,l} A_{k,l}$, and we need to find the smallest $(k',\ell')$ such that $g'g_0^r\in A_{k',\ell'}$. Again, there are several cases to consider:
    \begin{itemize}
        \item \emph{Case 1}: $\neg h(g_0)E \alpha$ for all $\alpha\in S_{k,\ell} (g')$ and $h(g_0) \notin S_{k,\ell} (g')$.
        
        Then $(k',\ell')=(k,\ell+1)$, $S_{k,\ell+1}(g'g_0^r)= S_{k,\ell}(g')$ and $S_{k,\ell+1}'(g'g_0^r)= S_{k,\ell}'(g')\cup \{h(g_0)\}$.
        \item \emph{Case 2}: For some $t>0$, $h(g_0)$ is related (according to $E$) with exactly $t$-many elements of $ S_{k,\ell} (g')$ (and by minimality of $(k,\ell)$, $h(g_0) \notin S_{k,\ell}'(g')$). 
        
        Then $(k',\ell')=(k-t,\ell+1)$, $S_{k-t,\ell+1}(g'g_0^r)= S_{k,\ell}(g') \setminus \{ \alpha \in S_{k,\ell}(g') \ \vert \ \alpha E h(g_0)\}$, and $S_{k-t,\ell+1}'(g'g_0^r)= S_{k,\ell}'(g') \cup \{h(g_0)\}$.
        \item \emph{Case 3}: $h(g_0) \in S_{k,\ell}' (g')$ (then by minimality of $(k,\ell)$, $\neg h(g_0) E \alpha$ for all $\alpha \in  S_{k,\ell}(g')$). 

        We now need to consider two further subcases:
        \begin{itemize}
            \item \emph{Case 3(A)}: There is an integer $t$ such that that $A_{t,\ell+m-2}(g'g^r)$ and $h(g_0) \notin S_{t,\ell+m-2}'(g'g^r)$. 
            This means that the handle $h(g_0)$ vanishes in the product $g_0^rg'$. At the cost of multiplying $g'$ with elements of type $1^\nu$, we may assume that $g' \in V_{g^r}^{h(g_0)}$. We have then $(k',l')=(k,\ell-1)$, $A_{k,\ell-1}(g'g_0^r)$.
   \[S_{k,\ell-1}'(g'g_0^r) =S_{k,\ell}'(g')\setminus \{h(g_0)
            \}\]
            and 
        \[S_{k,\ell-1}(g'g_0^r) =S_{k,\ell}(g') .\]         
            \item \emph{Case 3(B)}: There is no such integer $t$. 
            
            Then we have $(k',\ell')=(k,\ell)$,
            \[S_{k,\ell}(g'g_0^r) =S_{k,\ell}(g')\]
            and 
            \[S_{k,\ell}'(g'g_0^r) =S_{k,\ell}'(g').\]
        \end{itemize}
    \end{itemize}
    
    By choice of $h_0$,  we have that same statement holds in $\sN$ with $h_0$ instead of $g_0$, $f_G(g')$ instead of $g'$ and $f_G(g)$ instead of $g$. It follows that that $f_G \cup \{(g_0,h_0)\}$  preserves all the predicates $A_{k,l}$ and the functions $S_{k,l},S_{k,l}'$. 
        This argument shows that $f_G \cup \{(g_0,h_0)\}$ is a partial isomorphism. We can reset the notation by setting $f_G=f_G \cup \{(g_0,h_0)\}$, and therefore assume that $h_0\in G_\sA$.

    By repeating Step 2, we may assume that for any $g=a_0\cdots a_{m-1}\cdot g_0\cdots g_{n-1}\in G_{\sA}\cap A_{m,n}$, we have $g_0,\dots,g_{n-1} \in G_{\sA}$. \\
    
    \textbf{Step 3:} We find a transversal $X$ compatible with the substructure $\sA$, i.e. such that $G_{\sA}=\braket{X_{\sA}}$ where $X_{\sA}= X\cap G_{\sA}$.  
    
    We may, indeed, consider first a set $X_{\sA}^\nu$ of representatives in $\sA$ of the $\sim$-classes of elements of type $1^\nu$ in $\sA$. Then, $X_{\sA}^p$ is the set of representative in $A$ of $\sim$-classes of elements of type $p$ independent over $X_{\sA}^\nu$. If a product $g_1\cdots g_n \in \sA$ of elements of type $p$ is a finite product $a_1\cdots a_m \in G$ of elements of type $1^\nu$, these elements are, by the previous steps, in $\sA$, which is a contradiction. Therefore, they must also be independent over all elements of type $1^\nu$ in $G$.  Finally, $X_{\sA}^\iota$ is the set of representatives of elements in $\sA$ of type of $1^\iota$ independent over  $X_{\sA}^\nu,X_{\sA}^p$. By the previous steps, and with a similar argument, they are also independent over all elements of type $p$ and $1^\nu$ in $G$. \\

    

    
    
    \textbf{Step 4:} We extend $f$ to a full embedding of $\Gamma(X)$, then to a full embedding of $\sM$.
    
    By \Cref{Lem:TransversalQFDefinable}, and since $f$ preserves quantifier-free formulas, $f(X_A)$ can also be extended to a transversal $Y$ of $B$.
    Denote by $\Gamma(X)$ the graph with vertices $\{[x]_\sim : x\in X\}$ and edges   $\{([x]_\sim,[y]_\sim) : xRy\in \sM\}$ (this is exactly the graph $\Gamma(\tilde{X})$ where $\tilde{X}$ is a lift of $X$ in the Mekler group).
    
    We therefore have a partial isomorphism $f_\Gamma$ between $\Gamma(X)$ and $\Gamma(Y)$ as graphs. By saturation and \Cref{lem:infinitecover}, these graphs are infinite covers of $\gG(\mathcal{M})$ and $\gG(\mathcal{N})$, respectively. By $\vert \Gamma(X)\vert$-saturation of $\Gamma(Y)$ (\cite[2.11 (3)]{CH18}) and by quantifier elimination relative to $C$, see \Cref{lem:RQECover}) we can extend it to a full embedding $f_\Gamma:\Gamma(X) \hookrightarrow \Gamma(Y)$ .
    
    This embedding gives immediately an extension of $\restriction{f}{\braket{X_A}}$ to \[{f:G(\sM)=\braket{X} \hookrightarrow G(\sN)=\braket{Y}}\] as wanted.
\end{proof}

\section{Transfer principles}\label{sec:transfers}

A \emph{transfer principle for Mekler groups} (or, in our context, simply a \emph{transfer principle}) is any statement characterising a model-theoretic property of the Mekler group $\Msf$ at the level of the graph $\gG$. Quite a lot of transfer principles for Mekler groups are already known. We recall below the current state of affairs. The reader can refer to the cited sources for relevant definitions and proofs:

\begin{theorem}\label{thm:KnownTransfers}
    A Mekler group $\Msf$ has the property $P$ if and only if its associated graph $\gG$ has the property $P$, where $P$ is one of the following properties:
    \begin{itemize}
        \item $\lambda$-stability for every cardinal $\lambda$, \cite{Mek81,Hod93};
        \item CM-triviality \cite{Bau02};
        \item The $n$-independence property, for every $n\in \Nbb$, \cite{CH18};
        \item The tree property of the second kind, \cite{CH18};
        \item The first and second strict order properties, \cite{Ahn20};
        \item The anti-chain tree property, \cite{AKL22}.
    \end{itemize}
\end{theorem}  

In this section, we prove new transfer principles. Our main tool is the relative quantifier elimination result developed in the previous section, and our method is again a three-step reduction. In the first subsection we discuss pairs of Mekler groups. The second subsection is dedicated to dividing lines and includes some negative results. 


\subsection{Model completeness and Stable embeddedness}\label{subsec:SE}
From relative quantifier elimination one can always deduce a relative model completeness result:

\begin{lemma}\label{lem:RMCGeneral}
    Assume that a complete theory $T$ eliminates quantifiers relative to a closed sort $\Sigma$ in a language $\lL$, and let $\sN,\sM\vDash T$ be two structures such that:
    \begin{itemize}
        \item $\sM \subseteq_{\lL} \sN$,
        \item $\Sigma(\sM) \preceq \Sigma(\sN)$.
    \end{itemize}
    Then $\sM \preceq \sN$.
\end{lemma}

However, we can often optimise this result by weakening the first condition, considering only a reduct of the language $\lL$ and adding a more relevant algebraic requirement. This is what we shall do in the next paragraphs: We will deduce model completeness statements from the three relative quantifier elimination results stated in the previous section (\Cref{prop:RMCFG,prop:RMCVR,prop:RMCC} below). We will then combine these results in the last part of this subsection to obtain model completeness for Mekler groups relative to their graphs (\Cref{prop:RMCMekler}). 

We will implicitly use that ``elementary extension'' is a property of a pair of structures, and does not depend on the language in which the structures are taken:

\begin{fact}
    Let $\sM\preceq \sN$ be an elementary extension. Then $\sM^\eq \preceq \sN^{\eq}$.
\end{fact}

At the same time, we will look at stably embedded pairs of Mekler groups. Let us now recall the definition:

\begin{definition}[Stable Embeddedness]
    Let $T$ be a complete theory. An elementary extension $\mathcal{M} \preceq \mathcal{N}$ is called \emph{stably embedded} if for every formula $\phi(x,b)$ with $b\in \sN$, there is a formula $\psi(x,c)$ with $c\in \sM$ such that $\phi(\sM,b)=\psi(\sM,c)$. In this case, we write $\sM \preceq^{\st} \sN$.
    
    The pair is called \emph{uniformly stably embedded} if, in addition, the choice of the formula $\psi(x,z)$ depends only on $\phi(x,y)$, and is independent of the parameters $b$. In this case, we write $\sM \preceq^{\ust} \sN$.
\end{definition}

Stable embeddedness says that every subset of $\sM$ which is \emph{externally} definable in $\sN$ is \emph{internally} definable. The proof of \Cref{thm:RSEMekler} will follow, without much effort, from quantifier elimination, as the characterisation does not require any new assumptions on the pair of Mekler groups. Again, stable embeddedness is a property of a pair of structures and does not depend on the language in which the structures are taken:

\begin{fact}[e.g. \cite{Tou23}]
    Let $\sM\preceq^{\st} \sN$ (resp. $\sM\preceq^{\ust} \sN$) be a (uniformly) stably embedded elementary (resp. \emph{uniformly} stably embedded) pair of structures. Then $\sM^\eq \preceq^{\st} \sN^{\eq}$ (resp. $\sM^\eq \preceq^{\ust} \sN^{\eq}$).
\end{fact}

The above fact justifies our stepwise approach and will be used implicitly.

\subsubsection{First reduction: From \texorpdfstring{$G$}{G} to \texorpdfstring{$\mathcal{F}(G)$}{F(G)}}

\begin{proposition}[Relative Model Completeness]\label{prop:RMCFG}
    Let $G$ and $G'$ be $2$-nilpotent groups of exponent $p$. Assume that $G$ is a subgroup of $G'$ such that $\Zsf(G) \subseteq \Zsf(G')$.
    Then, $G \preceq G'$ if and only if $\mathcal{F}(G) \preceq \mathcal{F}(G')$, as alternating bilinear systems.
\end{proposition}
\begin{proof}
    Since $\Zsf(G) \subseteq \Zsf(G')$, the following diagram commutes:
    \[\begin{tikzcd}
	{\mathsf{Z}(G)} & G & {G/\mathsf{Z}(G)} \\
	{\mathsf{Z}(G')} & {G'} & {G'/\mathsf{Z}(G')}
	\arrow[from=1-1, to=1-2]
	\arrow[hook', from=1-1, to=2-1]
	\arrow["{\rho_\mathsf{Z}}"', shift left, bend right=60, from=1-2, to=1-1]
	\arrow["{\pi_{G/\mathsf{Z}}}", from=1-2, to=1-3]
	\arrow[hook', from=1-2, to=2-2]
	\arrow[dashed, hook', from=1-3, to=2-3]
	\arrow[from=2-1, to=2-2]
	\arrow[from=2-2, to=2-3]
\end{tikzcd}\]
    In particular, we have that $G \subseteq_{\lL} G'$, where $\lL=\mathcal{L}_{G,\Fcal(G)}$ is the language of relative quantifier elimination in \Cref{fact:RQEGtoF(G)}. We may conclude by \Cref{lem:RMCGeneral}.
\end{proof}

\begin{proposition}[Relative Stable Embeddedness]\label{prop:RSEFG}
    Let $G \preceq G'$ be $2$-nilpotent groups of exponent $p$. Then:
    \[
        G \preceq^{\st} G' \ \text{if and only if } \ \mathcal{F}(G)\preceq^{\st} \mathcal{F}(G'), 
    \]
    and: 
    \[
        G \preceq^{\ust} G' \ \text{if and only if } \ \mathcal{F}(G)\preceq^{\ust} \mathcal{F}(G'). 
    \]
\end{proposition}

\begin{proof}
    Assume $\mathcal{F}(G)\preceq^{\st} \mathcal{F}(G') $. Let $b$ be a tuple from $G'$, and $\phi(x,y)$ be an $\lL_\mathsf{grp}$-formula. We want to define $\phi(G,b)$ using parameters in $G$. By \Cref{fact:RQEGtoF(G)}, $\phi(x,b)$ is equivalent to a formula of the form:
    \[
    \phi_{\mathcal{F}(G)}(\pi(t'(x,b)),\rho(t(x,b)))
    \]
    where $\phi_{\mathcal{F}(G)}(y,z,w)$ is a formula in the language $\{(\vV,+,0),(\vW,+,0),\beta\}$, and $t'$ and $t$ are tuples of $\lL_{\mathsf{grp}}$-terms. To simplify the argument, we assume that $t$ consists of only one term. 
    
    Since $\pi$ is a morphism, we have $\pi(t'(x,b))=t'(\pi(x),\pi(b))$ (where the term $t'$ is understood additively in the right-hand side of the equality). 
    
    Therefore, after incorporating $t'$ into $\phi_{\mathcal{F}(G)}$, the formula becomes equivalent to one of the form:
    \[
        \phi_{\mathcal{F}(G)}'(\pi(x),\pi(b),\rho(t(x,b))).
    \]
    Notice, however, that $\rho$ is a morphism only when it is restricted to the centre, $\Zsf$. In particular, if $t(x,b)=z(x,b)$ is a product of commutators, it takes value in $\Zsf$, and we have 
    \[
        \rho(z(x,b))= z_\beta(\pi(x),\pi(b)),
    \]
    where $z_\beta$ is a sum of $\beta$-terms (we can separate the variables $x$ from the parameters $b$ using bilinearity of $\beta$).
    
    In general, $t(x,b)$ need not be a product of commutators, so we have to reduce to that case. If there is no $a_0\in G$ such that $t(a_0,b)$ is in $\Zsf$, then, by definition of $\rho$, we can replace $\rho(t(x,b))$ by $0$ (recall that groups in $\lL_{\Fcal(G)}$ are written additively). So, assume that such an element $a_0$ exists. Since the group $G'$ is $2$-nilpotent, we can rewrite the term $t(x,y)$ as $t_1(x)t_2(y)z(x,y)$ where $t_1(x),t_2(y)$, and $z(x,y)$ are $\lL_\mathsf{grp}$-terms and $z(x,y)$ is a product of commutators (and therefore takes values in $\Zsf$).
    Then, for all $a\in G$, we have that:
    
\begin{align*}
    t(a,b) \in \Zsf & \Leftrightarrow t_1(a)t_2(b)z(a,b) \in \Zsf \\
        & \Leftrightarrow t_1(a)t_1(a_0)^{-1}t_1(a_0)t_2(b)z(a_0,b) \in \Zsf \\
        & \Leftrightarrow t_1(a)t_1(a_0)^{-1}t(a_0,b) \in \Zsf\\
        & \Leftrightarrow t_1(a)t_1(a_0)^{-1} \in \Zsf.
\end{align*}

    We rewrite $t(x,b)$ as follows:
\begin{align*}
    t(x,b) 
    &= t_1(x)t_1(a_0)^{-1}t_1(a_0)t_2(b)z(x,b) \\
    &= t_1(x)t_1(a_0)^{-1}t(a_0,b)z(a_0,b)^{-1}z(x,b).
\end{align*}
    Let $\sigma(x,a_0,b)$ be the following term:
    \[
    \sigma(x,a_0,b) \coloneq \rho\left(t_1(x)t_1(a_0)^{-1}\right) +\rho(t(a_0,b))+\rho\left(z(a_0,b)^{-1}\right)+\rho(z(x,b)).
    \]
    Then our formula is equivalent to:
    \[
    \begin{aligned}
    \left(t_1(x)t_1(a_0)^{-1} \in \Zsf \wedge 
    \phi_{\mathcal{F}(G)}'\left(\pi(x),\pi(b), \sigma(x,a_0,b)\right) \right) \\
   \lor
   \left(  t_1(x)t_1(a_0)^{-1} \notin \Zsf \wedge \phi_{\mathcal{F}(G)}'(\pi(x),\pi(b),0)\right).
    \end{aligned}
    \]
    As before, we can replace the terms $\rho(z(x,b))$ in $\sigma(x,a_0,b)$ by equivalent terms of the form $z_{\beta}(\pi(x),\pi(b))$ which are the sum of $\beta$-terms. Therefore, we get a formula of the form:
    \[
        \phi_{\mathcal{F}(G)}''\left(\pi(x),\rho\left(t_1(x)t_1(a_0)^{-1}\right), \pi(b), \rho\left(t\left(a_0,b\right)\right)\right).
    \]
    Since $\mathcal{F}(G)\preceq^{\st} \mathcal{F}(G') $, there is a formula $\psi_{\mathcal{F}(G)}(x_\mathcal{F},y_\mathcal{F}, a_\mathcal{F})$ where $a_\mathcal{F} \in \mathcal{F}(G)$ such that:
    \[
        \phi_{\mathcal{F}(G)}''\left(\mathcal{F}(G)^{\vert x\vert+1 }, \pi(b), \rho(t(a_0,b))\right)= \psi_{\mathcal{F}(G)}\left(\mathcal{F}(G)^{\vert x\vert + 1 }, a_\mathcal{F}\right).
    \]

    At the end, we have that $\phi(G,b) = \psi(G,a)$ where $a\in G$ is such that $\pi(a)=a_{\mathcal{F}}$, and where $\psi(x,a)$ is the formula:
    \[
        \psi_{\mathcal{F}(G)} \left(\pi(x),\rho\left(t_1(x)t_1(a_0)^{-1}\right), \pi(a)\right) .  
    \]

    We have shown that $G$ is stably embedded in $G'$. The uniform stable embeddedness statement follows the same way once we notice that the formula $\phi_{\mathcal{F}(G)}''$ does not depend on the choice of $b$ -- it depends only on whether or not there is $a_0\in G$ such that $t(a_0,b)$ is in $\Zsf$, and the case distinction can be handled using parameters in $G$.
\end{proof}

\subsubsection{Second reduction: From \texorpdfstring{$\{(\vV,+,0),(\vW,+,0),\beta\}$}{\{(V,+,0),(W,+,0), β\} } to \texorpdfstring{$(\vV,+,0,R)$}{(V,+,0,R)}}

Recall we denote $\cup_n \vW_n \coloneq \braket{\beta(\vV,\vV)}$ the subspace of $\vW$ generated by all terms $\beta(v,v')$, where $v,v'\in \vV$. We can always find a complement $\vW_\omega$, such that $\vW=\braket{\beta(\vV,\vV)} \oplus \vW_\omega$. We recommend the reader remind themself of \Cref{notation:pi-f} before delving into the proofs. 

\begin{proposition}[Relative Model Completeness]\label{prop:RMCVR}
    Let $(\vV,\vW,\beta) \subseteq (\vV',\vW',\beta')$ be two alternating bilinear systems satisfying Property $(*_f)$ of \Cref{cor:SepVectorsupspace}. Denote by $R$ (resp. $R'$) the relation induced on $\vV$ (resp. on $\vV'$) by $\beta$ (resp. $\beta'$). Assume that a complement $U$ of $\vW_\omega$ is disjoint from $\cup_n \vW_n'=\braket{\beta(\vV',\vV')}$.
    Then $(\vV,\vW,\beta) \preceq (\vV',\vW',\beta')$ if and only if $(\vV,+,0,R) \preceq (\vV',+,0,R')$.
\end{proposition}
\begin{proof}
    Since $\restriction{\beta'}{\vV^2}=\beta$, the embedding $(\vV,\vW,\beta) \subseteq (V',W',\beta')$ preserves $\pi_A$ for $A\in \mathbb{A}_n(\Fbb_p)$. It also preserves the functions $f_n$ for $n\in \mathbb{N}$ restricted to $\braket{\beta(\vV,\vV)}$. The condition on the complement $U$ ensures that on $\vW\setminus \braket{\beta(\vV,\vV)}$, we have $f_n(w)=f_n'(w)=\und$. Therefore, the embedding also preserves the language of relative quantifier elimination stated in \cref{prop: bilinearsystemEQR}. Again, we may conclude by \Cref{lem:RMCGeneral}.
\end{proof}

\begin{proposition}[Relative Stable Embeddedness]\label{prop:RSEVR}
    Let $(\vV,\vW,\beta) \preceq (\vV',\vW',\beta')$ be an elementary extension of bilinear systems satisfying Property $(*_f)$ of \Cref{cor:SepVectorsupspace}. With the same notation as in the previous proposition, we have: 
    \[
        (\vV,\vW,\beta) \preceq^{\st} (\vV',\vW',\beta') \text{ if and only if } (\vV,+,0,R) \preceq^{\st} (\vV',+,0,R),
    \]
    and
    \[
        (\vV,\vW,\beta) \preceq^{\ust} (\vV',\vW',\beta') \text{ if and only if } (\vV,+,0,R) \preceq^{\ust} (\vV',+,0,R).
    \]
\end{proposition}
\begin{proof}
    Assume $(\vV,+,0,R) \preceq^{\st} (\vV',+,0,R)$ and let $\phi(x_\vV,x_\vW,v',w')$ be a formula in the language of $\{(\vV,+,0), (\vW,+,0) ,\beta \}$ with parameters $v'\in \vV'$ and $w'\in \vW'$.
    We need to find a formula $\psi(x_\vV,x_\vW,v,w)$ with parameters  $v$ in $\vV$ and parameters $w$ in $\vW$ such that:
    \[
        \phi\left(\vV^{\vert x_\vV\vert},\vW^{\vert x_\vW\vert},v',w'\right)=\psi\left(\vV^{\vert x_\vV\vert},\vW^{\vert x_\vW\vert},v,w\right).
    \]
    By \Cref{prop: bilinearsystemEQR}, the formula $\phi(x_\vV,x_\vW,v',w')$ is equivalent to a Boolean combination of formulas of the form:
    \begin{itemize}
        \item $a \cdot x_\vW = w''$ where $a$ is a $x_\vW$ tuple in $\mathbb{F}_p$, $w''\in \vW'$, (the operation $\cdot$ denotes here the sum of products of the components $a$ and $x_\vW$).
        \item $\phi_\vV( f_{n}( a_{1} \cdot x_\vW + w_1'),\dots,f_{n}(a_{k}\cdot x_\vW + w_k'), x_\vV,v')$ where $\phi_\vV$ is a formula in the language of $(\vV,+,0,R)^{\eq}$, the $a_{i}$'s are $\vert x_\vW \vert$-tuples in $\mathbb{F}_p$, and the $w_i'$'s are parameters in $\vW'$ (the operation $\cdot$ is as in the previous point).
    \end{itemize}

    The first formula, which is only an equality between terms, is not a problem. If there is no $w\in \vW^{\vert x_\vW \vert}$ satisfying the formula $a \cdot x_\vW = w''$, we can replace this formula with $\bot$ and we are done. Otherwise, it means that $w''$ belongs to $\vW$ and we don't need to find a new parameter.

    We analyse the second formula. To simplify, we assume that there is only one term in $f_{n}$, and we assume that the formula $\phi(x_\vV,x_\vW,v',\vW')$ is given by:
    \[
        \phi_\vV( f_{n}( a \cdot x_\vW + w'), x_\vV,v').
    \]
    The general case follows by repeating the same process for each term.
 
    Let $\vW_n'$ be the set of elements of the form $\sum_{i<j<n}\beta(u_i,u_j')$ with $u_0,\dots,u_{n-1} \in\vW'$. If for all $w\in \vW^{\vert x_\vW \vert}$, $a \cdot w + w'$ is not in $\vW_n'$ (and thus $f_{n}( a \cdot w + w')=\und$ for all $w\in \vW^{\vert x_\vW \vert}$), then
    we can replace the term $ f_{n}( a \cdot x_\vW + w')$ by $\und$.
    
    Assume there is $w_0\in \vW^{\vert x_\vW \vert}$ such that $a \cdot w_0 + w'$ is in $\vW_n'$, then for all $w\in \vW^{\vert x_\vW \vert}$, we can write
    \[
        a \cdot w + w' = a \cdot (w-w_0) + a \cdot w_0 + w'.
    \]
    It follows that $a \cdot w + w'$ is in $\vW_n$ only if $a \cdot (w-w_0) $ is in $\vW_{2n}$. Then, we have that $a \cdot w + w' \in \vW_n$ and the value of $f_n(a \cdot w + w')$ can be expressed using $f_{2n}(a \cdot (w-w_0) )$ and $f_n(a\cdot w_0+w')$. Therefore, the formula
    \[
        \phi_\vV( f_{n}( a \cdot x_\vW + w'), x_\vV,v')
    \]
    is equivalent to:
    \[
    \begin{aligned}
        \left(f_{2n}(a \cdot (w-w_0) )\neq \und   \wedge   \phi_\vV'\left(f_{2n}(a \cdot (w-w_0) ) , f_n\left(a\cdot w_0+w'\right), x_\vV,v'\right)\right) \\
        \lor\left(f_{2n}(a \cdot (w-w_0))= \und  \wedge  \phi_\vV\left( \und , x_\vV,v'\right)\right)
    \end{aligned} 
    \]
    where $\phi_\vV'$ is some formula in the language $(\vV,+,0,R)^{\eq}$.

    Using the fact that $(\vV,+,0,R) \preceq^{\st} (\vV',+,0,R)$, and considering $ f_n(a\cdot w_0+w')$ as parameters in $(\vV')^{\eq}$, we can find a formula $\psi_\vV( x_{B_n}, x_\vV,v_0)$ with parameters $v_0\in \vV$ such that for all $w\in \vW^{\vert x_\vW \vert} $ and $v\in \vV^{\vert x_\vV \vert}$:
    \[
        (\vV,+,0,R)\vDash\phi_\vV( f_{n}( a \cdot w + w'), v ,v') \leftrightarrow  \psi_\vV( f_{2n}(a \cdot (w-w_0) ), v,v_0). 
    \]

    In other words, if we set $\psi(x_\vV,x_\vW,v_0,w_0) = \psi_\vV( f_{2n}(a \cdot (x_\vW-w_0) ) , x_\vV,v_0)$, we have:
    \[
        \phi\left(\vV^{\vert x_\vV \vert},\vW^{\vert x_\vW \vert},v',w'\right) = \psi\left(\vV^{\vert x_\vV \vert},\vW^{\vert x_\vW \vert},v_0,w_0\right),
    \]
    as wanted. 
    
    We have shown that $(\vV,\vW,\beta) \preceq^{\st} (\vV',\vW',\beta')$. The uniform stable embeddedness statement follows the same way, as discussed in the proof of \Cref{prop:RMCFG} -- the case distinctions can again be encoded with parameters in $\vV$.  
\end{proof}

\subsubsection{Third reduction: From \texorpdfstring{$(\vV,+,0,R)$}{(V,+,0,R)} to \texorpdfstring{$(\gG,R)$}{(C,R)}} Let $\Msf=(G,\cdot), \Msf'=(G',\cdot)$ be Mekler groups of respective graphs $(\gG,E)$ and $(\gG',E')$, denote additively $(\vV,+,0,R)$ (resp. $(\vV',+,0,R')$) the abelian quotient $G/\Zsf$ (resp. $G'/\Zsf'$), and where $R$ (resp. $R'$) denotes the relation induced by the commutativity relation on $G$ (resp. $G'$) on $\vV$ (resp. $V'$) as well as on $\gG$ (resp. $\gG'$).

\begin{proposition}[Relative Model Completeness]\label{prop:RMCC}
    In the notation above, assume that $(\vV,+,0,R) \subseteq (\vV',+,0,R')$ and that a transversal $X=X^\nu X^pX^\iota$ of $\vV$ that can be extended to a transversal $X'={X^\nu}' {X^p}'{X^\iota}'$ of $V'$ (with ${X^\nu}\subseteq {X^\nu}'$,  $X^p\subseteq {X^p}'$ and $X^\iota\subseteq {X^\iota}'$).
    Then
$(\vV,+,0,R) \preceq (\vV',+,0,R')$ if and only if $(\gG,E) \preceq (\gG',E')$ as graphs.
\end{proposition}
\begin{proof}
    It is easy to see that the conditions for being a transversal imply that for all $n,m\in \mathbb{N}^2$, we have:
    \[
        \restriction{S_{n,m}^{\vV'}}{\vV}=S_{n,m}^{\vV} \text{ and } \restriction{{S'_{n,m}}^{\vV'}}{\vV}={S_{n,m}'}^{\vV}.
    \] 
    We can conclude by \Cref{lem:RMCGeneral}.
\end{proof}

\begin{proposition}[Relative Stable Embeddedness]\label{prop:RSEC}
    In the notation above, assume that $G \preceq G'$ is an elementary extension of Mekler groups. Then:
        \[
            (\vV,+,0,R) \preceq^{\st} (\vV',+,0,R') \text{ if and only if } (\gG,E) \preceq^{\st} (\gG',E'),
        \]
    and
        \[
            (\vV,+,0,R) \preceq^{\ust} (\vV',+,0,R') \text{ if and only if } (\gG,E) \preceq^{\ust} (\gG',E').
        \]
\end{proposition}

A proof similar to that of \Cref{prop:RSEVR} will not work, and we have to be more careful, since we lose some information with the support function $S_{n,m}$. We are going to use the predicates $A_{n,m}$ as a pre-distance, and approximate elements from $\vV'$ by elements from $\vV$:
    
\begin{definition}
    A \emph{best approximation in $\vV$} of a singleton $b\in {\vV'}$ is an element $a_0\in \vV$ such that $f_{n,m}(b-a_0)$ holds and $(n,m)$ is minimal for the reverse lexicographic order on $\mathbb{N}^2$. 
    If no such element exists, we say that $b$ has \emph{no approximation} in $\vV$.
\end{definition}
     
\begin{lemma}\label{claim:bestapprox}
    In the notation of \Cref{prop:RSEC}, let $a_0$ be a best approximation of $b\in \vV'$, with $(k,l)$ minimal such that $A_{k,l}(b-a_0)$. Then:
    \begin{itemize}
        \item  The element $b-a_0$ is the sum of $k$ elements of types $1^\nu$ in $\vV'\setminus \vV$ and $l$ elements of type $p$ of $\vV'$ which are independent over $\braket{\tE^\nu(\vV')}$.  
        \item  For all $a\in \vV$, for all $n,m \in \Nbb$, we have:
        \[
            A_{n+k,m+l}(b-a) \Leftrightarrow A_{n,m}(a-a_0)
        \]
        \[  
            S_{n+k,m+l}(b-a)= S_{k,l}(b-a_0) \sqcup S_{n,m}(a-a_0),
        \]
        and
        \[ 
            S_{n+k,m+l}'(b-a) =S_{k,l}'(b-a_0) \sqcup S_{n,m}'(a-a_0) .
        \]
     \end{itemize}
\end{lemma}

\begin{proof}
    The first point easily follows from the minimality of $(n,m)$.
         
    The second follows from the fact that, since the support of $b-a_0$ is in $\vV'\setminus \vV$, the support of $a-a_0\in \vV$ and $b-a_0$ must be disjoint.
\end{proof}

\begin{proof}[Proof of \Cref{prop:RSEC}]
    Assume $(\gG,E) \preceq^{\st} (\gG',E')$ and let $\phi(x,b)$ be a group formula with parameters $b\in \vV'$.  By \Cref{prop:RQEGraph}, the formula $\phi(x,b)$ in the structure $(G/\Zsf,\cdot,1,R)$ is equivalent to a Boolean combination of formulas of the form:
    \begin{itemize}
        \item $t(x,b)=1$ where $t(x,b)$ is a group term,
        \item $A_{n,m}(t(x,b))$, where $n,m\in \mathbb{N}$, $t(x)$ is a group term,
        \item $\phi_\gG(S_{n,m}(t(x,b)),S_{n,m}'(t(x,b)))$, where $n,m\in \mathbb{N}$, $t(x,y)$ is a tuple of $\lL_\mathsf{grp}$-terms and $\phi_\gG(x_\gG,x_\gG')$ is a formula in the language $(\gG,E)$.
    \end{itemize}

    We handle the first form as before, so let us assume that $\phi(x,b)$ is a conjunction of formulas of the second and of the third form. Since the group is abelian, we can separate the variables $x$ from the parameters $b$ and assume that the terms are of the form $t(x)+b_1$ where $b_1\in b$. 
    
    So, let us assume that our formula is given by:
    \[
        \phi(x,b)=\phi_\gG(S_{n,m}(t(x)+b),S_{n,m}'(t(x)+b))\wedge A_{n,m}(t(x)+b),
    \]
    and that $t(x)$ is a single term and $b$ is a singleton-- of course, the general case with many terms follows the same way.
    
    If $b$ does not have a best approximation in $\vV$, then $A_{n,m}(t(a)+b)$ never holds, and we can replace the formula by $\bot$.
    
    Assume that there exists a best approximation $a_0$ of $b$ in $\vV$, and let $(k,l)$ be minimal (for the reverse lexicographic order of $\Nbb^2$) such that $A_{k,l}(b-a_0)$ holds. Then, by \Cref{claim:bestapprox}, for all $a\in \vV$ and $k,l\in \Nbb$, we have:
    \[
        A_{n,m}(t(a)+b) \Leftrightarrow A_{n-k,m-l}(t(a)+a_0) 
    \]
    and
    \[
    S_{n,m}(t(a)+b)= S_{k,l}(b-a_0) \sqcup S_{n-k,m-l}(t(a)+a_0),
    \]
    and similarly for $S'_{n,m}(t(a)+b)$. 
    
    Observe that if $k>n$ or $l>m$, then neither $S_{n,m}(t(a)+b)$ nor $S_{n,m}'(t(a)+b)$ can hold. This means that $\phi(x,b)$ can be rewritten as a formula of the form:
    \[
        \phi(x,b)= \phi_\gG(S_{n-k,m-l}(t(a)+a_0),S_{k,l}'(t(a)+a_0), S_{k,l}(b-a_0),S_{k,l}'(b-a_0)),
    \]
    and the terms $S_{k,l}(b-a_0)$, $S_{k,l}(b-a_0)$ can be seen as parameters in $(\gG',R)^\eq$. 
    
    Using our assumption that $(\gG,E) \preceq^{\st} (\gG',E')$, there is a formula  $\psi_\gG(x,y,\alpha)$ in $(\gG',E')^\eq$ with parameters $\alpha\in \gG$ such that:
    \[
        \phi_\gG\left(\mathcal{P}_{n-k}(\gG),\mathcal{P}_{m-l}(\gG), S_{k,l}(b-a_0),S_{k,l}'(b-a_0)\right)= \psi_\gG\left(\mathcal{P}_{n-k}(\gG),\mathcal{P}_{m-l}(\gG),\alpha\right), 
    \]
    where $\mathcal{P}_{n-k}(\gG)$ denotes the set of all subsets of size $\leq n-k$ of $\gG$.
    
    Let $a\in \gG^{\vert \alpha \vert}$ be such that $S_{0,1}(a)=\{\alpha\}$, and by $\psi(x,a)$ the formula: 
    \[
        \psi_\gG(S_{n-k,m-l}(t(x)+a_0),\{S_{n-k,m-l}'(t(x)+a_0): k\leq n, l\leq m\}, S_{1,0}(a)),
    \]
    (we identify $\alpha$ with $\{\alpha\}=S_{1,0}(a)$).
    
    We have 
    \[
        \phi(G(\gG),b)=\psi(G(\gG),a),
    \]
    as desired and thus $(\vV,+,0,R) \preceq^{\st} (\vV',+,0,R')$.  
    
    Notice that the formula $\psi$ depends on the existence of a minimal $(k,l)$ such that $A_{k,l}(b-a_0)$ holds, but if $N$ is the largest integer that occurs in the formula, one need only consider the case for $k<N$ and $l<N$.  Thus, the uniform statement follows in a similar way.
\end{proof}

\subsubsection{Reduction from \texorpdfstring{$(G,\cdot,1)$}{(G,•,1)} to \texorpdfstring{$(\gG,R)$}{(C,R)}} Now we combine all the previous results. 

\begin{proposition}[Relative Model Completeness] \label{prop:RMCMekler}
    Let $\gG$ and $\gG'$ be nice graphs and let $\Msf=(G,\cdot,1)$ and $\Msf=(G',\cdot,1)$ be their respective Mekler groups. Assume that:
    \begin{itemize}
        \item $\Msf \subseteq \Msf'$, and $\Zsf \subseteq \Zsf'$;
        \item there is a transversal $X=X^\nu X^pX^\iota$ of $G$ which can be extended to a transversal $X'={X^\nu}' {X^p}'{X^\iota}'$ of $G'$ (with ${X^\nu}\subseteq {X^\nu}'$,  $X^p\subseteq {X^p}'$ and $X^\iota\subseteq {X^\iota}'$);
        \item there is a subspace $H$ of $\Zsf$ such that $G = \braket{X} \times H$ and such that $\braket{X'}\cap H = \emptyset$, where $X$ and $X'$ are transversal as above. 
    \end{itemize} 
    Then $\Msf \preceq \Msf'$ if and only if $\gG \preceq \gG'$ as graphs.
\end{proposition}

 Some similarities may be observed with statements already known, such as \cite[Lemma 2.14]{CH18}, which we will recall in \Cref{sec:alternative}. The proposition was perhaps already known, but we take the occasion to give a proof that makes use of our stepwise reductions.

\begin{proof}
    The embedding of $G$ in $G'$ preserves the centre $\Zsf \subseteq \Zsf'$  by definition. The induced embedding  $\mathcal{F}(G) \subseteq \mathcal{F}(G')$ preserves $\beta$ automatically (since it comes from an embedding of groups). The subspace $H$ of $\Zsf$ does not intersect $\braket{\beta(V,V')}$. The set $X \bmod \Zsf$ is a transversal of $G/\Zsf$ that extends to the transversal $X' \bmod \Zsf$ in $G'/\Zsf'$. Therefore, this proposition follows immediately from the previous ones, as we have:
    \begin{align*}
        \Msf \preceq \Msf' 
        & \Leftrightarrow \mathcal{F}(G) \preceq \mathcal{F}(G')        &\text{(\Cref{prop:RMCFG})}\\
        & \Leftrightarrow (G/\Zsf,\cdot,R) \preceq (G'/\Zsf',\cdot,R')  &\text{(\Cref{prop:RMCVR})}\\
        & \Leftrightarrow \gG \preceq \gG'                              &\text{(\Cref{prop:RMCC})}
    \end{align*}
    and the result follows.
\end{proof}

The following example shows that it is not enough to assume that $G$ is a subgroup of $G'$, and that the requirement of preserving a transversal is necessary. 

\begin{example}
    Consider the nice graphs $\gG$ and $\gG'$ given respectively by: 
    \begin{center}
    \begin{tikzpicture}
        \draw (-2,0) -- (4,0);
        \foreach \i in {-1,...,3}{
            \fill (\i,0) circle (2pt);
            \draw (\i,0) node[below] {$a_{\i}$};
    }
    \end{tikzpicture}
    \end{center}
and 
    \begin{center}
    \begin{tikzpicture}[scale=0.8]
        \draw (-2,0) -- (4,0);
        \draw (6,0) -- (12,0);
        \draw (5,0) node {$\cdots$};
        \foreach \i in {-1,...,3}{
            \fill (\i,0) circle (2pt);
            \draw (\i,0) node[below] {$a_{\i}$};
            
            \fill (8+\i,0) circle (2pt);
            \draw (8+\i,0) node[below] {$b_{\i}$};
    }
    \end{tikzpicture}
    \end{center}
    
    Denote by $\Msf$ and $\Msf'$ their corresponding Merkler group. The elementary embedding $\gG \rightarrow \gG'$ can easily be extended to an embedding of $\Msf\hookrightarrow \Msf'$ that is elementary by \Cref{prop:RMCMekler}.
    
    Consider now an elementary extension $\Nsf=\braket{G,g}$ of $\Msf$, generated by $\Msf$ and a new element $g$ of type $1^\iota$ and independent over $\Msf$. Then there is a group embedding $f$ of $\Nsf$ in $\Msf'$ such that $f(\Zsf)\subset \Zsf'$ and such that $f(g)$ is of type $1^\nu$ and $[f(g)]_\sim = b_0$. In particular, for this embedding, we have $\Nsf\not\preceq \Msf'$, even if the embedding satisfies the first item of \Cref{prop:RMCMekler}.
\end{example}

\setcounter{thmx}{1}

\begin{thmx}[Relative Stable Embeddedness] \label{thm:RSEMekler}
    Let $\Msf \preceq \Msf'$ be an elementary extension of Mekler groups of nice graphs $\gG$ and $\gG'$, respectively. Then:
    \begin{center}
        $\Msf \preceq^{\st} \Msf'$ if and only if $\gG \preceq^{\st} \gG'$
    \end{center}
    and
    \begin{center}
        $\Msf \preceq^{\ust} \Msf'$ if and only if $\gG \preceq^{\ust} \gG'$.
    \end{center}
\end{thmx}
\begin{proof}
    We simply combine the previous statements. Assume $G \preceq G'$. Then, with the notation previously established, we have:
    \begin{align*}
        \Msf \preceq^{(\mathsf{u})\st} \Msf' 
        & \Leftrightarrow \mathcal{F}(G) \preceq^{(\mathsf{u})\st} \mathcal{F}(G') &       \text{(\Cref{prop:RSEFG})}  \\
        & \Leftrightarrow (G/\Zsf,\cdot,R) \preceq^{(\mathsf{u})\st} (G'/\Zsf',\cdot,R')  & \text{(\Cref{prop:RSEVR})}\\
        & \Leftrightarrow \gG \preceq^{(\mathsf{u})\st} \gG' &\text{(\Cref{prop:RSEC})}
    \end{align*}
    and the result follows.
\end{proof}

\begin{examples}
    Let $\gG$ be any of the following graphs:
        \begin{itemize}
            \item $\gG=(V,E)$ with $V=\Zbb$ infinite and $E=\{(n,n+1) \ \vert \ n\in \Zbb\}$:
                \begin{center}
    \begin{tikzpicture}
        \draw (-2,0) -- (4,0);
        \foreach \i in {-1,...,3}{
            \fill (\i,0) circle (2pt);
    }
    \end{tikzpicture}
    \end{center}
        \item $\gG=(V,E)$ is a regular tree of any degree.
        \item $\gG=(V,E)$, any nice graph bi-interpretable with: $(\Zbb,<,+,0)$, $(\Rbb,<)$, ${(\Rbb,<,+,0)}$, $(\Rbb,+,\cdot,0,1)$.
        \end{itemize}
    Then $\gG$, and therefore $\Msf(\gG)$, is stably embedded in every elementary extension.
\end{examples}

\subsection{Characterisation of indiscernibles}

We will characterise indiscernible sequences in a Mekler group $\Msf$ in terms of indiscernible sequences in the graph $\gG$. To simplify the presentation, we introduce some new terminology. We fix an indexing structure $\Ical$ in a language $\Lcal'$. By an ``indiscernible sequence'' in this subsection we always mean an $\Ical$-indiscernible sequence.

\begin{definition}
    Let $\mathcal{M}$ be a structure and $(\bar{a}_i)_{i\in \Ical}$ and $(\bar{b}_i)_{i\in \Ical}$ be sequences in $\mathcal{M}$. We say that the indiscernibility of $(\bar{a}_i)_{i\in \Ical}$ over $B\subseteq \sM$ is \emph{witnessed} by the sequence $(\bar{b}_i)_{i\in \Ical}$ if there exist parameters $\bar{b}$ and a definable tuple of functions $\bar{f}$ such that:
    \begin{enumerate}
        \item $(\bar{b}_i)_{i\in \Ical}$ is indiscernible over $B\cup \bar b$,
        \item for every $i\in \Ical$, $\bar{a}_i=\bar{f}(\bar{b}_i,\bar{b})$.
    \end{enumerate}
 As usual, when $B=\emptyset$, then we do not mention it.   
\end{definition}

Notice that if $(\bar{b}_i)_{i\in \Ical}$ and $(\bar{a}_i)_{i\in \Ical}$ are as above, then $(\bar{a}_i)_{i\in \Ical}$ is indiscernible over $B$.

We work in $\Msf^{\eq}$. We shall show that, if $\Ical$ is Ramsey, then the $\Ical$-indiscernibility of a sequence $(\bar{a}_i)_{i\in \Ical}$ in $\Msf^{\eq}$ must be witnessed by an $\Ical$-indiscernible sequence $(\bar{b}_i)_{i\in \Ical}$ in $\gG$. In general, the length of the witnessing tuples $\bar{b}_i$ will be (much) greater than that of $\bar a_i$. In any case, we will have $\vert \bar{b}_i \vert \leq N \vert \bar{a}_i \vert $ for some $N$ depending on $\qftp_\Ical(i)$ and of course on the sequence $(a_i)_i$.

As for relative quantifier elimination, we proceed in three steps. Before we get into that, we give the following useful lemma: 

\begin{lemma}\label{lem:IndiscWitnessIndependant}
   Let $\Ical$ be a Ramsey indexing structure. Let $D$ be a $\emptyset$-definable subset of a group $G$ of  exponent $p$. The indiscernibility of any sequence $(\bar{g}_i)_{i\in \Ical}$ in $G$ is witnessed by an indiscernible sequence $(\bar{a}_i\bar{b}_i)_{i\in \Ical}$ where $\cup_{i\in \Ical}\bar{b}_i \subseteq D$ and $\cup_{i\in \Ical}\bar{a}_i$ is independent (in the sense of \Cref{def:independence}) over $\braket{D}$.
\end{lemma}
\begin{proof}
    Observe that since $\Ical$ is Ramsey, by the Generalised Standard Lemma -- \Cref{thm:genstandlem} -- it suffices to prove the proposition when $\Ical$ is the \emph{Fra\"iss\'e limit}\footnote{See \cite[Theorem~7.1.2]{Hod93} for a statement of Fra\"iss\'e's Theorem. We will denote the \emph{Fra\"iss\'e limit} of an amalgamation class $\Ccal$ by $\Flim(\Ccal)$.} of its \emph{age}, that is, when $\Ical=\Flim(\Age(\Ical))$. To see this, suppose that the result holds for all $\Flim(\Age(\Ical))$-indexed sequences. Then, given an $\Ical$-indexed sequence $(\bar a_i)_{i\in\Ical}$, by the Generalised Standard Lemma, we can find a $\Flim(\Age(\Ical))$-indexed indiscernible $(\bar a_i')_{i\in\Flim(\Age(\Ical))}$ which is locally based on it and a $\Flim(\Age(\Ical))$-indiscernible of the desired form $(\bar b_i')_{i\in\Flim(\Age(\Ical))}$ witnessing the indiscernibility of $(\bar a_i')_{i\in\Flim(\Age(\Ical))}$. Applying the Generalised Standard Lemma again, this time to the sequence $(\bar a_i'\frown\bar b_i')_{i\in\Flim(\Age(\Ical))}$ we can find an $\Ical$-indiscernible $(\bar a_i''\frown\bar b_i'')_{i\in\Ical}$ locally based on $(\bar a_i'\frown\bar b_i')_{i\in\Flim(\Age(\Ical))}$. It is clear that $(\bar a_i'')_{i\in\Ical}$ is locally based on $(\bar a_i)_{i\in\Ical}$ and that they are both $\Ical$-indiscernible, so we can find an automorphism $\sigma$ sending $\bar a_i''$ to $\bar a_i$. Then $(\sigma(\bar b_i''))_{i\in\Ical}$ remains of the desired form and witnesses the indiscernibility of $(\bar a_i)_{i\in\Ical}$. Thus, in the remainder of the proof, assume $\Ical=\Flim(\Age(\Ical))$. Observe also that since $\Ical$ is Ramsey, by \cite[Theorem~B]{MP23}, it expands a linear order, which we will assume is already in $\Lcal'$.
    
    Let us first write $(\bar{g}_i)_{i\in \Ical}$ as a sequence $(\bar{a}_i\bar{b}_i)_{i\in \Ical}$ such that $\bar{a}_i\notin \braket{D}$ and $\bar{b}_i \in \braket{D}$, for all $i\in \Ical$.

    Consider a group term $t(\bar{x}_0,\dots, \bar{x}_{k-1},\bar{x})$ where one of the variables, say $x^0\in\bar{x}$ really occurs.\footnote{We mean that the term is not equal to a term $t'(\bar{x}_0,\dots, \bar{x}_{n-1})$} Assume that $t(\bar{a}_{i_0},\dots,\bar{a}_{i_{n-1}},\bar{a}_{i})\in \braket{D}$ for some $i_0<\cdots<i_{n-1}<i$. For every $j$ with $\qftp(j)=\qftp(i)$, we want to replace $a_{j}^0$ with a term in $\braket{D}$ such that the new sequence witnesses the indiscernibility of the original one. 
    
    For every such $j$, there are $j_0,\dots,j_{n-1}$ such that  the term $b_j''=t(\bar{a}_{j_0},\dots,\bar{a}_{j_{n-1}},\bar{a}_{j})$ is in $\braket{D}$. For $j$ such that  $\qftp(j)\neq \qftp(i)$, we set $b_j''=\emptyset$, and we can consider the sequence $(\bar{a}_i\bar{b}_ib_i'')_{i\in\Ical}$. At this point, we need the following claim:

    \begin{claim}\label{lem:same-types}
        There is an elementary extension $\Ical'\elext\Ical$ and elements $j_1<\cdots<j_n\in\Ical'$ such that for all $k\in\Ical$ for which there exist $k_1<\cdots<k_n<k$ with:
        \[
        \qftp^\Ical(i_1,\dots,i_n,i)=\qftp^\Ical(k_1,\dots,k_n,k),
        \]
        we have that: 
        \[
            \qftp^{\Ical'}(i_1,\dots,i_n,i) = \qftp^{\Ical'}(j_1,\dots,j_n,k).
        \]
    \end{claim}

    \begin{claimproof}
        This follows from the fact that $\Age(\Ical)$ has the amalgamation property and compactness. Formally, fix an enumeration $(c_i:i\in\Nbb)$ of $\Ical$, and for every finite $\Delta(y_1,\dots,y_n,y)\subseteq_\mathsf{fin} \qftp(i_1,\dots,i_n,i)$, let $\phi_m^\Delta(x_1,\dots,x_n,c_1,\dots,c_m)$ be the following formula:
        \[
            \bigwedge_{i\leq m} \left(\left(\exists z_1,\dots,z_n\Delta(z_1,\dots,z_n,c_i) \right)\limplies \Delta(x_1,\dots,x_n,c_i)\right).
        \]
        Intuitively, $\phi_m^\Delta$ says that $x_1,\dots,x_n$ is a \emph{common $\Delta$-witness}, that is, it verifies the claim, for formulas in $\Delta$ and for the first $m$ elements in the enumeration of $\Ical$, so it suffices to show that: 
        \[
            \Sigma(x_1,\dots,x_n) \coloneq \{\phi_m^\Delta(x_1,\dots,x_n,c_1,\dots,c_m):m\in\Nbb,\Delta\subseteq_\mathsf{fin}\qftp(i_1,\dots,i_n,i)\}
        \]
        is satisfiable. Clearly, it is enough to show that a single $\phi_m^\Delta$ is satisfiable, and this easily follows from the amalgamation and homogeneity of $\Ical$. Indeed, we can build a witness for $\phi_m^\Delta$ inductively: once the witness for the first $l<m$ elements has been constructed, we use amalgamation to find a structure where the $(l+1)$-st element has a common witness with the first $l$ elements (if one exists for this element). Then, since $\Ical$ is homogeneous, we can embed this structure appropriately in $\Ical$.
    \end{claimproof}

    By \Cref{lem:same-types}, we find an elementary extension $\Ical' \succcurlyeq \Ical''$ and $k_0<\cdots<k_{n-1} \in \Ical'$ such that for all $j\in \Ical$ with $\qftp(j)=\qftp(i)$, we have
    \[
        \qftp^{\Ical'}(k_0,\dots,k_{n-1},j)=\qftp^\Ical(i_0,\dots,i_{n-1},i).
    \]

    By the Generalised Standard Lemma, there is an indiscernible sequence $(\tilde{a}_i\tilde{b}_i\tilde{b}''_i)_{i\in \Ical'}$ indexed by $\Ical'$ that is locally based on $(\bar{a}_i\bar{b}_ib''_i)_{i\in\Ical}$. Notice that the sequence $(\tilde{a}_i\tilde{b}_i\tilde{b}''_i)_{i\in \Ical'}$ restricted to $\Ical\subseteq\Ical'$ is indiscernible over $\{\tilde{a}_{k_0},\dots,\tilde{a}_{k_{n-1}}\}$.

    We find an automorphism $\sigma$ sending $\bar{a}_i\bar{b}_i$ to $ \tilde{a}_i\tilde{b}_i$, for all $i\in \Ical$. Denote by $\bar{b}\coloneq\bar{b}_0,\dots,\bar{b}_{n-1}$ the tuple $\sigma^{-1}(\tilde{a}_{k_0},\dots,\tilde{a}_{k_{n-1}})$. 

    For $j\in \Ical$ with $\qftp(i)=\qftp(j)$, denote $c_{j}=t(\bar{b}_{0},\dots,\bar{b}_{k-1},\bar{a}_{j})\in D$ and by $\bar{a}_j'$ the tuple $\bar{a}_j$ from which we removed $a_j^0$. For $j \in \Ical$ with $\qftp(i)=\qftp(j)$, we set $\bar{a}_j'=\bar{a}_j$ and $c_{j}=\emptyset$.

    Since we can express $a_i^0$ in terms of $c_{i}$ and $\bar{a}_i\setminus \{a_i^0\}$, we have that the indiscernibility of $(\bar{a}_i\bar{b}_i)_{i\in\Ical}$ is witnessed by that of $(\bar{a}_i' \bar{b}_i\frown c_i)_{i>k}$ over $\bar{b}$. 
    We may repeat the process until, for all $\qftp(j)=\qftp(i)$, $\bar{a}_j$ is independent over $\braket{D}$, and we reset the notation. Of course, this process must essentially be carried out for all quantifier-free types in $\Ical$, but the transformations for different quantifier-free types do not interact with each other, and hence we can perform them all simultaneously.
    
    It remains to find $\bar{b}_i'$ in $D$ such that $(\bar{a}_i,\bar{b}_i')_{i\in\Ical}$ is indiscernible and witnesses the indiscernibility of $(\bar{a}_i,\bar{b}_i)_{i\in\Ical}$. To simplify the notation, assume $\bar{b}_i=b_i$ is a singleton and $b_i = b_{0,i}\cdot \cdots \cdot b_{n-1,i}$ where $b_{k,i}\in D$.  
    
    By the Generalised Standard Lemma, we can now find an indiscernible sequence $(\tilde{a}_i,\tilde{b}_{0,i},\dots,\tilde{b}_{n-1,i})_{i\in\Ical}$ that is locally based on $(\bar{a}_i,b_{0,i},\dots,b_{n-1,i})_{i\in\Ical}$. For every $i\in\Ical$, set $\tilde{b}_i=\tilde{b}_{0,i}\cdot \cdots \cdot \tilde{b}_{n-1,i}$. 
    
    Finally, let $\sigma$ be an automorphism which for every $i$ sends $(\tilde{a}_i,\tilde{b}_i)$ to $(\bar{a}_i,b_i)$. Then $\sigma(\tilde{b}_{0,i}'),\dots,\sigma(\tilde{b}_{n-1,i}')$ are as desired.
\end{proof}

\subsubsection{Reduction from \texorpdfstring{$G$}{G} to \texorpdfstring{$\mathcal{F}(G)$}{F(G)}}
Let $G$ be a $2$-nilpotent group of prime exponent $p$. We characterise the indiscernibility of sequences in $G$ in terms of indiscernibility in $\mathcal{F}(G)$.

\begin{proposition}\label{prop:IndiscerniblereductionF(G)}
\
    \begin{itemize}
        \item ($\Ical$ Ramsey) The indiscernibility of a sequence $(g_i)_{i\in\Ical}$ from $G$ is witnessed by an indiscernible sequence  $(\bar{a}_i\bar{b}_i)_{i\in\Ical}$ such that $\bar{b}_i \in \Zsf$ and $\bar{a}_i$ are independent over $\Zsf$.
        \item (Arbitrary $\Ical$) Let $(\bar{a}_i\bar{b}_i)_{i\in\Ical}$ be a sequence in $G$ such that $\bar{b}_i \in \Zsf$ and $\bar{a}_i$ are independent over $\Zsf$.     Then $(\bar{a}_i\bar{b}_i)_{i\in\Ical}$ is indiscernible if and only if $\left(\pi(\bar{a}_i),\rho\left(\bar{b}_i\right)\right)_{i\in\Ical}$ is indiscernible in $\mathcal{F}(G)$.
    \end{itemize}
\end{proposition}

\begin{proof}
    The first point follows from \Cref{lem:IndiscWitnessIndependant}, by setting $D=\Zsf$, the centre of $G$. 
    
    It remains to show the second point. Notice that, since $\bar{b}_i$ are central elements, for every $\lL_\mathsf{grp}$-term $t(\bar{x},\bar{y})$, there is an $\lL_\mathsf{grp}$-term $t''(\bar{x})$ such that:
    \[
        \pi(t(\bar{a}_i,b_i)) = \pi(t''(\bar{a}_i))=t''(\pi(\bar{a}_i)).\footnote{The term $t''$ should be understood multiplicatively in the group $G$ and additively in the sort $\vW$ of $\Fcal(G)$.}
    \]
    Also, since $\bar{a}_i$ are independent over $\Zsf$, for every $\lL_\mathsf{grp}$-term $t(\bar{x},\bar{y})$ there is an $\lL_\mathsf{grp}$-term $t'(\bar{y})$ such that:  
    \[
        \rho(t(\bar{a}_i,b_i)) = \rho(t'(\bar{b}_i))=t'(\rho(\bar{b}_i)).
    \]
    Then, by quantifier elimination, we know that any formula $\phi(\bar{x},\bar{y})$ is equivalent to a formula of the form:
    \[
        \phi_{\mathcal{F}(G)}(\pi(t(\bar{x},\bar{y})),\rho(t(\bar{x},\bar{y}))),
    \]
    where $\phi_{\mathcal{F}(G)}(\bar{y},\bar{z})$ is a formula in the language $\{(V,+,0),(W,+,0),\beta\}$. 

    By the above, we find group terms $t'(\bar{y})$ and $t''(\bar{x})$ such that $\phi(\bar{a}_i,\bar{b}_i)$ is equivalent to:
    \[
        \phi_{\mathcal{F}(G)}(t'(\pi(\bar{a}_i)),t''(\rho(\bar{b}_i))).
    \]
    The statement follows easily.

 \end{proof}
\subsubsection{Reduction from \texorpdfstring{$\{(\vV,+,0),(\vW,+,0),\beta\}$}{\{(V,+,0),(W,+,0), β\} } to \texorpdfstring{$(\vV,+,0,R)$}{(V,+,0,R)}}
Consider an alternating bilinear system $\{(\vV,+,0),(\vW,+,0),\beta\}$ satisfying Property $(*_f)$ of \Cref{cor:SepVectorsupspace}. Let $(\bar{v}_i\bar{u}_i \bar{w}_i)_{i\in\Ical}$ be a sequence with $\bar{v}_i \in \vV$, $\bar{u}_i\in\bigcup_n \vW_n$, and $\bar{w}_i \notin \bigcup_n \vW_n$.

\begin{lemma} \label{lem:IndiscWitness}
    Let $\Ical$ be a Ramsey indexing structure. Let $(\bar{v}_i\bar{u}_i \bar{w}_i)_{i\in\Ical}$ be an indiscernible sequence with $\bar{v}_i \in \vV$, $\bar{u}_i\in\bigcup_n \vW_n$, and $\bar{w}_i \notin \bigcup_n \vW_n$. Then, for every realised quantifier-free type $q$ in $\Ical$ there exist a positive integer $n_q$, coefficients $\bar{d}^{k,l}_q$, for $k<l<n_q$ and for every $i\in\Ical$ with $\qftp(i) = q$, a sequence of tuples $(\bar{v}^{0,i}_q, \dots, \bar{v}^{n_q-1,i}_q)$ such that:
    \begin{itemize}
        \item $(\bar{v}_i,\bar{v}^{0,i}_{\qftp(i)},\dots,\bar{v}^{n-1,i}_{\qftp(i)}, \bar{w}_i)_{i\in\Ical}$ is indiscernible; and
        \item For all $i\in\Ical$ we have:
        \[
            \bar{u}_i= \sum_{k<l} \bar{d}^{k,l}_{q} \beta (\bar{v}^{k,i}_{\qftp(i)},\bar{v}^{l,i}_{\qftp(i)}),\footnote{The sum and $\beta$ are componentwise.}
        \]
        where $q= \qftp(i)$.
    \end{itemize}
\end{lemma}

\begin{proof}
    To simplify the notation, assume that $\bar{u}_i=u_i$ is a singleton. For every $i\in\Ical$, we can rewrite $u_i$ as a sum $\beta(v',v'')$ for $v',v''$ certain tuples $v_{0,i},\dots,v_{n-1,i}$. By indiscernibility, the we can choose the $n$ and the coefficients depending only on $\qftp(i)$, so, for all $i\in\Ical$ with $\qftp(i)=q$ we have:
    \[
        u_i= \sum_{k<l} \bar{d}^{k,l}_{q} \beta (\bar{v}^{k,i},\bar{v}^{l,i}),
    \]
    for some $d^{k,l}_q\in \mathbb{F}_p$. For every realised quantifier-free type $q$ in $\Ical$ and all $i\in\Ical$ realising $q$, we substitute $u_i$ by the tuples $v^{0,1},\dots,v^{n_q-1,i}$ we found above. The new sequence need not be indiscernible, but by the Generalised Standard Lemma, we can now find an indiscernible sequence 
    \[          (\tilde{v}_i,\tilde{v}^{0,i},\dots,\tilde{v}^{n_{\qftp(i)}-1,i}, \tilde{w}_i)_{i\in\Ical}
    \]
    locally based on our new sequence. Set 
    \[\Tilde{u}= \sum_{k<l} \bar{d}^{k,l}_{q} \beta (\tilde{v}^{k,i},\tilde{v}^{l,i}).\]
    Let $\sigma$ be an automorphism which, for every $i$, sends $(\tilde{v}_i,\tilde{u}_i,\tilde{w}_i)$ to $(\bar{v}_i,\bar{u}_i,\bar{w}_i)$. Then, for each $i\in\Ical$, the tuple $\sigma(\tilde{v}^{0,i}),\dots,\sigma(\tilde{v}^{n_{\qftp(i)}-1,i})$ is as desired.
\end{proof}
\begin{proposition}\label{prop:IndiscerniblereductionVR} \
    \begin{itemize}
        \item ($\Ical$ Ramsey) Any indiscernible sequence $(\bar{u}_i)_{i\in\Ical}$ is witnessed by a sequence $(\bar{v}_i\bar{w}_i)_{i\in\Ical}$ such that $\bar{v}_i \in V$ and $\cup\bar{w}_i$ are in $\vW$ and linearly independent over $\bigcup_n \vW_n$.
        \item (Arbitrary $\Ical$) Let $(\bar{v}_i\bar{w}_i)_{i\in\Ical}$ be a sequence in $(\vV,\vW,\beta)$ with $\bar{v}_i \in \vV$ and $\bar{w}_i\in \vW$. Assume that $\bigcup_i \bar{w}_i$ is a linearly independent subset of $\vW$ over $\bigcup_n \vW_n$.
        
        Then $(\bar{v}_i\bar{w}_i)_{i\in\Ical}$ is indiscernible if and only if $(\bar{v}_i)_{i\in\Ical}$ is indiscernible in $(\vV,+,0,R)$.
    \end{itemize}
\end{proposition}
\begin{proof} 
    By applying \Cref{lem:IndiscWitnessIndependant} with $D=\vW_n$ for every $n$, we may assume that $\cup\bar{w}_i \setminus\bigcup_n \vW_n $ are linearly independent over $\bigcup_n \vW_n$. If $w\in \cup\bar{w}_i$ is in $\vW_n$, we use \Cref{lem:IndiscWitness} (for every quantifier-free type) to replace $w$ with $n$ (potentially) new elements $v\in \vV$. This gives us the first point.

    It remains to show the second. Let $\phi(v,w)$ be a formula. By relative quantifier elimination, the formula is equivalent to a Boolean combination of formulas for the form: 
    \begin{itemize}
        \item $\sum a_i w_i=0$ where $a_i's$ are in $\mathbb{F}_p$,
        \item $\phi_\vV( f_{n}(\sum a_{0,i} w_i),\dots,f_{n}(\sum a_{k-1,i} w_i), \bar{v})$ where $\phi_\vV$ is a formula in the language of $(\vV,+,0,R)^{\eq}$, and $a_{j,i}$'s are in $\mathbb{F}_p$.
    \end{itemize}
    
    Notice that, by definition of the linear independence of $\bar{w}_i$'s, a formula of the form $\sum a_i w_i=0$ never holds unless all coefficients $a_i$ are $0$. Also, since non-trivial sums $\sum a_{i} w_i$ are never in $\vW_n$, we must have $f_{n}(\sum a_{i} w_i)=\und$ (that is, the function $f_n$ is not defined on $\sum a_{i} w_i$). It follows that $\phi((\bar{v}_i)_i,(w_i)_i)$ is equivalent to a formula of the form $\phi_\vV'( (\bar{v}_i)_i)$, and the statement follows.
\end{proof}

\subsubsection{Reduction from \texorpdfstring{$(\vV,+,0,R)$}{(V,+,0,R)} to \texorpdfstring{$(\gG,E)$}{(C,E)}}
Let $\Msf=(G,\cdot,1)$ be a Mekler group, denote additively $(\vV,+,0)$ the abelian quotient $G/\Zsf$, and denote by $R$ (resp. $E$) the relation induced on $\vV$ (resp. on $\gG$) by the commutativity relation on $\Msf$.

\begin{proposition}\label{prop:IndiscerniblereductionCR}\
    \begin{itemize}
        \item ($\Ical$ is Ramsey) Any indiscernible sequence $(\bar{u}_i)_{i\in\Ical}$ in $\vV$ is witnessed by a sequence $(\bar{v}_i\bar{v}_i'\bar{v}_i'')_{i\in\Ical}$         which is part of a transversal, i.e. such that:
            \begin{itemize}
                \item  $\cup_i\bar{v}_i$ is composed of elements of type $1^\nu$,
                \item $\cup_i\bar{v}_i'$ is composed of elements of type $p$ and linearly independent over elements of type $1^\nu$,
               \item $\cup_i\bar{v}_i''$ is composed of elements of type $1^\iota$ and linearly independent over elements of type $1^\nu$ and $p$.    
        \end{itemize}
 
        \item (Arbitrary $\Ical$) Consider a sequence $(\bar{v}_i,\bar{v}_i',\bar{v}_i'')_{i\in\Ical}$ such that:
        \begin{itemize}
            \item  $\cup_i\bar{v}_i$ is composed of elements of type $1^\nu$,
            \item $\cup_i\bar{v}_i'$ is composed of elements of type $p$ and linearly independent over element of type $1^\nu$,
            \item $\cup_i\bar{v}_i''$ is composed of elements of type $1^\iota$ and linearly independent over elements of type $1^\nu$ and $p$.
        \end{itemize}
    \end{itemize}
    Then $(\bar{v}_i,\bar{v}_i',\bar{v}_i'')_{i\in\Ical}$ is indiscernible if and only if, the sequence 
    \[
        ([\bar{v}_i]_\sim,h(\bar{v}_i'))_i,\footnote{This is the sequence of tuples consisting of the class modulo $\sim$ of each component of $\bar{v}_i$ and of the handles of each component of $\bar{v}_i'$.}
    \]
    is indiscernible in $(\gG,E)$.
\end{proposition}
\begin{proof}
    We deduce the first point directly from \Cref{lem:IndiscWitnessIndependant}, and the second follows immediately from relative quantifier elimination. Notice that we don't need to consider terms of the form $S_{n,m}(v_{i_1}\cdots v_{i_n} v_{j_1}'\cdots v_{j_m}')= \{[v_{i_1}]_\sim, \dots, [v_{i_n}]_\sim\}$ for $n+m>1$ since they can be recovered from $S_{1,0}(v_{i_j})$ for $1\leq j\leq n$. 
\end{proof}

\subsubsection{Reduction from \texorpdfstring{$(G,\cdot,1)$}{(G,•,1)} to \texorpdfstring{$(\gG,E)$}{(C,E)}}

Let $\mathcal{M}=(G,\cdot,1)$ be a monster model of a Mekler group, and let $\gG=(V,E)$ be the corresponding nice graph.

\begin{lemma} \label{lem:IndiscWitness2}
    Let $(\alpha_i \beta_i)_{i\in\Ical}$ be an indiscernible sequence where for all $i$'s, $\beta_i \in \Zsf$ is a finite product of commutator. Then, for every realised quantifier-free type $q$ in $\Ical$ there exist a positive integer $n_q$, coefficients $\bar{d}^{k,l}_q$, for $k<l<n_q$ and for every $i\in\Ical$ with $\qftp(i) = q$, a sequence of tuples $(\beta^{0,i}_q, \dots, \beta^{n-1,i}_q)$ such that:
    \begin{itemize}
        \item The sequence $(\alpha_i,\beta^{0,i}_{\qftp(i)},\dots,\beta^{n-1,i}_{\qftp(i)})_{i\in\Ical}$ is indiscernible; and
        \item For all $i\in\Ical$, we have:
        \[
            \beta_i= \prod_{k<l} \left[\beta^{k,i}_{\qftp(i)},\beta^{l,i}_{\qftp(i)}\right]^{d^{k,l}_q},
            \footnote{If $\beta_i$ is a tuple, all operations are done componentwise.}
        \]
        where $q=\qftp(i)$
\end{itemize}
\end{lemma}
\begin{proof}
    The proof is similar to that of \Cref{lem:IndiscWitness}. 
\end{proof}

\begin{proposition}\label{prop:CharacterisationIndiscernible} \
    \begin{itemize} 
        \item ($\Ical$ Ramsey) Any indiscernible sequence $(\bar{g}_i)_{i\in\Ical}$ in $G$ is witnessed by an indiscernible sequence $(\alpha_i,\beta_i,\gamma_i,\delta_i)_{i\in\Ical}$ in $G$ which is part of a full transversal, i.e. such that: 
        \begin{itemize}
            \item $\cup_i \alpha_i \cup \beta_i \cup\gamma_i$ is part of a transversal,
            \item $\cup_i \alpha_i$ is a subset of $\tE^\nu$ (and is independent over $\Zsf$),
            \item $\cup_i\beta_i$ is a subset of $\tE^p$ (and is independent over $\braket{\Zsf,\tE^\nu}$),
            \item $\cup_i\gamma_i$ is a subset of $\tE^\iota$ (and is independent over $\braket{\Zsf,\tE^\nu,\tE^p}$),
            \item $\cup_i \delta_i$ is a subset of $\Zsf$ and is independent over $\braket{\tE^\nu,\tE^p,\tE^\iota}$.
        \end{itemize}
        \item (Arbitrary $\Ical$) Such a sequence $(\alpha_i,\beta_i,\gamma_i, \delta_i)_{i\in\Ical}$ which is part of a full transversal -- with the same notation as above -- is indiscernible if and only if the sequence
        \[ 
            ([\alpha_i]_\sim, h(\beta_i))_{i\in\Ical}
        \]
        is indiscernible in $\gG$.
    \end{itemize}
\end{proposition}
\begin{proof}
    One can easily see why a sequence of indiscernibles must be witnessed by an indiscernible sequence which is part of a transversal using \Cref{lem:IndiscWitnessIndependant,lem:IndiscWitness2} . The second point follows immediately from \Cref{prop:IndiscerniblereductionF(G),prop:IndiscerniblereductionVR,prop:IndiscerniblereductionCR}. 
\end{proof}

\subsection{Transfers of dividing lines}\label{subsec:Collapse}

\subsubsection{Collapsing and $\ncodingcla{K}$} 
Having established our characterisation of indiscernibles in Mekler groups, we can now deduce (almost immediately) our main theorem:

\setcounter{thmx}{0}
\begin{thmx}[Collapsing Transfer]\label{thm:TransferCollapsing}
    Let $\Ical$ be Ramsey and $\sJ$ a (not necessarily Ramsey) reduct of $\sI$. Let $\Msf$ be a Mekler group of the nice graph $\gG$. Then $\Msf$ collapses $\Ical$-indiscernibles (resp. collapses $\Ical$-indiscernibles to $\sJ$-indiscernibles) if and only if $\gG$ collapses $\Ical$-indiscernibles (resp. collapses $\Ical$-indiscernibles to $\sJ$-indiscernibles). 
\end{thmx}

\begin{proof}
    Assume that $\gG$ collapses $\Ical$ to $\sJ$ indiscernibles. Consider an $\Ical$-indiscernible sequence $(g_i)_{i\in\Ical}$ in $\Msf$. By \Cref{prop:CharacterisationIndiscernible}, the $\Ical$-indiscernibility of $(g_i)_{i\in\Ical}$ is witnessed by an indiscernible sequence $(\alpha_i,\beta_i,\gamma_i,\delta_i)_{i\in\Ical}$, with $\alpha_i$ of type $1^\nu$ and $\beta$ of type $p$ and such that $\cup_i \alpha_i \cup \beta_i \cup \gamma_i$ is part of a full transversal. 
    In particular, $([\alpha_i]_\sim,h(\beta_i))_{i\in\Ical}$ is  $\Ical$-indiscernible, and therefore $\sJ$ indiscernible since $\gG$ collapses $\Ical$-indiscernibles to $\sJ$-indiscernibles. 
    
    By the equivalence, $(\alpha_i,\beta_i,\gamma_i)_{i\in\dom(\Ical)}$ is $\sJ$-indiscernible. As it witnesses the indiscernibility of $(g_i)_{i\in\Ical}$, the latter is also $\sJ$-indiscernible. Notice that no assumption has been made on $\sJ$; therefore, the proof can be adapted to (non-specific) collapsing.  
\end{proof}

\begin{corollary}[$\ncodingcla{K}$ transfer]
    Let $\cla{K}$ be a Ramsey class in a countable language with an $\aleph_0$-categorical Fraïssé limit. A Mekler group $\Msf$ is $\ncodingcla{K}$ if and only if its associated graph $\gG$ is $\ncodingcla{K}$. 
\end{corollary}
\begin{proof}
    By \Cref{thm:Ramsey class collapse}, we have that for any complete theory $T$, $T$ is $\ncodingcla{K}$ if and only if $T$ has no uncollapsed $\Flim(\cla{K})$-indiscernible. Therefore, this is only a reformulation of the previous theorem.
\end{proof}

We can now, for instance, recover the main theorem of \cite[Theorem 4.7]{CH18}. For that, we recall a characterisation of NIP$_k$, for an integer $k>1$. Let $\Ccal_{\Hsf_{k}}$ be the class of all finite ordered $k$-partite hypergraphs. This is, of course, a Ramsey class by a classical result of Nešetřil and Rödl \cite{NR1977}.

\begin{fact}[{\cite[Proposition 5.4]{CPT19}, \cite[Theorem 3.14]{GH19}}]\label{fact:nip-n-nc-k}
    The following are equivalent for a first-order theory $T$ and $k\in\Nbb$:
        \begin{enumerate}
            \item $T$ is NIP$_k$.
            \item $T$ is $\ncodingcla{\Ccal_{\Hsf_{k}}}$.
        \end{enumerate}
\end{fact}

\begin{corollary}[{\cite[Corollary~4.8]{CH18}}]
    For every $k\geq 1$, a Mekler group $\Msf$ is NIP$_k$ if and only if its associated graph $\gG$ is NIP$_k$.
\end{corollary}

We can also immediately deduce the following, from \Cref{fact:NFOP_k-collapse}:

\begin{corollary}
    For every $k\geq 1$ a Mekler group $\Msf$ is NFOP$_k$ if and only if its associated graph $\gG$ is NFOP$_k$. 
    
    In particular, for all $k\geq 2$ there is a strictly NFOP$_k$ pure group.
\end{corollary}

\begin{proof}
    The first part of the corollary is immediate from \Cref{fact:NFOP_k-collapse} and \Cref{thm:TransferCollapsing}. 

    By \cref{fact:everything-is-nice}, any structure in a finite relational language is bi-interpretable with a nice graph. From \cite[Proposition~3.23]{AACT23} we know that any $k$-ary structure (i.e. any structure which admits quantifier elimination in a relational language where all symbols have arity at most $k$) is NFOP$_k$, so in particular, the random $k$-hypergraph, $k\geq 2$, is NFOP$_k$. We know, from \cite[Proposition~2.8]{AACT23} that:
    \begin{center}
        NFOP$_k$ $\implies$  NIP$_k$,
    \end{center}
    and since the random $k$-hypergraph has IP$_{k-1}$ it has FOP$_{k-1}$, is thus strictly NFOP$_k$. Let $\gG$ be any nice graph bi-interpretable with the random $k$-hypergraph. Then $\Msf(\gG)$ is a strictly NFOP$_k$ pure group, by the first part of the corollary.
\end{proof}

\subsection{Resplendence} Notice that the quantifier elimination results and the transfer principles obtained from them are all resplendent, in the sense of \cite[Appendix~A]{Rid17}. This means that \Cref{thm:TransferCollapsing,thm:RSEMekler}, and  \Cref{prop:CharacterisationIndiscernible,prop:RMCMekler} remain valid if we replace, \emph{mutatis mutandis}, the Mekler groups $\Msf$ with a $\gG$-enrichment 
\[
    \Msf^*= \{\Msf=(G,\cdot,1), \gG=(V,R,\dots), \pi \colon E^{\nu}\rightarrow \gG \},
\]
where $\dots$ denotes additional structure on $\gG$.

To finish off this section, we give some negative results. These are derived from basic observations and do not require the relative quantifier elimination.

\subsubsection{Distality}

In this paper, a precise definition of distality is not necessary. The reader can find one in \cite[Chapter~9]{Simon_2015}. For our purposes, it suffices to note that a distal theory does not admit any non-constant totally indiscernible sequences.

\begin{proposition}
    Let $\Msf$ be a Mekler group. Then $\Th(\Msf)$ is non-distal.
\end{proposition}
\begin{proof}
    We show that $\Msf$ admits a non-constant totally indiscernible sequence. Consider an infinite sequence $(\gamma_i)_{i\in \mathbb{N}}$ of elements of type $1^\iota$ which are independent over $\braket{\Zsf,\tE^\nu,\tE^p}$. By \Cref{prop:CharacterisationIndiscernible}, this sequence is totally indiscernible. It only remains to show that such a sequence exists, which we briefly justify for the sake of completeness. Since $\gG$ is an infinite graph, it has, by Ramsey, an infinite clique (i.e. a subset where every two elements are connected one to the other), or an infinite anticlique (i.e. an infinite subset where no two elements are connected one to the other). Since $\gG$ is nice, it cannot have an infinite clique and, therefore, it must have an infinite anticlique $A$. By compactness and considering larger and larger products of elements in $A$, we can find elements of type $1^{\iota}$ which are independent over $\braket{\Zsf,\tE^\nu,\tE^p}$.
\end{proof}

A more relevant question is therefore the following:

\begin{question}
    Assume that a nice graph $\gG$ admits a distal expansion. Does $\Msf(\gG)$ admit a distal expansion?
\end{question}

 It seems that if $\gG$ admits a linear order $<$ such that the structure $(\gG,<,E)$ is distal, then one can define a natural valuation on $G/\Zsf$ and $\Zsf$, and this additional structure would eliminate all traces of stability that Mekler's construction can bring `on top' of the structure $\gG$. 

\subsubsection{Burden}\label{subsubsec:burden}

As we have seen, Mekler's construction preserves many dividing lines. However, it seems that it does not (always) preserve notions of dimension. We will treat here the \emph{burden}, i.e. the notion of dimension attached to NTP$_2$ theories ( which coincide with the dp-rank if the theory is NIP).

\begin{proposition}
    Let $\gG$ be an infinite nice graph. Then, $\Msf(\gG)$ has burden at least $\aleph_{0-}$.
\end{proposition}

We refer the reader to \cite{Adl} for the precise definitions needed for the remainder of this section. 

This proposition, in particular, provides a negative answer to \cite[Problem~5.8]{CH18}. One should expect similar results for any other reasonable notion of dimension. The proof below describes how one can obtain an inp-pattern of arbitrary finite size.  

\begin{proof}
    Consider the pattern consisting of an array $(a_{i,j})_{i,j\in\omega}$ of distinct elements of $\gG$ and the formula $\phi(x,y_\gG):$
 \[
    A_{m,0}(\pi(x)) \wedge y_\gamma \in S_{m,0}(x)
\]
where $x$ is a group sorted variable, $y_\gamma$ is a $\gG$-sorted variable, $\pi$ is the projection modulo $\Zsf$, $A_{m,0}$ is the predicate for elements $x$ which are sum of $m$ element of type $1^\nu$, and $S_{m,0}(x)$ is the support of such element $x$.

Then notice that $\{\phi(x,a_{i,j})|i\in\omega\}$ is $(m+1)$-inconsistent (but $m$-consistent!) and $\{\phi(x,a_{i,f(i)})|i\in\omega\}$ is consistent for all $f \colon m \rightarrow \omega$.
 
We thus have an inp-pattern of depth $m$ for every integer $m$, therefore $\Msf(\gG)$ has burden at least $\aleph_{0-}$.
\end{proof}

In \cite[Remark 5.7]{CH18}, Chernikov and Hempel argue that a Mekler group $\Msf(\gG)$ is \emph{strong} if and only if $\gG$ is strong. More generally, one can certainly ask what is the exact burden of $\Msf(\gG)$. We conjecture the following formula:

\begin{conjecture}
    We have $\bdn(\Msf(\gG)) = \min \{\aleph_{0-},\bdn(\gG)\}$.
\end{conjecture}

\section{Second proof of Theorem A}\label{sec:alternative}

We now present a second proof of \Cref{thm:TransferCollapsing}. We follow a proof strategy similar to that of Hempel and Chernikov in \cite{CH18}. It avoids the formal description of definable sets from \Cref{sec:RQE}, at the cost of introducing a (reasonable) additional assumption on the collapsing property of the nice graph. Let us recall the following key lemma from \cite{CH18}:

\begin{lemma}[{\cite[Lemma~2.14]{CH18}}]\label{lem:basic}
    Let $\gG$ be a nice graph and $\Mcal$ a saturated model of $\Th(\Msf(\gG))$. Let $X$ be a transversal and $K_X\leq\Zsf(\Mcal)$ be such that $\Mcal = \langle X\rangle\times\langle K_X\rangle$. Let $W$ and $Y$ be two small subsets of $X$ and $\bar h_1,\bar h_2$ two tuples in $K_X$. Suppose that the following conditions hold:
    \begin{itemize}
        \item There is a bijection $f:W\to Y$ which respects the $1^\nu$-, $p$-, $1^\iota$-parts ,the handles, and $\tp_\Gamma(W^\nu) = \tp_\Gamma(f(Y)^\nu)$.
        \item $\tp_{K_X}(\bar h_1) = \tp_{K_X}(\bar h_2)$.
    \end{itemize}
    Then, there is an automorphism $\sigma\in\Aut(\Mcal)$ extending $f$ such that $\sigma(\bar h_1) = \bar h_2$.
\end{lemma}

In \cite{CH18}, amongst other results, the authors essentially show the following (\cite[Theorem 4.7]{CH18}): If, for some nice graph $\gG$, every model of $\Th(\gG)$ collapses $\Hsf_k$-indiscernibles (where $\Hsf_k$ is the random $k$-partite hypergraph) to $\Psf_k$-indiscernibles (where $\Psf_k$ is the random $k$-partite set), then so must every model of $\Th(\Msf(\gG))$. Below, we adapt their argument for the general case of collapsing $\Ical$-indiscernibles (in place of $\Hsf_k$-indiscernibles) to $\Jcal$-indiscernibles (in place of $\Psf_k$-indiscernibles). 

\begin{definition}[Specific collapsing]
    Let $\Ical$ be an $\aleph_0$-categorical Fra\"iss\'e limit of a Ramsey class and $\Jcal$ a proper reduct of $\Ical$. We say that \emph{$\Ical$ has specific collapsing to $\Jcal$} if for all theories $T$, and all $\Mcal\vDash T$, every collapsing $\Ical$-indiscernible sequence in $\Mcal$, is (a possibly collapsing) $\Jcal$-indiscernible sequence.
\end{definition}

Many dividing lines are characterised by specific collapsing to a reduct. Classical examples include stability (where we have specific collapsing of order-indiscernibles to indiscernible sets), NIP$_k$ (here we have specific collapsing of random ordered $k$-hypergraph indiscernibles to linear orders), and recently NFOP$_k$ (this is more complicated to describe but follows from results in \cite{AACT23}, in particular, from \cite[Lemma~4.14]{AACT23}) for all $k\in\Nbb$.

We start with an easy lemma:
\begin{lemma}\label{lem:collapse-in-stable}
    Let $T$ be a stable theory, $\Ical$ an $\aleph_0$-categorial Fra\"iss\'e limit of a Ramsey class in a finite relational language, and $\Jcal$ a proper reduct of $\Ical$. If $\Ical$ has specific collapsing to $\Jcal$ then every model of $T$ collapses $\Ical$-indiscernibles to $\Jcal$-indiscernibles.
\end{lemma}

\begin{proof}
    Since $\Ical$ is Ramsey, it expands a linear order (see \cite[Corollary~2.26]{Bodirsky_2015}), and since $T$ is stable, we must have that $\Ical$-indiscernibles collapse to the reduct of $\Ical$ where we forget the order in models of $T$. By specific collapsing, we then have that $\Ical$-indiscernibles collapse to $\Jcal$-indiscernibles in models of $T$.
\end{proof}

\textbf{Theorem A'} (Second Statement)\textbf{.} 
    Let $\Ical$ be an $\aleph_0$-categorial Fra\"iss\'e limit of a Ramsey class in a finite relational language, and $\Jcal$ a proper reduct of $\Ical$, also in a finite relational language. For every nice graph $\gG$, if:
    \begin{itemize}
        \item $\Ical$ has specific collapsing to $\Jcal$; and
        \item $\Th(\gG)$ collapses $\Ical$-indiscernibles to $\Jcal$-indiscernibles, 
    \end{itemize}
    then:
    \begin{itemize}
        \item $\Th(\Msf(\gG))$ collapses $\Ical$-indiscernibles to $\Jcal$-indiscernibles.
    \end{itemize} 

\begin{proof}
    Assume toward a contradiction that $\Th(\gG)$ collapses $\Ical$-indiscernibles to $\Jcal$-indiscernibles but $\Th(\Msf(\gG))$ does not. Since both $\Ical$ and $\Jcal$ are finitely homogeneous, they are $k$-ary, for some $k\in\Nbb$. By appropriately padding the symbols in $\Lcal_{\Ical}$ and $\Lcal_{\Jcal}$, the respective languages of $\Ical$ and $\Jcal$, we may assume that they all have arity exactly $k$. Moreover, since $\Ical$ has quantifier elimination we can also assume that $\Lcal_\Jcal\subseteq\Lcal_\Ical$.
	
    By assumption, we can find a model $\Mcal\vDash\Th(\Msf(\gG))$ and an $\Ical$-indiscernible sequence $B = (b_i :i\in\dom(\Ical))$ in $\Mcal$ which is not $\Jcal$-indiscernible. Since $\Ical$ has specific collapsing to $\Jcal$, this must mean that $B$ is non-collapsing. In particular, for all all $(i_1,\dots,i_k),(j_1,\dots,j_k)\in\dom(\Ical)^k$ we have that
    \[
        \qftp_{\Ical}(i_1,\dots,i_k) = \qftp_{\Ical}(j_1,\dots,j_k)\iff \tp(b_{i_1},\dots,b_{i_k}) = \tp(b_{j_1},\dots,b_{j_k}).
    \]	
    \begin{claim}\label{claim:formulas}
        There is some $R\in\Lcal_\Ical\setminus\Lcal_\Jcal$ for which we can find an $\lL_\mathsf{grp}$-formula $\phi_{R}$ such that for and all $i_1,\dots,i_k\in\Ical$ we have that:
        \[
            \Mcal\vDash\phi_{R}(b_{i_1},\dots,b_{i_k}) \text{ if and only if } \Ical\vDash R(i_1,\dots,i_{k}).
        \]
    \end{claim}
    \begin{claimproof}
        First, since $\Ical$ is $\aleph_0$-categorical there are finitely many quantifier-free $k$-types in $\Th(\Ical)$, say $p_1,\dots,p_n$. Every relation symbol $R\in\Lcal_\Ical$ belongs to a finite Boolean combination of these types, and we claim that for some $R\in\Lcal_\Ical\setminus\Lcal_\Jcal$ there is a Boolean combination of (isolating formulas of) quantifier-free types $\psi$ such that:
        \[
            \Ical\vDash\psi(i_1,\dots,i_k)\text{ if and only if } \Ical\vDash R(i_1,\dots,i_k).
	\]  
        If not, then we would be able to find a reduct $\Jcal'$ of $\Ical$ which would make $B$ a $\Jcal'$-indiscernible sequence, contradicting our initial assumption.
		
        Now, since $B$ is an uncollapsed $\Ical$-indiscernible, for each quantifier-free $k$-type $p_i$ of $\Ical$ there is a unique $k$-type $q_i$ of $\Mcal$ in $B$, corresponding to the $k$-tuples indexed by $p_i$. For each such type we can find an $\lL_{\mathsf{grp}}$-formula $\phi_i$ separating it from the rest. 

        Recall that since $\Ical$ is $\aleph_0$-categorical for each quantifier-free $n$-type $p\in S_n^\Ical(\emptyset)$, there is a formula $\mathsf{iso}(p)\in p$ isolating $p$. Moreover, for all $n\in\Nbb$, by Ryll-Nardzewski, $|S_n^\Ical(\emptyset)|<\aleph_0$.
        
        Given the above, for a relation symbol $R\in\Lcal_\Ical$, there is a unique Boolean combination of isolating formulas
        \[
            \psi \coloneq \bigwedge\bigvee\mathsf{iso}(p_i)^{\epsilon_i},
        \]
        equivalent to $R$. 
        
        For $\psi$ as above, take $\phi_R$ to be the formula $\bigwedge\bigvee\phi_i^{\epsilon_i}$, i.e. the corresponding Boolean combination of the formulas separating the types $q_i$ in $\Mcal$. \end{claimproof}
	
    Fix $\kappa$ to be $\aleph_0^+$ and $\Ibb\elext\Ical$ be an elementary extension of $\Ical$ with $|\Ibb| = \kappa$. Without loss of generality, we may assume that $\Mcal$ is saturated. Since $\Ical$ is Ramsey, $\Ical$-indiscernibles have the modelling property, so by compactness/saturation we can actually find, in $\Mcal$, an $\Ibb$-indexed indiscernible sequence $A=(a_i:i\in\dom(\Ibb))$ which is based on $B$. Let $\Jbb\elext\Jcal$ be the elementary extension of $\Jcal$ given by taking the $\Lcal_\Jcal$-reduct of $\Ibb$.
    \begin{claim}\label{claim:not-indiscernible}
        The sequence $(a_i:i\in\dom(\Ibb))$ is not $\Jbb$-indiscernible.
    \end{claim}

    \begin{claimproof}
        Since $(a_i:i\in\dom(\Ical))$ is not $\Jcal$-indiscernible, we can find $i_1,\dots,i_k,j_1,\dots,j_k\in\dom(\Ical)$ such that $\qftp_{\Jcal}(i_1,\dots,i_k) = \qftp_{\Jcal}(j_1,\dots,j_k)$ and $\tp(b_{i_1},\dots,b_{i_k}) \neq \tp(b_{j_1},\dots,b_{j_k})$. Let $\Delta_1\subseteq\tp(b_{i_1},\dots,b_{i_k})$, and $\Delta_2\subseteq\tp(b_{j_1},\dots,b_{j_k})$ be two finite sets of $\lL_{\mathsf{grp}}$-formulas such that $\Delta_1\cup\Delta_2$ is inconsistent. 
		 
        Since $(a_i:i\in\dom(\Ibb))$ is based on $(b_i:i\in\dom(\Ical))$ and $\Ibb\elext\Ical$, given $l_1,\dots,l_k\in\Ibb$ such that $\qftp_{\Ibb}(l_1,\dots,l_k) = \qftp_{\Ical}(i_1,\dots,i_k)$, we can find a tuple $i_1'\dots,i_k'\in\Ical$ such that $\qftp_{\Ibb}(l_1,\dots,l_k) = \qftp_{\Ical}(i_1',\dots,i_k')$ and thus:
        \[
            \tp^{\Delta_1}(a_{l_1},\dots,a_{l_k}) = \tp^{\Delta_1}(b_{i_1'},\dots,b_{i_k'}) =  \tp^{\Delta_1}(b_{i_1},\dots,b_{i_k}).
        \]
        Similarly, we find $m_1,\dots,m_k$ such that $\qftp_{\Ibb}(m_1,\dots,m_k)=\qftp_\Ical(j_1,\dots,j_k)$ and $\tp^{\Delta_2}(a_{m_1},\dots,a_{m_k}) = \tp^{\Delta_2}(b_{j_1},\dots,b_{j_k})$. But, by construction 
        \[
            \qftp_{\Jbb}(l_1,\dots,l_k)=\qftp_{\Jbb}(m_1,\dots,m_k),
        \]
        and thus $(a_i:i\in\dom(\Ibb))$ is not $\Jbb$-indiscernible.
    \end{claimproof}
	
    Let us write $\Mcal$ in the form $\langle X\rangle\times\langle H\rangle$, where $X$ is a transversal and $H\subseteq \Mcal$ a set that is linearly independent over $[\Mcal,\Mcal]$.
	
    After fixing an enumeration of $\dom(\Ibb)$, and rearranging $A$ to be in the form $(a_i:i<\kappa)$, we express the elements in $A$ as $\lL_\mathsf{grp}$-terms built up from elements of $X$ and $H$, as follows: For each $\lambda<\kappa$ let $t_{\lambda}$ be an $\lL_{\mathsf{grp}}$-term, and $\bar x_\lambda$, $\bar h_\lambda$ be finite tuples from $X$ and $H$, respectively, such that $a_\lambda = t_\lambda(\bar x_\lambda,\bar h_\lambda)$. 
	
    By passing to a cofinal subsequence of $A$ of cardinality $\kappa$ (in the fixed enumeration), we can find a term $t\in\Lcal_\mathsf{grps}$ such that for all $\lambda<\kappa$ we have $t_\lambda = t$. Since each $\bar x_\lambda$ is a tuple from $X$, we can assume that it is of the form $\bar x_\lambda^\nu\frown\bar x_\lambda^p\frown\bar x_\lambda^\iota$, where we simply list all elements of $\bar x_\lambda$ of the corresponding types. We may also append the handles of the elements in the tuple $\bar x_\lambda^p$ to the beginning of $\bar x_\lambda^\nu$ (so that the handle of the $j$-th element of $\bar x_\lambda^p$ is the $j$-th element of $\bar x_\lambda^\nu$, and we allow repetition of elements). 
	
    Thus, at this point, after rearranging, we have an $\Ibb$-indiscernible sequence $(t(\bar x_i,\bar h_i):i\in\dom(\Ibb))$. Of course, since $\Ical\elsub\Ibb$, we may actually work with the subsequence $(t(\bar x_i,\bar h_i):i\in\dom(\Ical))$. By construction, this sequence is $\Ical$-indiscernible, and arguing as in \Cref{claim:not-indiscernible}, $(\bar x_i\frown\bar h_i:i\in\dom(\Ical))$ is not $\Jcal$-indiscernible.
	 
    By \Cref{claim:formulas} and the fact that $A$ is based on $B$, we can find some $R\in\Sigma = \Lcal_\Ical\setminus\Lcal_\Jcal$ and some $\lL_\mathsf{grp}$-formula $\phi_R$ such that:
    \[
        \Mcal\vDash\phi_R(a_{i_1},\dots,a_{i_k})\text{ if and only if }\Ical\vDash R(i_1,\dots,i_k).
    \]
    Now, let $\Gamma(\Mcal)$ be a saturated model of $\Th(\gG)$, containing all the elements $(\bar x_i^\nu:i\in \dom(\Ical))$. Since $\Gamma(\Mcal)$ is interpretable in $\Mcal$, this sequence remains $\Ical$-indiscernible, and since $\Th(\gG)$ collapses $\Ical$-indiscernibles to $\Jcal$-indiscernibles this is actually a $\Jcal$-indiscernible in $\Gamma$. Similarly, let $\Hcal$ be a saturated model of $\Th(\langle H\rangle)$, it is easy to see that $(\bar h_i:i\in\dom(\Ical))$ is $\Ical$-indiscernible in $\Hcal$, and since $\Th(\langle H \rangle)$ is stable, from \Cref{lem:collapse-in-stable}, we have that $(\bar h_i:i\in\dom(\Ical))$ is $\Jcal$-indiscernible in $\Hcal$.
	
    Since the sequence $(\bar x_i\frown\bar h_i:i\in\dom(\Ical))$ is not $\Jcal$-indiscernible, there are $i_1,\dots,i_k,j_1,\dots,j_k\in\dom(J)$ such that $\qftp_{\Jcal}(i_1,\dots,i_k) = \qftp_{\Jcal}(j_1,\dots,j_k)$ and $\tp((\bar x_{i_1}\frown\bar h_{i_1}),\dots,(\bar x_{i_k}\frown\bar h_{i_k}))\neq \tp((\bar x_{j_1}\frown\bar h_{j_1}),\dots,(\bar x_{j_1}\frown\bar h_{j_1}))$. Of course, by $\Ical$-indiscernibility of the sequence we must have that $\qftp_{\Ical}(i_1,\dots,i_k) \neq \qftp_{\Ical}(j_1,\dots,j_k)$, so there exists some relation symbol $R\in\Lcal_{\Ical}$ such that:
    \[
        \Ical\vDash R(i_1,\dots,i_{k})\land\lnot R(j_1,\dots,j_k).
    \]
    In particular, we can find an $\lL_{\mathsf{grp}}$-formula $\phi_R$ such that:
    \[
        \Mcal\vDash\phi_R((\bar x_{i_1}\frown\bar h_{i_1}),\dots,(\bar x_{i_k}\frown\bar h_{i_k}))\land\lnot\phi_R((\bar x_{j_1}\frown\bar h_{j_1}),\dots,(\bar x_{j_1}\frown\bar h_{j_1}))
    \]	
    But, observe that at this point we are in the following situation:
    \begin{itemize}
        \item Since $\qftp_\Jcal(i_1,\dots,i_k)=\qftp_\Jcal(j_1,\dots,j_k)$ and the sequences $(\bar x_i^\nu:i\in \dom(\Ical))$ and $(\bar h_i:i\in\dom(\Ical))$ are $\Jcal$-indiscernible we have that:
        \[
            \tp_\Gamma(\bar x_{i_1}^\nu,\dots,\bar x_{i_k}^\nu) = \tp_\Gamma(\bar x_{j_1}^\nu,\dots,\bar x_{j_k}^\nu)\text{ and }\tp_{\langle H\rangle}(\bar h_{i_1},\dots,\bar h_{i_k}) = \tp_{\langle H\rangle}(\bar h_{j_1},\dots,\bar h_{j_k})
        \]
        \item The map sending $\bar x_{i_l}\frown\bar h_{i_l} \mapsto \bar x_{j_l}\frown\bar h_{j_l}$ respects the $1^\nu$-, $p$-, $1^\iota$-parts, and the handles.
    \end{itemize}
    Thus, by applying \Cref{lem:basic}, we can extend this map to an automorphism $\sigma\in\Aut(\Mcal)$ sending $(\bar x_{i_l}\frown\bar h_{i_l})$ to $(\bar x_{j_l}\frown\bar h_{j_l})$ for all $l\in\{1,\dots,k\}$, which is a contradiction.
\end{proof}

\section{Open questions}\label{sec:OpenQuestions}

We suggest some possible transfers for Mekler groups that we try to briefly motivate.

\underline{Beautiful pairs}. Extending \Cref{thm:RSEMekler}, one could try to characterise the existence of beautiful pairs of Mekler groups, in the sense of \cite{CHY23}. In this work, Cubides Kovacsics, Hils and Ye develop a general theory of beautiful pairs, extending the ideas of Poizat. They show, in particular, that the existence of a $\lambda$-saturated $\lambda$-beautiful pair for $\lambda>\vert T\vert ^+$ implies strict pro-definability of the space of definable types. Hence saturated beautiful pairs of Mekler groups would give examples of $2$-nilpotent groups with a strict pro-definable space of definable types.

\underline{Imaginaries}. To complement our result of relative quantifier elimination, one could try to characterise imaginary elements of Mekler groups (up to finite imaginaries). Transfer principles for imaginaries are known in the context of valued fields (e.g. \cite{RV23}). A similar result for Mekler groups could certainly be interesting; we would then have a complete understanding of the first-order structure of Mekler's construction (definable sets, and quotients by definable equivalence relations).

\underline{Domination}. In a given structure, if the tensor product respects the domination relation between invariant types, the set of invariant types modulo the domination equivalence is a monoid (the \emph{domination monoid}), when equipped with the multiplicative law induced by the tensor product. In \cite{HM21}, Mennuni and Hils prove various transfer principles $A\rightarrow B$ for domination: they compute the domination monoid of $A$ in terms of the domination monoid of $B$. A similar computation $\Msf \rightarrow C$ for Mekler groups would give many examples of $2$-nilpotent groups with a well-defined domination monoid.

\textbf{Acknowledgements.} The authors thank Paolo Marimon for reminding them of the results of Abd-Aldaim, Conant, and Terry. We also thank Raf Cluckers and Dugald Macpherson for their constant and precious support. We thank the anonymous referees for their careful reading of the paper and detailed comments, which greatly improved its quality.

\bibliographystyle{alpha}
\bibliography{bibliography}
\end{document}

%% file: preamble.tex
\usepackage[T1]{fontenc} 
\usepackage[utf8]{inputenc}
\usepackage{marvosym}
\usepackage{amsmath}
\usepackage{scalerel}
\usepackage{amsthm}
\usepackage{amsfonts}
\usepackage{amssymb}
\usepackage{comment}
\usepackage{graphicx}
\usepackage{slashed} 
\usepackage{braket}
\usepackage{hyperref}
\usepackage[nameinlink]{cleveref}
\usepackage{tikz-cd}
\usepackage{adjustbox}
\usepackage{parskip}
\usepackage{euscript}
\usepackage{thmtools}
\usepackage{chngcntr}
\usepackage{float}
\usepackage{thm-restate}
\usepackage{mathtools}

\usepackage[inline]{enumitem}

\hypersetup{
  colorlinks   = true, 
  urlcolor     = {blue!50!black}, 
  linkcolor    = {blue!50!black}, 
  citecolor   = {red!50!black} 
}

\def\restriction#1#2{\mathchoice
              {\setbox1\hbox{${\displaystyle #1}_{\scriptstyle #2}$}
              \restrictionaux{#1}{#2}}
              {\setbox1\hbox{${\textstyle #1}_{\scriptstyle #2}$}
              \restrictionaux{#1}{#2}}
              {\setbox1\hbox{${\scriptstyle #1}_{\scriptscriptstyle #2}$}
              \restrictionaux{#1}{#2}}
              {\setbox1\hbox{${\scriptscriptstyle #1}_{\scriptscriptstyle #2}$}
              \restrictionaux{#1}{#2}}}
\def\restrictionaux#1#2{{#1\,\smash{\vrule height .8\ht1 depth .85\dp1}}_{\,#2}}

\newcommand{\limplies}{\rightarrow}
\newcommand{\elext}{\succcurlyeq}
\newcommand{\elsub}{\preccurlyeq}

\DeclareMathOperator{\Aut}{Aut}

\DeclareMathOperator{\bdn}{bdn}

\DeclareMathOperator{\tp}{\mathsf{tp}}
\DeclareMathOperator{\qftp}{\mathsf{qftp}}

\DeclareMathOperator{\Th}{Th}
\DeclareMathOperator{\Sig}{Sig}

\DeclareMathOperator{\M}{M}

\newcommand{\Acal}{\ensuremath{\mathcal{A}}}

\newcommand{\Ccal}{\ensuremath{\mathcal{C}}}

\newcommand{\Fcal}{\ensuremath{\mathcal{F}}}

\newcommand{\Hcal}{\ensuremath{\mathcal{H}}}
\newcommand{\Ical}{\ensuremath{\mathcal{I}}} 
\newcommand{\Jcal}{\ensuremath{\mathcal{J}}} 
 
\newcommand{\Lcal}{\ensuremath{\mathcal{L}}}
\newcommand{\Mcal}{\ensuremath{\mathcal{M}}}
\newcommand{\Ncal}{\ensuremath{\mathcal{N}}}

\newcommand{\Vcal}{\ensuremath{\mathcal{V}}}
\newcommand{\Wcal}{\ensuremath{\mathcal{W}}}

\newcommand{\Csf}{\ensuremath{\mathsf{C}}}

\newcommand{\Hsf}{\ensuremath{\mathsf{H}}}

\newcommand{\Msf}{\ensuremath{\mathsf{M}}}
\newcommand{\Nsf}{\ensuremath{\mathsf{N}}}

\newcommand{\Psf}{\ensuremath{\mathsf{P}}}

\newcommand{\Zsf}{\ensuremath{\mathsf{Z}}}

\newcommand{\Abb}{\ensuremath{\mathbb{A}}}

\newcommand{\Fbb}{\ensuremath{\mathbb{F}}}

\newcommand{\Ibb}{\ensuremath{\mathbb{I}}} 
\newcommand{\Jbb}{\ensuremath{\mathbb{J}}}

\newcommand{\Mbb}{\ensuremath{\mathbb{M}}}
\newcommand{\Nbb}{\ensuremath{\mathbb{N}}}

\newcommand{\Rbb}{\ensuremath{\mathbb{R}}}

\newcommand{\Zbb}{\ensuremath{\mathbb{Z}}}

\newcommand{\eq}{\ensuremath{\mathsf{eq}}}
\newcommand{\st}{\ensuremath{\mathsf{st}}}
\newcommand{\ust}{\ensuremath{\mathsf{ust}}}

\let\str\CMcal
\newcommand\cla\EuScript

\let\lang\CMcal
\newcommand{\lL}{\lang{L}}

\let\graph\mathrm
\newcommand{\gG}{\graph{C}}

\let\type\mathfrak
\newcommand{\tE}{\type{E}}

\let\sp\CMcal
\newcommand{\vV}{\sp{V}}
\newcommand{\vW}{\sp{W}}

\newcommand{\sA}{\str{A}}
\newcommand{\sB}{\str{B}}

\newcommand{\sM}{\str{M}}
\newcommand{\sN}{\str{N}}

\newcommand{\sS}{\str{S}}

\newcommand{\sI}{\str{I}}
\newcommand{\sJ}{\str{J}}

\newcommand{\cC}{\cla{C}}

\newcommand{\cK}{\cla{K}}

\newcommand{\NOT}{\mathrm{N}}
\newcommand{\codingcla}[1]{\cC_{#1}}
\newcommand{\ncodingcla}[1]{\NOT\codingcla{#1}}

\newcommand{\dom}{\mathsf{dom}}
\newcommand{\und}{\mathsf{u}}

\usepackage{ulem}
\usepackage{todonotes}
\setuptodonotes{noinline}

\colorlet{PierreColor}{cyan}
\colorlet{ArisColor}{red}

\DeclareMathOperator{\Age}{\mathsf{Age}}
\DeclareMathOperator{\Flim}{\mathsf{Flim}}

\theoremstyle{plain}
\newtheorem{theorem}{Theorem}[section]
\newtheorem{thmx}{Theorem}
\newtheorem{lemma}[theorem]{Lemma}
\newtheorem{proposition}[theorem]{Proposition}
\newtheorem{fact}[theorem]{Fact}
\newtheorem{corollary}[theorem]{Corollary}

\theoremstyle{definition}
\newtheorem{definition}[theorem]{Definition}

\newtheorem{conjecture}[theorem]{Conjecture}

\newtheorem{notation}[theorem]{Notation}
\newtheorem*{notation*}{Notation}
\newtheorem{question}[theorem]{Question}
\newtheorem*{question*}{Question}
\newtheorem{example}[theorem]{Example}
\newtheorem*{examples}{Examples}

\theoremstyle{remark}

\newtheorem{remark}[theorem]{Remark}
\newtheorem*{remark*}{Remark}

\newtheorem{claim}{Claim}
\newtheorem{poc}{Proof of Claim}

\newenvironment{claimproof}{%
	\begin{poc}
	}{%
	\hfill$\blacktriangleleft$
	\end{poc}
	}

\counterwithin*{claim}{theorem}
\counterwithin*{poc}{theorem}

\definecolor{black}{rgb}{0,0,0}

\colorlet{savedColor}{.}

\crefname{subsection}{subsection}{subsections}
\crefname{subsubsection}{subsubsection}{subsubsections}
\crefname{figure}{figure}{figures}
\Crefname{figure}{Figure}{Figures}
\Crefname{thmx}{Theorem}{Theorems}%

\let\oldtocsection=\tocsection

\let\oldtocsubsection=\tocsubsection

\renewcommand{\tocsection}[2]{\hspace{0em}\oldtocsection{#1}{#2}}
\renewcommand{\tocsubsection}[2]{\hspace{1em}\oldtocsubsection{#1}{#2}}